\documentclass{article} 
\usepackage[doublespacing]{setspace}

\usepackage{amsmath,amssymb,amsfonts,color,enumerate,Theorems,geometry}
\DeclareMathOperator{\vspan}{span}
\usepackage{graphicx}
\usepackage{cancel}

\usepackage[normalem]{ulem}
\useunder{\uline}{\ul}{}

\usepackage{algorithm}
\usepackage{algorithmic}

\begin{document}

\title{Dynamic Mode Decomposition for Continuous Time Systems with the Liouville Operator}

\author{Joel A. Rosenfeld         \and
        Rushikesh Kamalapurkar \and
        L. Forest Gruss \and
        Taylor T. Johnson
}

\maketitle

\begin{abstract}
Dynamic Mode Decomposition (DMD) has become synonymous with the Koopman operator, where continuous time dynamics are examined through a discrete time proxy determined by a fixed timestep using Koopman (i.e. composition) operators. Using the newly introduced ``occupation kernels,'' the present manuscript develops an approach to DMD that treats continuous time dynamics directly through the Liouville operator, which can include Koopman generators. This manuscript outlines the technical and theoretical differences between Koopman based DMD for discrete time systems and Liouville based DMD for continuous time systems, which includes an examination of these operators over several reproducing kernel Hilbert spaces (RKHSs). While Liouville operators are modally unbounded, this manuscript introduces the concept of a scaled Liouville operator, which for many dynamical systems yields a compact operator over the exponential dot product kernel's native space. Hence, norm convergence of the DMD procedure is established when using scaled Liouville operators, which is a decided advantage over Koopman based DMD methods.

\end{abstract}

\section{Introduction}

DMD has emerged as an effective method of extracting fundamental governing principles from high-dimensional time series data. The method has been employed successfully in the field of fluid dynamics, where DMD methods have demonstrated an ability to determine dynamic modes, also known as ``Koopman modes,'' which agree with Proper Orthogonal Decomposition (POD) analyses (cf. \cite{budivsic2012applied,mezic2019,kutz2016dynamic,mezic2005spectral,mezic2013analysis,williams2015data,williams2015kernel}). However, DMD methods employing Koopman operators do not address continuous time dynamical systems directly. Instead, current DMD methods analyze discrete time proxies of continuous time systems \cite{kutz2016dynamic}. The discretization process constrains Koopman based DMD methods to systems that are forward complete \cite{bittracher2015pseudogenerators}. The objective of the present manuscript is to develop DMD methods that avoid discretization of continuous time dynamical systems, while providing convergence results that are stronger than Koopman based DMD and applicable to a broader class of dynamical systems.

The connection between Koopman operators and DMD relies on the idea that a finite dimensional nonlinear dynamical system can be expressed as a linear operator over an infinite dimensional space to enable treatment of the nonlinear system via tools from the theory of linear systems and linear operators. The idea of lifting finite dimensional nonlinear systems into infinite dimensional linear ones has been successfully utilized in the literature to achieve various identification and control objectives; however, a few fundamental limitations severely restrict the class of systems for which the connection between Koopman operators and DMD can be established via lifting to infinite dimensions. In particular, this article focuses on the following.

\noindent \textbf{Existence of Koopman operators in continuous time systems:} Consider the continuous time dynamical system given as $\dot x = 1 + x^2$. The discretization with time step $1$ yields the following discrete dynamics: $x_{i+1} = F(x_i) := \tan(1+\arctan(x_i)).$ It should be immediately apparent that $F$ is not well defined over $\mathbb{R}$. In fact, through the consideration of $x_i = \tan(\pi/2 - 1)$ it can be seen that $F(x_i)$ is undefined. Since the symbol for a Koopman operator must be defined over the entire domain, there is no well defined Koopman operator arising from this discretization. Hence, the resultant Koopman operator cannot be expected to be well-defined. Note that the example above is not anecdotal. In addition to commonly used examples in classical works, such as \cite{khalil2002nonlinear}, mass-action kinetics in thermodynamics \cite[Section 6.3]{haddad2019dynamical}, chemical reactions \cite[Section 8.4]{toth2018reaction}, and species populations \cite[Section 4.2]{hallam2012mathematical} often give rise to such models. In general, unless the solutions of the continuous time dynamics are constrained to be forward complete, (for example, by assuming that the dynamical systems are globally Lipschitz \cite[Chapter 1]{coddington1955theory}) the resultant Koopman operator cannot be expected to be well-defined. This observation is validated by \cite{bittracher2015pseudogenerators}, but otherwise conditions on the dynamics are largely absent from the literature. 

\noindent \textbf{Boundedness of Koopman operators:} Even in the case of globally Lipschitz models, results regarding convergence of the DMD operator to the Koopman operator rely on the assumption that the Koopman operator is bounded over a specified RKHS (cf. \cite{korda2018convergence}). Boundedness of composition operators, like the Koopman operator, has been an active area of study in the operator theory community. Indeed, it turns out there are very few bounded composition operators over many function spaces. A canonical example is in the study of the Bargmann-Fock space, where only affine symbols yield bounded composition operators and of those the compact operators arise from $F(z) = a z + b$ where $|a| < 1.$ This is a very small collection, and reveals how unlikely it is to have a bounded Koopman operator arising from the discretization of continuous time nonlinear systems. 

\noindent \textbf{Practical utility of convergence results:} In the DMD literature, convergence of the DMD operator to the Koopman operator is typically established in the strong operator topology (SOT). However, as noted in \cite{korda2018convergence}, since SOT convergence is the topology of pointwise convergence \cite{pedersen2012analysis}, it is not sufficient to justify use of the DMD operator to interpolate or extrapolate the system behavior from a collection of samples. Furthermore, by selecting a complete set of observables and adding them one at a time to a finite rank representation of the Koopman operator, pointwise convergence is to be expected. This is a restatement of the more general result that all bounded operators may be approximated by finite rank operators in SOT, which itself is a specialization of a much broader result for topological vector spaces (cf. \cite[pg. 172]{pedersen2012analysis}). While \cite{korda2018convergence} also provides theoretically interesting insights into convergence of the eigenvalues and the eigenvectors of the DMD operator to eigenvalues and eigenfunctions of the Koopman operator along a subsequence, without the means to identify the convergent subsequences, practical utility of subsequential convergence is limited. In contrast, norm convergence is uniform convergence for operators, and yields a bound on the error over the kernels corresponding to the entire data set. Thus, a meaningful convergence result would arise from the norm convergence of finite rank representations to Koopman operators. However, this result is only possible for \textit{compact} Koopman operators, which are virtually nonexistent in applications of interest.

The above discussion brings into question the impact of various approaches to the study of continuous time dynamical systems through discretization and Koopman operators, which all rely on the compactness, boundedness, or existence of Koopman operators. These limitations necessitate a new direction in the study of operators that avoid the limiting processes that lean on the assumption of bounded or even \textit{extant} Koopman operators. In this paper, we develop a new DMD technique that directly utilizes Liouville operators (which include Koopman generators).

{\color{blue}A subset of Liouville operators that require the assumption of forward completeness on the dynamical system, called Koopman generators, have been studied as limits of Koopman operators in works such as \cite{cvitanovic2005chaos,das2020koopman,froyland2014detecting,giannakis2019data,giannakis2020extraction,giannakis2018koopman}.} The present work sidesteps the limiting process, and as a result, the assumptions regarding existence of Koopman operators, through the use of ``occupation kernels''. Specifically, occupation kernels remove the burden of approximation from that of operators and places it on the estimation of occupation kernels from time-series data, which requires much less theoretical overhead. The action of the adjoint of a Liouville operator on an occupation kernel provides the input-output relationships that enable DMD of time series data. Consequently, Liouville operators may be directly examined via occupation kernels, while avoiding limiting relations involving Koopman operators that might not be well defined for a particular discretization of a continuous time nonlinear dynamical system.

The direct involvement of Liouville operators in a DMD routine allows for the study of dynamics that are locally rather than globally Lipschitz, since Liouville operators do not impose an \textit{a priori} discretization. For the adjoint of a Liouville operator to be well defined, the operator must be densely defined over the underlying RKHS \cite{rosenfeld2019occupation2,rosenfeld2019occupation}. As a result, the exact class of dynamical systems that may be studied using Liouville operators depends on the selection of the RKHS. However, the requirement that the Liouville operator must be densely defined is not overly restrictive. For example, on the real valued Bargmann-Fock space, Liouville operators are densely defined for a wide range of dynamics that are expressible as real entire functions (which includes polynomials, exponentials, sine, and cosine, etc.).

Perhaps the strongest case for Liouville operators is the fact that they can be ``scaled'' to generate compact operators. Section \ref{sec:compact} this paper introduces the idea of scaled Liouville operators as variants of Liouville operators that are compact for a large class of dynamical systems over the Bargmann-Fock space. Scaled Liouville operators make slight adjustments to the data by scaling the trajectories by a single parameter $|a| < 1$. Through the selection of $a$ close to $1$, scaled Liouville operators yield compact operators that are numerically indistinguishable from the corresponding unbounded Liouville operators over a given compact workspace. More importantly, the DMD procedure performed on scaled Liouville operators yields a sequence of finite rank operators that \textit{converge in norm} to the scaled Liouville operators (see Theorem \ref{thm:DMD-norm-convergence}).

\noindent \textbf{Practical benefits of the developed method:} In addition to the theoretical benefits of Liouville operators detailed above, there are several practical benefits that arise from the use of occupation kernels and Liouville operators. Quadrature techniques, such as Simpson's rule, allow for the efficient estimation of occupation kernels while mitigating signal noise \cite{rosenfeld2019occupation2,rosenfeld2019occupation}, and also provide a robust estimation of the action of Liouville operators on occupation kernels. Furthermore, as snapshots are being integrated into trajectories for the generation of occupation kernels, the method presented in this manuscript can naturally incorporate irregularly sampled data. The results of this manuscript allow for the natural separation of DMD methods for discrete time systems using the Koopman operator and continuous time systems using the Liouville operator, where previously continuous time systems were artificially discretized to fit within the Koopman framework.

In DMD, a large finite dimensional representation of the linear operator is constructed from data (i.e., \textit{snapshots}) using a collection of \textit{observables}. A subsystem of relatively small rank is then determined via a singular value decomposition (SVD) and approximation of the linear operator by the small rank subsystem is supported by a direct mapping between the eigenfunctions of the former and the eigenvectors of the latter \cite{williams2015kernel}. The fact that the rank of the smaller subsystem is typically in agreement with the number of snapshots, which can be considerably smaller than the number of observables, makes DMD particularly useful when there is a small number of snapshots of a high dimensional system. However, direct application of DMD to high dimensional systems sampled at high frequencies still poses a significant computational challenge, where many snapshots may have to be discarded to produce a computationally tractable problem, as was done in \cite[Example 2.3]{kutz2016dynamic}. Such systems include mechanical systems with high sampling frequencies \cite{cichella2015cooperative,walters2018online}, and neurobiological systems recorded via electroencephalography (EEG) where the typical sampling frequencies are of the order of $500$ Hz \cite{gruss2019sympathetic}. In the development of DMD for continuous time systems the methods in the present manuscript also replace snapshots with trajectories or segments of trajectories of the system. The emphasis on trajectories over individual snapshot reduces the dimensionality of systems with an insurmountable number of snapshot without discarding any data.

The presented algorithm obviates the truncated SVD that is utilized throughout DMD methods. For example, in \cite{williams2015kernel} the truncated SVD is leveraged to convert from a feature space representation of the action of the Koopman operators to an approximation of the Gram matrix and an ``interaction matrix.'' This stands in opposition of the spirit of the ``kernel trick,'' where kernel functions are a means to avoid any direct interface with feature space. Following \cite{rosenfeld2021dynamic}, the presented algorithm is given purely with respect to the occupation kernels, and the resultant methods are considerably simpler than what is seen in \cite{williams2015kernel}.

\noindent \textbf{A comparison with similar literature:} Liouville operators are studied in the context of DMD procedures using limiting definitions in works such as \cite{klus2020data}. The manuscript \cite{klus2020data}, which was posted to arXiv around the same time as the first draft of this manuscript, approaches the Koopman generator through Galerkin methods. While the signs that the field is expanding beyond Koopman operators is encouraging, the authors of \cite{klus2020data} still adopt the limiting definitions of the Koopman generator in their work, which is an artifact from ergodic theory. Quantities similar to occupation kernels have been studied in the literature previously, such as occupation measures and time averaging functionals. These other quantities lack some of the essential characeristics that accompany occupation kernels. Occupation kernels and occupation measures both represent the same functional over different spaces. Occupation measures are in the dual space of the Banach space of continuous functions, while occupation kernels are functions in a RKHS. As such, functions in the RKHS may be estimated through projections on occupation kernels, and this fact is leveraged in Section \ref{sec:occkernel-dmd} where finite rank representations of the Liouville operators arise from the matrix representation of the projection operator. Occupation kernels are distinct from time average functionals, where the latter is the average of a sum of iterated applications of the Koopman operator to an observable. As per the discussion above, this limits the applicability of time average functionals to globally Lipschitz dynamics, and therefore their utilization is more constrained from that of occupation kernels, whose definition is independent of Koopman and Liouville operators.

The relevant preliminary concepts for the theoretical underpinnings of the approach taken in the present manuscript are reviewed in Section \ref{sec:rkhs}. This includes definitions and properties of RKHSs as well as densely defined operators and their adjoints.

\section{Technical Preliminaries}
\subsection{Reproducing Kernel Hilbert Spaces}
\label{sec:rkhs}

\begin{definition}
A reproducing kernel Hilbert space (RKHS) over a set $X$ is a Hilbert space of functions from $X$ to $\mathbb{R}$ such that for each $x \in X$, the evaluation functional $E_x g := g(x)$ is bounded.
\end{definition}

By the Reisz representation theorem, corresponding to each $x \in X$ there is a function $k_x \in H$ such that for all $g \in H$, $\langle g, k_x \rangle_H = g(x)$. The kernel function corresponding to $H$ is given as $K(x,y) = \langle k_y, k_x \rangle_H.$ The kernel function is a positive definite function in the sense that for any finite number of points $\{ c_1, c_2, \ldots, c_M \} \subset X$, the corresponding Gram matrix,
\[ \begin{pmatrix}
K(c_1,c_1) & \cdots & K(c_1,c_M)\\
\vdots & \ddots & \vdots\\
K(c_M,c_1) & \cdots & K(c_M,c_M)
\end{pmatrix}
\]
is positive semi-definite. The Gram matrix arises in many contexts in machine learning, such as in support vector machines (cf. \cite{hastie2005elements}), and particular to the subject matter of this manuscript, it plays a pivotal role in the construction of the kernel-based extended DMD method of \cite{williams2015kernel} and the occupation kernel approach presented herein.

The Aronszajn-Moore theorem states that there is a unique correspondence between RKHSs and positive definite kernel functions \cite{aronszajn1950theory}. That is the RKHS may be constructed directly from the kernel function itself or the kernel function may be determined by a RKHS through the Reisz theorem. When the RKHS is obtained from the kernel function, it is frequently referred to as the native space of that kernel function \cite{wendland2004scattered}.

RKHSs interact with function theoretic operators, such as Koopman (composition) operators \cite{jury2007c,luery2013composition,williams2015kernel}, multiplication operators \cite{rosenfeld2015densely,rosenfeld2015introducing}, and Toeplitz operators \cite{rosenfeld2016sarason}, in many nontrivial ways. For example, the kernel functions themselves play the role of eigenfunctions for the adjoints of multiplication operators \cite{szafraniec2000reproducing}, and when the function corresponding to a Koopman operator has a fixed point at $c \in X$, the kernel function centered at that point (i.e. $K(\cdot,c) \in H$) is an eigenfunction for the adjoint of the Koopman operator \cite{cowen1995composition}. The kernel functions can also be demonstrated to be in the domain of the adjoint of densely defined Koopman operators as will be demonstrated in Section \ref{sec:occkernels}.

For machine learning applications kernel functions are frequently used for dimensionality reduction by expressing the inner product of data cast into a high dimensional feature space as evaluation of the kernel function itself \cite{steinwart2008support,hastie2005elements}. Specifically, a feature map corresponding to a RKHS is given as the mapping $x \mapsto \Psi(x) := (\Psi_1(x),\Psi_2(x),\ldots)^T \in \ell^2(\mathbb{N})$ for $x \in X$ such that $K(x,y) = \langle \Psi(y), \Psi(x) \rangle_{\ell^2(\mathbb{N})}$.  That is, kernel function may be expressed as \[ K(x,y) = \sum_{m=1}^\infty \Psi_m(x) \Psi_m(y). \]
The feature space expression for a function $ g \in H $ is given as $ \mathbf g = (g_1, g_2, \ldots)^T \in \ell^2(\mathbb{N}) $ so that $ g(x) = \langle \mathbf g, \Psi(x) \rangle_{\ell^2(\mathbb{N})} = \langle g, K(\cdot,x) \rangle_H $. This representation of inner products of vectors in a features space as evaluation of a kernel function is central to the usage of kernel methods in data science, where the feature space is generally unknown but may be accessed through the kernel function. The approach taken in \cite{williams2015kernel} uses the feature space as the fundamental basis for their representation and obtains kernel functions through a truncated SVD, whereas the present work avoids the invocation of the feature space and the truncated SVD.

The most frequently employed RKHS in machine learning applications is that of the Gaussian radial basis function's native space. The Gaussian radial basis function is given as $K(x,y) = \exp\left( -\frac{1}{\mu} \| x-y\|_2^2 \right)$, and it is a positive definite function over $\mathbb{R}^n$ for all $n$.

Another important kernel is the exponential kernel dot product kernel, $K(x,y) = \exp\left( \frac{1}{\mu} x^Ty \right)$, which is also a positive definite function over $\mathbb{R}^n$. What is significant concerning the exponential dot product kernel is that its native space is the Bargmann-Fock space, where bounded Koopman operators have been completely classified. Another significant feature which will be leveraged is that polynomials are dense inside the Bargmann-Fock space with respect to the Hilbert space norm.

\subsection{Adjoints of Densely Defined Liouville Operators}
\label{sec:occkernels}

In the study of operator theory, the theory concerning bounded operators is most complete (cf. \cite{pedersen2012analysis,folland2013real}). A bounded operator over a Hilbert space is a linear operator $W: H \to H$ such that $\| Wg \|_H \le C \| g \|_H$ for some $C > 0$. The minimum $C$ that holds for all $g \in H$ is the norm of $W$ and written as $\| W \|$. A classical theorem in operator theory states that the collection of bounded operators is precisely the collection of continuous operators over a Hilbert space (or more generally a Banach space) \cite[Chapter 5]{folland2013real}.

Unbounded operators over a Hilbert space are linear operators given as $W: \mathcal{D}(W) \to H$, where $\mathcal{D}(W)$ is the domain contained within $H$ on which the operator $W$ is defined \cite[Chapter 5]{pedersen2012analysis}. When the domain of $W$ is dense in $H$, $W$ is said to be a densely defined operator over $H$. While unbounded operators are by definition discontinuous, closed operators over a Hilbert space satisfy weaker limiting relations. That is, an operator is closed if $\{ g_m \}_{m=1}^\infty \in \mathcal{D}(W)$, and both $\{ g_m \}_{m=1}^\infty$ and $\{ Wg_m \}_{m=1}^\infty$ are convergent sequences where $g_m \to g \in H$ and $Wg_m \to h \in H$, then $g \in \mathcal{D}(W)$ and $Wg = h$ \cite[Chapter 5]{pedersen2012analysis}. The Closed Graph Theorem states that if $W$ is a closed operator such that $\mathcal{D}(W) = H$, then $W$ is bounded.

\begin{lemma}\label{lem:closedoperators}A Liouville Operator with symbol $f$, defined as $ A_f g = \nabla g \cdot f $, that has the canonical domain \[\mathcal{D}(A_f) := \{ g \in H : \nabla g \cdot f \in H \}, \] is closed over RKHSs that are composed of continuously differentiable functions.\end{lemma}

\begin{proof}
Liouville operators were demonstrated to be closed in \cite{rosenfeld2019occupation}.\qed
\end{proof}

The closedness of Koopman operators is well known in the study of RKHS, where they are more commonly known as composition operators (cf. \cite{jury2007c,luery2013composition}). Beyond the limit relations provided by closed operators, the closedness of an unbounded operator plays a signficant role in the study of the adjoints of unbounded operators \cite[Chapter 5]{pedersen2012analysis}.

\begin{definition}
For an operator $W$ let \[\mathcal D(W^*) := \{ h \in H : g \mapsto \langle Wg, h \rangle_H \text{ is bounded } \}\] be dense in $H$. 
For each $h \in \mathcal{D}(W^*)$ the Reisz theorem guarantees a function $W^*h \in H$ such that $\langle Wg, h \rangle_H = \langle g, W^*h \rangle_H$. The adjoint of the operator $W$ is thus given as $W^* : \mathcal{D}(W^*) \to H$ via the assignment $h \mapsto W^*h$.
\end{definition}

For a closed operator over a Hilbert space, the adjoint is densely defined \cite{pedersen2012analysis}. Hence, Liouville operators with domains given as in Lemma \ref{lem:closedoperators}, their adjoints are densely defined. Specific members of the domain of the respective adjoints may be identified, and these functions will be utilized in the characterization of the DMD methods in the subsequent sections. To characterize the interaction between the trajectories of a dynamical system and the Liouville operator, the notion of occupation kernels must be introduced (cf. \cite{rosenfeld2019occupation}).

\begin{definition}
Let $X$ be a metric space, $\gamma:[0,T] \to X$ be an essentially bounded measurable trajectory, and let $H$ be a RKHS over $X$ consisting of continuous functions. Then the functional $g \mapsto \int_0^T g(\gamma(t)) dt$ is bounded, and the Reisz theorem guarantees a function $\Gamma_{\gamma} \in H$ such that \[ \langle g, \Gamma_\gamma \rangle_H = \int_0^T g(\gamma(t)) dt \] for all $g \in H$. The function $\Gamma_{\gamma}$ is the \textit{occupation kernel} corresponding to $\gamma$ in $H$.
\end{definition}

\begin{lemma}\label{lem:adjointAction}
Let $f: \mathbb{R}^n \to \mathbb{R}^n$ be the dynamics for a dynamical system, and suppose that $\gamma : [0,T] \to \mathbb{R}^n$ is a trajectory satisfying $\dot \gamma = f(\gamma(t))$ in the Caretheodory sense. In this setting, $\Gamma_\gamma \in \mathcal{D}(A_f^*)$. Moreover, $A_f^* \Gamma_{\gamma} = K(\cdot, \gamma(T)) - K(\cdot,\gamma(0)).$
\end{lemma}

\begin{proof}
This lemma was established in \cite{rosenfeld2019occupation}.\qed
\end{proof}




For Liouville operators, several examples can be demonstrated where particular symbols produce densely defined operators over the Bargmann-Fock space. In particular, since polynomials are dense in the Bargmann-Fock space, for polynomial dynamics, $f$, the function $A_f g = \nabla g \cdot f$ is a polynomial whenever $g$ is a polynomial. Hence, polynomial dynamical systems correspond to densely defined Liouville operators over the Bargmann-Fock space, and it should be noted that this is not a complete characterization of the densely defined Liouville operators over this space. Moreover, for other RKHSs, different classes of dynamics will correspond to densely defined operators, and this requires independent evaluation for each RKHS.

\section{A Compact Variation of the Liouville Operator}
\label{sec:compact}

One of the drawbacks of employing either the Koopman operator or the Liouville operator for DMD is that the finite rank matrices produced by the method are strictly heuristic representations of the modally unbounded operators. An important question to address is whether a DMD procedure may be produced using a compact operator other than those densely defined operators discussed so far. This section presents a class of compact operators for use in DMD applied to continuous time systems. The compactness and boundedness of the operators will depend on the selection of the RKHS and the dynamics of the system. The Bargmann-Fock space will be utilized in this section, and the compactness assumption will be demonstrated to hold for a large class of dynamics.

\begin{definition}
Let $H$ be a RKHS over $\mathbb{R}^n$, $a \in \mathbb{R}$ with $|a| < 1$, and let the scaled Liouville operator with symbol $f:\mathbb{R}^n \to \mathbb{R}^n$, \[ A_{f,a} : \mathcal{D}(A_{f,a}) \to H, \] be given as $A_{f,a} g(x) = a\nabla g(ax)f(x)$ for all $x \in \mathbb{R}^n$ and
\[g \in \mathcal{D}(A_{f,a}) = \{ h \in H : a\nabla h(ax) f(x) \in H\}.\]
\end{definition}

From the definition of scaled Liouville operators, if $\gamma : [0,T] \to \mathbb{R}^n$ is a trajectory satisfying $\dot \gamma = f(\gamma)$, then \[ \int_0^T A_{f,a} g(\gamma(t)) dt = \int_0^T a \nabla g(a\gamma(t)) f(\gamma(t)) dt = \langle A_{f,a} g, \Gamma_\gamma \rangle_{H}.\] 

\begin{proposition}\label{prop:scaled-adjoint}
For $\gamma : [0,T] \to \mathbb{R}^n$, such that $\dot \gamma = f(\gamma)$, $\Gamma_{\gamma} \in \mathcal{D}(A_{f,a}^*)$ and \[ A_{f,a}^* \Gamma_{\gamma} = K(\cdot,a\gamma(T)) - K(\cdot,a\gamma(0)).\]
\end{proposition}

Theorem \ref{thm:scaled-compact} and Corollary \ref{cor:polynomials-compact} demonstrate that for the Bargmann-Fock space, a large class of dynamics correspond to compact scaled Liouville operators.

\begin{theorem}\label{thm:scaled-compact}
Let $F^2(\mathbb{R}^n)$ be the Bargmann-Fock space of real valued functions, which is the native space for the exponential dot product kernel, $K(x,y) = \exp(x^Ty)$, $a \in \mathbb{R}$ with $|a| < 1$, and let $A_{f,a}$ be the scaled Liouville operator with symbol $f:\mathbb{R}^n \to \mathbb{R}^n$. There exists a collection of coefficients $\{ C_\alpha \}_{\alpha}$ indexed by the multi-index $\alpha$ such that if $f$ is representable by a multi-variate power series, $f(x) = \sum_{\alpha} f_\alpha x^\alpha$ satisfying \[ \sum_{\alpha} |f_{\alpha}| C_\alpha < \infty, \] then $A_{f,a}$ is bounded and compact over $F^2(\mathbb{R}^n)$.
\end{theorem}

\begin{proof}
The proof has been relegated to the appendix to ease exposition.\qed
\end{proof}

\begin{corollary}\label{cor:polynomials-compact}
If $f$ is a multi-variate polynomial, then $A_{f,a}$ is bounded and compact over $F^2(\mathbb{R}^n)$ for all $|a|<1$.
\end{corollary}

The compactness of scaled Liouville operators (over the Bargmann-Fock space) position them as an answer for the desire of norm convergence for DMD type methods. For bounded Koopman operators, work such as \cite{korda2018convergence} obtain convergence in the strong operator topology (SOT) of the DMD method to the Koopman operator. SOT convergence is only pointwise convergence over a Hilbert space, and does not provide any generalization guarantees in the learning sense. Norm convergence on the other hand gives a uniform bound on the error estimates for all functions in the Hilbert space. Specifically, the norm convergence of the DMD procedure given in \ref{sec:occkernel-dmd} will be demonstrated to converge in norm to the scaled Liouville operator.

While scaled Liouville operators are not identical to the Liouville operator, the selection of the parameter $a$ close to $1$ limits the adjustments to the time series data through the operator to within machine precision, and hence the decomposition of scaled Liouville operators is computationally indistinguishable from that of the Liouville operator for $a$ sufficiently close to $1$.


\section{Occupation Kernel Dynamic Mode Decomposition}
\label{sec:occkernel-dmd}
\subsection{Finite rank representation of the Liouville operator}
With the relevant theoretical background presented, this section develops the Occupation Kernel-based DMD method for continuous time systems. This method differs from the kernel-based extended DMD method of \cite{williams2015kernel}, where the kernel functions for the inputs are now replaced by occupation kernels, and the output is now a difference of kernel functions. This formulation allows for the snapshots of typical DMD methods to be strung together as trajectories. The occupation kernel-based DMD method then allows for the incorporation of all the snapshots of a given system to be incorporated into the DMD analysis in a way that reduces the dimensionality of the resultant problem to be less than the number of snapshots, while simultaneously allowing for the direct treatment of continuous time dynamical systems. If the rank of the resulting matrices needs to be increased, the trajectories may be segmented up to the number of snapshots.

It should also be noted that this method differs from \cite{williams2015kernel} in that it avoids direct evaluations of the feature space. Thus, the succeeding method keeps with the spirit of the ``kernel trick,'' where the feature space is only accessed through evaluation of the kernel functions \cite[pg. 19]{steinwart2008support}.

Let $K$ be the kernel function for a RKHS, $H$, over $\mathbb{R}^n$ consisting of continuously differentiable functions. Let $\dot x = f(x)$ be a dynamical system corresponding to a densely defined Liouville operator, $A_f$, over $H$. Suppose that $\{ \gamma_i : [0,T_i] \to X \}_{i=1}^M$ is a collection of trajectories satisfying $\dot \gamma_i = f(\gamma_i)$. There is a corresponding collection of \textit{occupation kernels}, $\alpha := \{ \Gamma_{\gamma_i} \}_{i=1}^M \subset H$, given as $\Gamma_{\gamma_i}(x) := \int_0^{T_i} K(x,\gamma_i(t)) dt.$ For each $\gamma_i$ the action of $A_f^*$ on the corresponding occupation kernel is $A_f^* \Gamma_{\gamma_i} = K(\cdot, \gamma_i(T_{i})) - K(\cdot, \gamma_i(0))$.

Thus, when $\alpha$ is selected as an ordered basis for a vector space, the action of $A_f^*$ is known on $\vspan(\alpha)$. The objective of the DMD procedure is to express a matrix representation of the operator $A_f^*$ on the finite dimensional vector space spanned by $\alpha$ followed by projection onto $\vspan(\alpha)$.

Let $w_1,\cdots,w_M$ be the coefficients for the projection of a function $g \in H$ onto $\vspan(\alpha) \subset H$, written as $P_\alpha g = \sum_{i=1}^{M} w_i\Gamma_{\gamma_{i}}$. Using the fact that 
\[
    \langle g , \Gamma_{\gamma_j}\rangle_H = \langle P_\alpha g , \Gamma_{\gamma_j}\rangle_H = \begin{pmatrix} \langle \Gamma_{\gamma_1} , \Gamma_{\gamma_j}\rangle_H & \cdots & \langle \Gamma_{\gamma_M} , \Gamma_{\gamma_j}\rangle_H \end{pmatrix} \begin{pmatrix} w_1 \\ \vdots \\ w_M \end{pmatrix},
\]
for all $ j = 1, \cdots, M $, the coefficients $ w_1, \cdots, w_M $ may be obtained through the solution of the following linear system: 
\begin{equation}
    \begin{pmatrix}
        \langle \Gamma_{\gamma_1}, \Gamma_{\gamma_1} \rangle_H & \cdots & \langle \Gamma_{\gamma_M}, \Gamma_{\gamma_1} \rangle_H\\
        \vdots & \ddots & \vdots\\
        \langle \Gamma_{\gamma_1}, \Gamma_{\gamma_M} \rangle_H & \cdots & \langle \Gamma_{\gamma_M}, \Gamma_{\gamma_M} \rangle_H
    \end{pmatrix}
    \begin{pmatrix}
        w_1 \\ \vdots \\ w_M
    \end{pmatrix}
    =
    \begin{pmatrix}
        \langle g, \Gamma_{\gamma_1} \rangle_H\\
        \vdots\\
        \langle g, \Gamma_{\gamma_M} \rangle_H
    \end{pmatrix},\label{eq:projectionMatrix}
\end{equation}
where each of the inner products may be expressed as either single or double integrals as
\begin{gather}
    \langle \Gamma_{\gamma_j}, \Gamma_{\gamma_i} \rangle_H = \int_0^{T_i} \int_0^{T_j} K(\gamma_i(\tau),\gamma_j(t)) dt d\tau \text{, and }
    \langle g, \Gamma_{\gamma_i} \rangle_H = \int_0^{T_i} g(\gamma_i(t))dt.\label{eq:integralInnerProduct}
\end{gather}
Furthermore, if $h = \sum_{i=1}^{M} v_i\Gamma_{\gamma_{i}} \in \vspan(\alpha)$ for some coefficients $\{v_i\}_{i=1}^M\subset \mathbb{R}$, then $ A_f^* h \in H $, and it follows that
\begin{equation}\small
    \langle A_f^* h,\Gamma_{\gamma_j} \rangle = \left\langle \sum_{i=1}^{M} v_i A_f^*\Gamma_{\gamma_{i}},\Gamma_{\gamma_j} \right\rangle_H = \begin{pmatrix}\left\langle  A_f^*\Gamma_{\gamma_{1}},\Gamma_{\gamma_j} \right\rangle_H,\cdots,\left\langle  A_f^*\Gamma_{\gamma_{M}},\Gamma_{\gamma_j} \right\rangle_H\end{pmatrix}\begin{pmatrix}v_1\\\vdots\\ v_M\end{pmatrix}, \label{eq:occKernelonSpanAlpha}
\end{equation}
for all $j = 1,\cdots,M$. Using \eqref{eq:projectionMatrix} and \eqref{eq:occKernelonSpanAlpha}, the coefficients $ \{w_i\}_{i=1}^M $ in the projection of $ A_f^* h $ onto $ \vspan(\alpha) $ can be expressed as
\begin{gather*}
    \begin{pmatrix}
        w_1 \\ \vdots \\ w_M
    \end{pmatrix}=
    \begin{pmatrix}
        \langle \Gamma_{\gamma_1}, \Gamma_{\gamma_1} \rangle_H & \cdots & \langle \Gamma_{\gamma_M}, \Gamma_{\gamma_1} \rangle_H\\
        \vdots & \ddots & \vdots\\
        \langle \Gamma_{\gamma_1}, \Gamma_{\gamma_M} \rangle_H & \cdots & \langle \Gamma_{\gamma_M}, \Gamma_{\gamma_M} \rangle_H
    \end{pmatrix}^{-1}
    \times\begin{pmatrix}
        \langle A_f^* \Gamma_{\gamma_1}, \Gamma_{\gamma_1} \rangle_H & \cdots & \langle A_f^* \Gamma_{\gamma_M}, \Gamma_{\gamma_1} \rangle_H\\
        \vdots & \ddots & \vdots\\
        \langle A_f^* \Gamma_{\gamma_1}, \Gamma_{\gamma_M} \rangle_H & \cdots & \langle A_f^* \Gamma_{\gamma_M}, \Gamma_{\gamma_M} \rangle_H
    \end{pmatrix}
    \begin{pmatrix}
        v_1 \\ \vdots \\ v_M
    \end{pmatrix}.
\end{gather*}
Lemma \ref{lem:adjointAction} then yields the finite rank representation for $ P_\alpha A_f^* $, restricted to the occupation kernel basis, $ \vspan(\alpha) $, as $ [P_\alpha A_f^*]_{\alpha}^{\alpha} = G^{-1}\mathcal{I} $ where
\[
    G:=\begin{pmatrix}
        \langle \Gamma_{\gamma_1}, \Gamma_{\gamma_1} \rangle_H & \cdots & \langle \Gamma_{\gamma_1}, \Gamma_{\gamma_M} \rangle_H\\
        \vdots & \ddots & \vdots\\
        \langle \Gamma_{\gamma_M}, \Gamma_{\gamma_1} \rangle_H & \cdots & \langle \Gamma_{\gamma_M}, \Gamma_{\gamma_M} \rangle_H
    \end{pmatrix}
\]
is the Gram matrix of occupation kernels and
\[\small
    \mathcal{I}:=\begin{pmatrix}
        \langle K(\cdot,\gamma_1(T_1)) - K(\cdot, \gamma_1(0)), \Gamma_{\gamma_1} \rangle_H & \cdots & \langle K(\cdot,\gamma_M(T_M)) - K(\cdot, \gamma_M(0)), \Gamma_{\gamma_1} \rangle_H\\
        \vdots & \ddots & \vdots\\
        \langle K(\cdot,\gamma_1(T_1)) - K(\cdot, \gamma_1(0)), \Gamma_{\gamma_M} \rangle_H & \cdots & \langle K(\cdot,\gamma_M(T_M)) - K(\cdot, \gamma_M(0)), \Gamma_{\gamma_M} \rangle_H
    \end{pmatrix}.
\]
is the interaction matrix.
DMD requires a finite-rank representation of $ P_{\alpha}A_f $, instead of $ P_{\alpha}A_f^* $. Similar to the development above, Lemma \ref{lem:adjointAction} can be used to generate a finite rank representation of $ P_{\alpha}A_f $ under the following additional assumption.
\begin{assumption}
   The occupation kernels are in the domain of the Liouville operator, i.e., $\alpha \subset \mathcal{D}(A_f)$.\label{ass:kernels_in_domain_A_f}
\end{assumption}
Given $h = \sum_{i=1}^{M} v_i\Gamma_{\gamma_{i}} \in \vspan(\alpha)$ for some coefficients $ \{v_i\}_{i=1}^M \subset \mathbb{R} $,  Assumption \ref{ass:kernels_in_domain_A_f} implies that $A_f h\in H$ and
\begin{multline}
    \left\langle A_f h,\Gamma_{\gamma_j}\right\rangle_H = \sum_{i=1}^M v_i\left\langle A_f\Gamma_{\gamma_i},\Gamma_{\gamma_j}\right\rangle_H = \sum_{i=1}^M v_i\left\langle \Gamma_{\gamma_i}, A_f^* \Gamma_{\gamma_j}\right\rangle_H \\= \left(\left\langle \Gamma_{\gamma_1}, A_f^* \Gamma_{\gamma_j}\right\rangle_H,\ldots,\left\langle \Gamma_{\gamma_M}, A_f^* \Gamma_{\gamma_j}\right\rangle_H\right)\begin{pmatrix}
        v_1\\\vdots\\v_M
    \end{pmatrix}.
\end{multline}
Lemma 2 then yields a finite rank representation of $ P_\alpha A_f $, restricted to $\vspan(\alpha) $ as
\begin{equation}
    [P_\alpha A_f]_\alpha^\alpha = G^{-1}\mathcal{I}^T.
\end{equation}
\subsection{Dynamic mode decomposition}
Suppose that $\lambda_i$ is the eigenvalue corresponding to the eigenvector $v_i := (v_{i1}, v_{i2}, \ldots, v_{iM})^T$, $i=1,\ldots,M$, of $[P_\alpha A_f]_{\alpha}^\alpha$. The eigenvector $v_i$ can be used to construct a normalized eigenfunction of $ P_{\alpha} A_f $ restricted to $\vspan(\alpha)$, given as $\varphi_i = \frac{1}{N_i} \sum_{j=1}^M v_{ij} \Gamma_{\gamma_j}$, where $N_i := \sqrt{v_i^\dagger G v_i}$, and $(\cdot)^\dagger$ denotes the conjugate transpose. Let $V$ be the matrix of coefficients of the normalized eigenfunctions arranged so that each column corresponds to an eigenfunction.

The DMD procedure begins by expressing the identity function, also known as the full state observable, $g_{id}(x) := x \in \mathbb{R}^n$ as a combination of the approximate eigenfunctions of $A_f$ and \textit{Liouville modes} $\xi_i \in \mathbb{R}^n$ as $g_{id}(x) \approx \sum_{i=1}^M \xi_i \varphi_i(x)$. The $j$-th row of the matrix $\xi = (\xi_1 \cdots \xi_M)$ is obtained as
\[
    \begin{pmatrix}(\xi_1)_j & \cdots & (\xi_M)_j\end{pmatrix} = \left(\begin{pmatrix} \langle \varphi_1,\varphi_1 \rangle_H & \cdots & \langle \varphi_1,\varphi_M\rangle_H\\
    \vdots & \ddots & \vdots\\
    \langle \varphi_M,\varphi_1\rangle_H & \cdots & \langle \varphi_M,\varphi_M \rangle_H \end{pmatrix}^{-1}
    \begin{pmatrix} \langle (x)_j, \varphi_i \rangle_H \\ \vdots \\ \langle (x)_j, \varphi_i \rangle_H \end{pmatrix}\right)^T
\]
and $(x)_j$ is viewed here as the functional mapping $x \in \mathbb{R}^n$ to its $j$-th coordinate. By examining the inner products $ \left\langle g_{id},\Gamma_{\gamma_i}\right\rangle_H $, for $i=1,\ldots,M$, the matrix $\xi$ may be expressed as
\begin{equation}
    \xi = \left((V^T G V)^{-1} V^T \begin{pmatrix} \int_0^T \gamma_1(t)^T dt\\ \vdots\\ \int_0^T \gamma_M(t)^T dt \end{pmatrix}\right)^T. \label{eq:Liouville_modes}
\end{equation}

Given a trajectory $x$ satisfying $\dot x = f(x)$, each eigenfunction of $A_f$ satisfies $\dot \varphi_i(x(t)) = \lambda_i \varphi_i(x(t))$ and hence, $\varphi_i(x(t)) = \varphi_i(x(0)) e^{\lambda_i t}$, and the following data driven model is obtained:
\begin{equation}
    x(t) \approx \sum_{i=1}^M \xi_i \varphi_i(x(0)) e^{\lambda_i t}, \label{eq:data-driven_model}
\end{equation}
where
\begin{equation}
    \varphi_i(x(0)) = \frac{1}{N_i} \sum_{j=1}^M v_{ij} \Gamma_{\gamma_j}(x(0)) = \frac{1}{N_i} \sum_{j=1}^M v_{ij} \int_0^{T_j} K\left(x(0),\gamma_j(t)\right).\label{eq:coefficients}
\end{equation}

\begin{algorithm}
\caption{Pseudocode for the dynamic mode decomposition routine of Section \ref{sec:occkernel-dmd}. The choice of numerical integration routine can have a significant impact on the overall results, and it is advised that a high accuracy method is leveraged in practice.}
\label{alg:occDMD}
\begin{algorithmic}[1]
\STATE{Input: Sampled trajectories $\{ \gamma_{j} :[0,T] \to \mathbb{R}^n \}_{j=1}^M $, Kernel function $ K:\mathbb{R}^n\times\mathbb{R}^n\to\mathbb{R} $ of an RKHS, initial condition, $ x(0) \in \mathbb{R}^n $, and a numerical integration routine}
\STATE{Compute the Gram matrix $ G $ using \eqref{eq:integralInnerProduct} and a numerical integration routine}
\STATE{Compute the interaction matrix $ \mathcal{I} $ using \eqref{eq:integralInnerProduct} and a numerical integration routine}
\STATE{Compute eigenvalues, $\lambda_i$, and eigenvectors, $v_i$, of $G^{-1}\mathcal{I}^T$}
\STATE{Use \eqref{eq:coefficients} and a numerical integration routine to compute the coefficients $ \varphi_i(x(0)) $}
\RETURN{Liouville modes, $ \xi $, coefficients, $ \varphi_i(x(0)) $, and eigenvalues $ \lambda_i $}
\end{algorithmic}
\end{algorithm}

\subsection{Modifications for the Scaled Liouville Operator DMD Method}


Since Liouville operators are not generally compact, convergence of the finite rank representation $ P_\alpha A_f $, restricted to $\vspan{\alpha}$, to $A_f$ cannot be guaranteed as $M\to\infty$. Convergence of the finite rank representation can be established in the case of the scaled Liouville operators and the approximations obtained via DMD, under Assumption \ref{ass:kernels_in_domain_A_f}, are provably cogent. 

By Theorem 2, for an infinite collection of trajectories $\{ \gamma_{i}\}_{i=1}^\infty$ with a dense collection of corresponding occupation kernels, $\{ \Gamma_{\gamma_{i}}\}_{i=1}^\infty \subset H$, the resultant sequence of finite rank operators $P_{\alpha_M}A_{f,a} P_{\alpha_M}$ converges to $A_{f,a}$, where $\alpha_M := \{ \Gamma_{\gamma_{1}},\ldots, \Gamma_{\gamma_{M}}\}$. Consequently, the spectrum of $[P_{\alpha_M}A_{f,a}]_{\alpha_M}^{\alpha_M}$, the finite rank representation of $ P_{\alpha_M} A_{f,a} $, restricted to $\vspan(\alpha)$, converges to that of $A_{f,a}$.\footnote{It should be noted that the operator $P_{\alpha_M} A_{f,a} P_{\alpha_M}$ is simply $P_{\alpha_M} A_{f,a}$ when restricted to $\vspan(\alpha_M)$ as $P_{\alpha_M} g = g$ for all $g \in \vspan(\alpha_M)$.
}

Furthermore, when $a$ is sufficiently close to $1$ and the observed trajectories contained in a compact set are perturbed to within machine precision, the finite rank representations of $A_{f,a}$ and $A_{f}$ are computationally indistinguishable.

DMD using scaled Liouville operators is similar to the unscaled case. In particular, recall that for $|a| < 1$ and $f$ as above, $A^*_{f,a} \Gamma_{\gamma_{i}} = K(\cdot,a\gamma_i(T_i)) - K(\cdot,a\gamma_i(0)).$ Hence, a finite rank representation of $A_f$, obtained from restricting and projecting to $\vspan(\alpha)$, is given as 
\[
[P_\alpha A_{f,a}]_{\alpha}^{\alpha} = G^{-1}\mathcal{I}_a^T,
\]
where
\[\footnotesize\medmuskip=0mu\thickmuskip=0mu
    \mathcal{I}_a:=\begin{pmatrix}
        \langle K(\cdot,a\gamma_1(T_1)) - K(\cdot, a\gamma_1(0)), \Gamma_{\gamma_1} \rangle_H & \cdots & \langle K(\cdot,a\gamma_M(T_M)) - K(\cdot, a\gamma_M(0)), \Gamma_{\gamma_1} \rangle_H\\
        \vdots & \ddots & \vdots\\
        \langle K(\cdot,a\gamma_1(T_1)) - K(\cdot, a\gamma_1(0)), \Gamma_{\gamma_M} \rangle_H & \cdots & \langle K(\cdot,a\gamma_M(T_M)) - K(\cdot, a\gamma_M(0)), \Gamma_{\gamma_M} \rangle_H
    \end{pmatrix}.
\]
The approximate normalized eigenfunctions, $\{ \varphi_{i,a} \}_{i=1}^M$, for $A_{f,a}$ may then be obtained in an identical fashion as for the Liouville operator.

Thus, the expression of the full state observable, $g_{id}$, in terms of the eigenfunctions yields $g_{id}(x) \approx \sum_{i=1}^M \xi_{i,a} \varphi_{i,a}(x)$ with (scaled) Liouville modes $\xi_{i,a}$.

As the eigenfunctions satisfy \[\dot \varphi_{i,a}(a x(t)) = a \nabla \varphi_i(a x(t)) f(x(t)) = A_{f,a} \varphi_{i,a}(x(t)) = \lambda_{i,a} \varphi_{i,a}(x(t)),\] it can be seen that $\varphi_{i,a}(x(t)) \neq e^{t\lambda_{i,a}} \varphi_{i,a}(x(0)).$ When $a$ is close to $1$, it can be demonstrated that $\varphi_{i,a}(x(t))$ is very nearly equal to $e^{t\lambda_{i,a}} \varphi_{i,a}(x(0)),$ and the error can be controlled when $x(t)$ remains in a compact domain or workspace.

\begin{proposition}\label{prop:estimate-exponential}
Let $H$ be a RKHS of twice continuously differentiable functions over $\mathbb{R}^n$, $f$ be Lipschitz continuous, and suppose that $\varphi_{i,a}$ is an eigenfunction of $A_{f,a}$ with eigenvalue $\lambda_{i,a}$. Let $D$ be a compact subset of $\mathbb{R}^n$ that contains $x(t)$ for all $0 < t < T$. In this setting, if $\lambda_{i,a} \to \lambda_{i,1}$ and $\varphi_{i,a}(x(0)) \to \varphi_{i,1}(x(0))$ as $a \to 1^-$, then \[\sup_{0 \le t \le T} \|\varphi_{i,a}(x(t)) - e^{\lambda_{i,a}t}\varphi_{i,a}(x(0))\|_2 \to 0.\]
\end{proposition}

\begin{proof}
The proof has been relegated to the appendix to ease exposition.\qed
\end{proof}

Thus, under the hypothesis of Proposition \ref{prop:estimate-exponential}, for $a$ sufficiently close to $1$, a data-driven model for a trajectory $x$ satisfying $\dot x = f(x)$ is established as
\begin{equation}
    x(t) \approx \sum_{i=1}^M \xi_{i,a} \varphi_{i,a}(x(0)) e^{\lambda_{i,a} t}.\label{eq:scaled-data-driven_model}
\end{equation}

The principle advantage of using scaled Liouville operators is that these operators are compact over the Bargmann-Fock space for a large collection of nonlinear dynamics. Moreover, the sequence finite rank operators obtained through the DMD procedure achieve norm convergence when a sequence of trajectories corresponds to a collection of occupation kernels that are dense in the Hilbert space.

\begin{theorem}\label{thm:DMD-norm-convergence}
Let $|a| < 1$. Suppose that $\{ \gamma_{i}:[0,T_i] \to \mathbb{R}^n \}_{i=1}^\infty$ is a sequence of trajectories satisfying $\dot \gamma = f(\gamma)$ for a dynamical system $f$ corresponding to a compact scaled Liouville operator, $A_{f,a}$. If the collection of functions, $\{ \Gamma_{\gamma_i} \}_{i=1}^\infty$ is dense in the Bargmann-Fock space, then the sequence of operators $\{ P_{\alpha_M} A_{f,a} P_{\alpha_M} \}_{M=1}^\infty$ converges to $A_{f,a}$ in the norm topology, where $\alpha_M = \{ \Gamma_{\gamma_1}, \ldots, \Gamma_{\gamma_M} \}$. 
\end{theorem}

\begin{proof}The proof has been relegated to the appendix to ease exposition.\qed\end{proof}

\section{Numerical Experiments}
\label{sec:experiments}

This section gives the results two collections of numerical experiments using the methods of the paper. 
The first surround the problem of flow across a cylinder, which has become a classic example for DMD. This provides a benchmark for comparison of the present method with kernel-based extended DMD. There it will be demonstrated that scaled Liouville modes and Liouville modes are very similar.

The second experiment performs a decomposition using electroencephalography (EEG) data, which has been sampled at 250 Hz over a period of 8 seconds. The high sampling frequency gives a large number of snapshots, which then leads to a high dimensional learning problem when using the snapshots alone. The purpose of this experiment is to demonstrate how the Liouville operator based DMD can incorporate the large number of snapshots to generate Liouville modes without discarding data.

\subsection{Flow Across a Cylinder}

This experiment utilizes the data set from \cite{kutz2016dynamic}, which includes snapshots of flow velocity and flow vorticity generated from a computational fluid dynamics simulation. The data correspond to the wake behind a circular cylinder, and the Reynolds number for this flow is $100$. The simulation was generated with time steps of $\Delta t = 0.02$ second and ultimately sampled every $10 \Delta t$ seconds yielding $151$ snapshots. Each snapshot of the system is a vector of dimension $89,351$. More details may be found in \cite[Chapter 2]{kutz2016dynamic}.

Figure \ref{fig:liouville-modes} presents the Liouville modes obtained from the cylinder vorticity data set, and this figure should be compared with Figure \ref{fig:scaled-liouville-modes}, which presents the scaled Liouville modes, with parameter $a = 0.99$, corresponding to the same data set. The modes were generated using the Gaussian kernel with $\mu = 500$ and the collection snapshots was subdivided into 147 trajectories of length $5$. Figure \ref{fig:reconstruction} compares shapshots of the true vorticity against vorticity reconstructed using the unscaled and scaled Liouville DMD models in \eqref{eq:data-driven_model} and \eqref{eq:scaled-data-driven_model}, respectively.

\begin{figure}
    \begin{minipage}{0.25\textwidth}
        \centering Real part
    \end{minipage}\hfill\begin{minipage}{0.25\textwidth}
        \centering Imaginary part
    \end{minipage}
    
    \begin{minipage}{1.0\textwidth}
        \raisebox{-.5\height}{\includegraphics[width=0.25\textwidth]{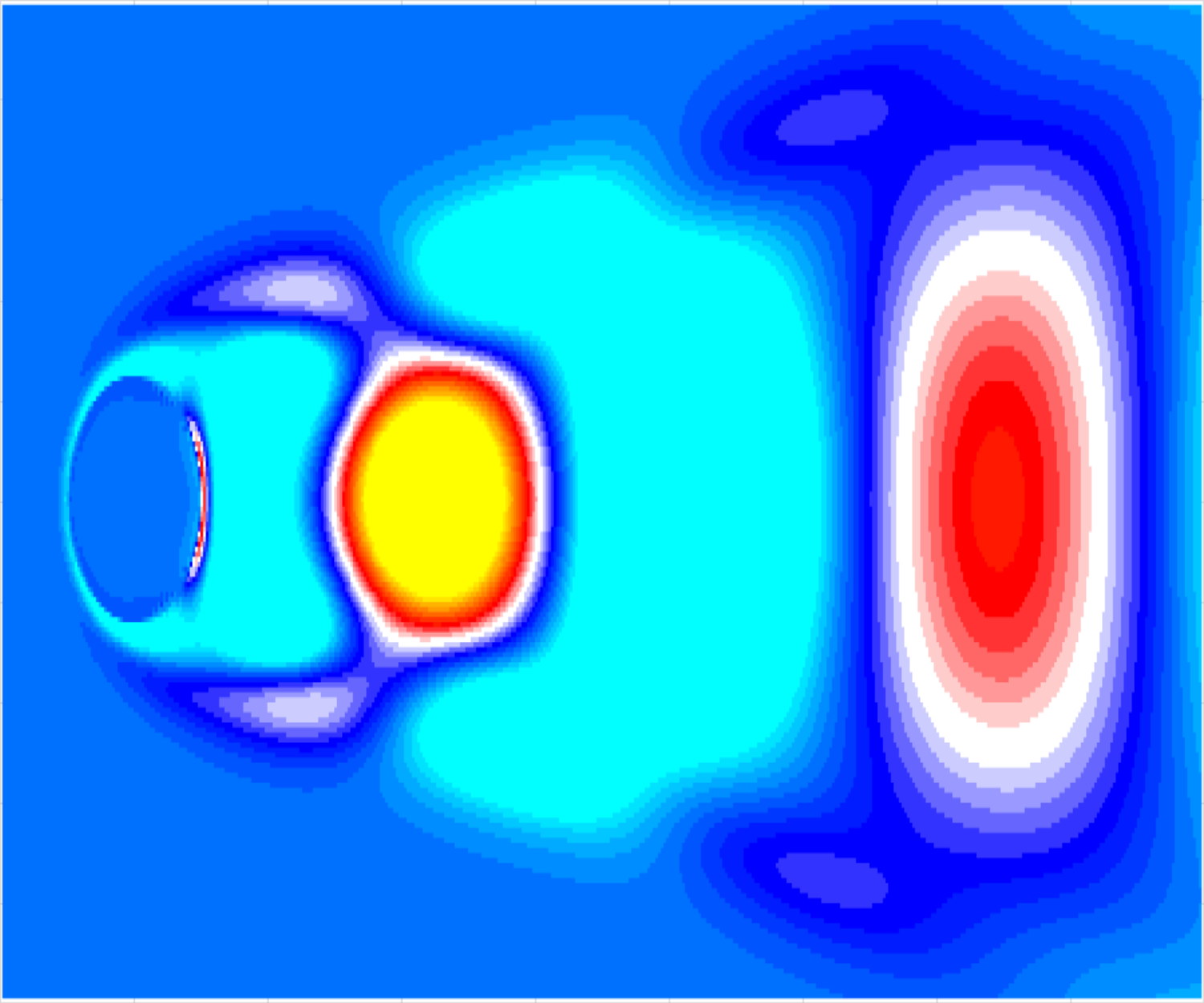}}
        \hfill
        Mode 142
        \hfill
        \raisebox{-.5\height}{\includegraphics[width=0.25\textwidth]{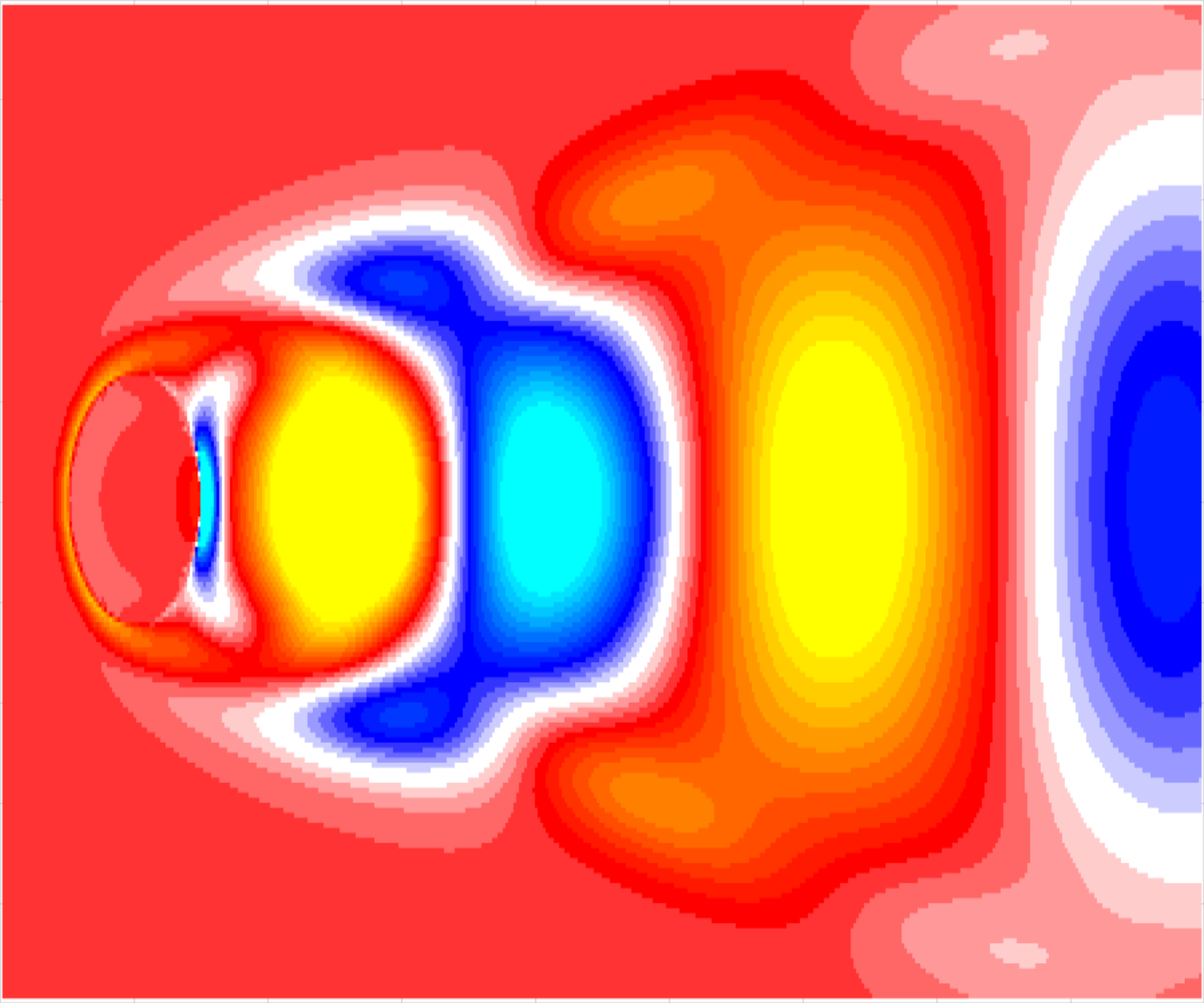}}
    \end{minipage}
    
    \begin{minipage}{1.0\textwidth}
        \raisebox{-.5\height}{\includegraphics[width=0.25\textwidth]{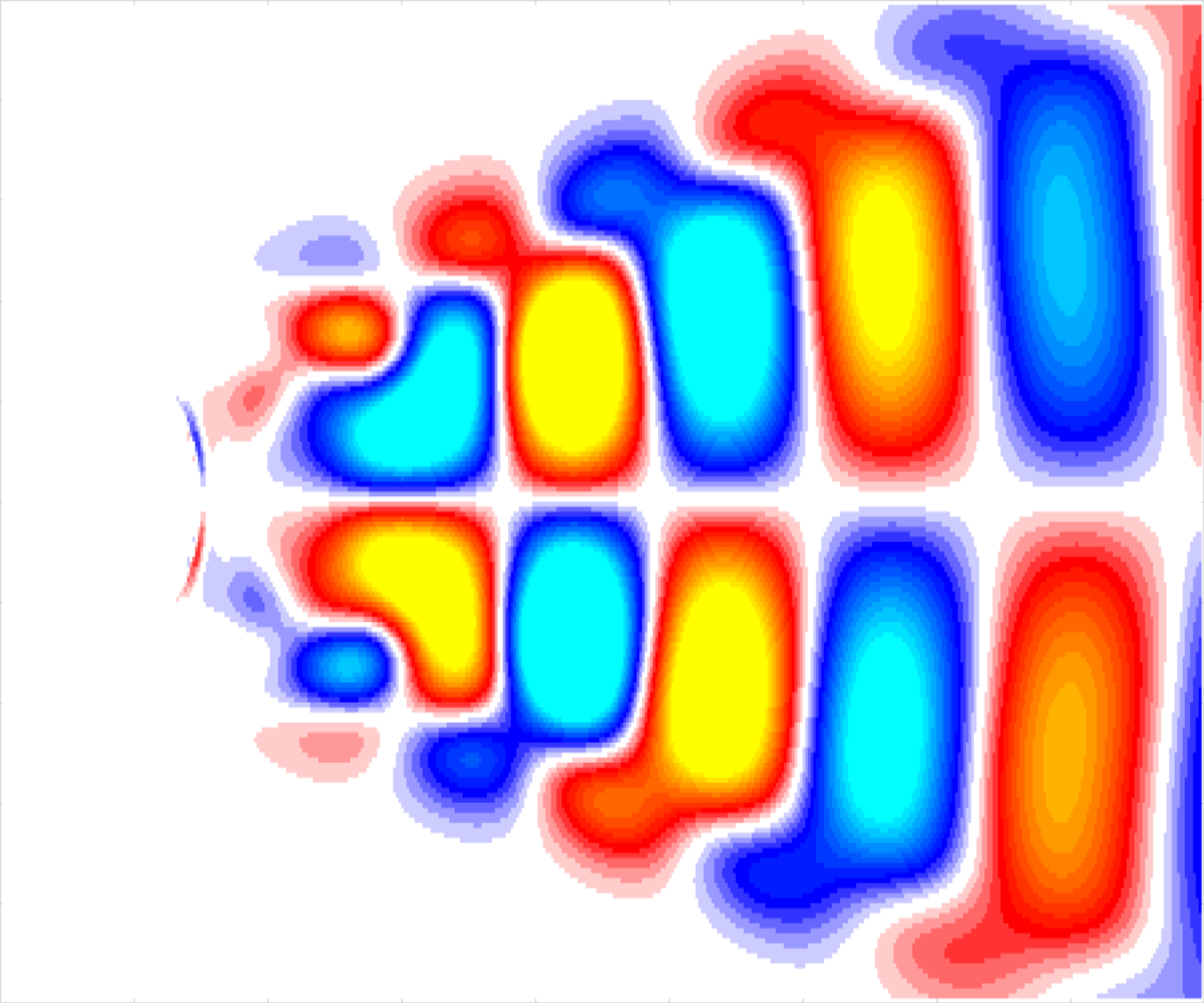}}
        \hfill
        Mode 120
        \hfill
        \raisebox{-.5\height}{\includegraphics[width=0.25\textwidth]{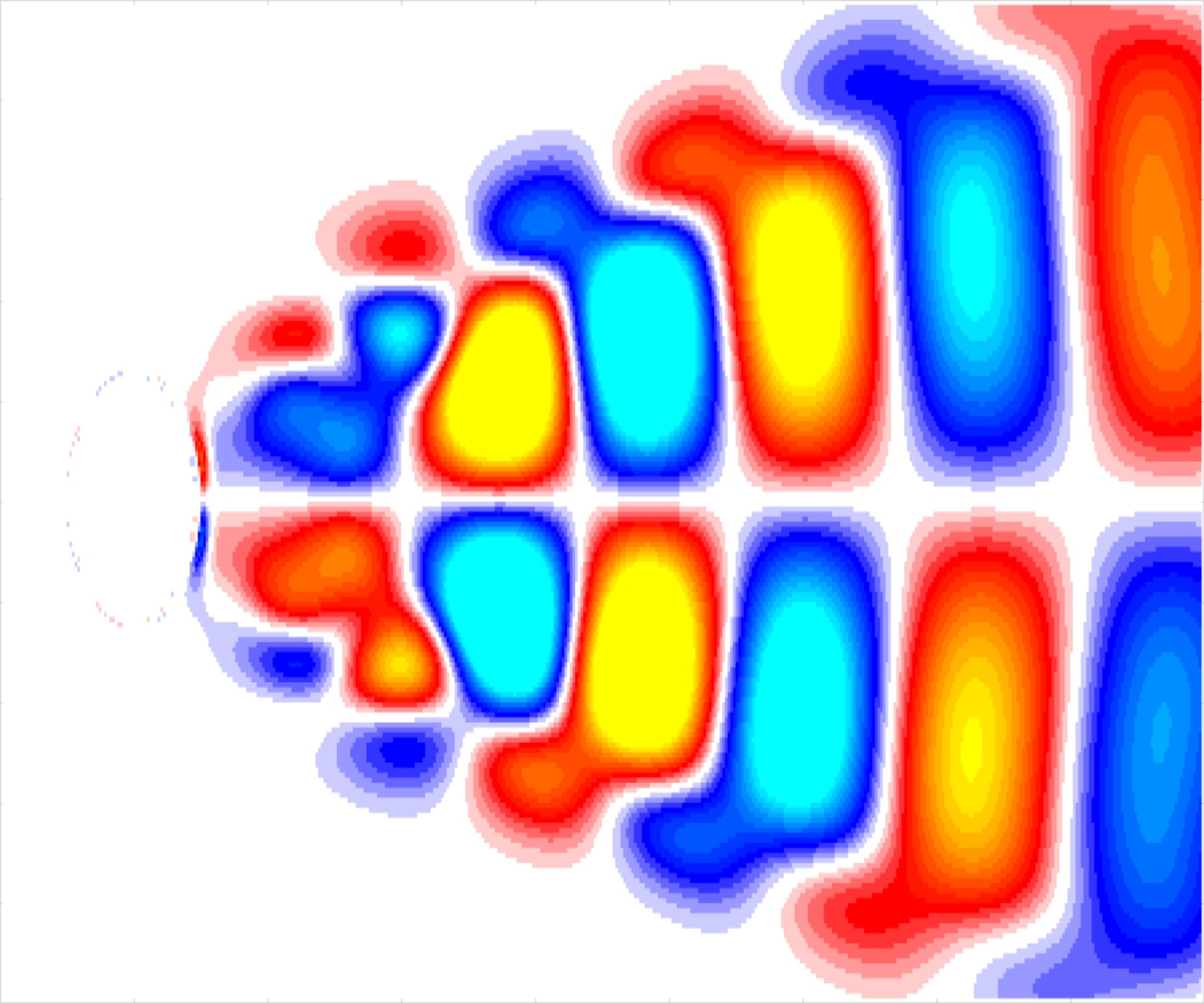}}
    \end{minipage}
    
    \begin{minipage}{1.0\textwidth}
        \raisebox{-.5\height}{\includegraphics[width=0.25\textwidth]{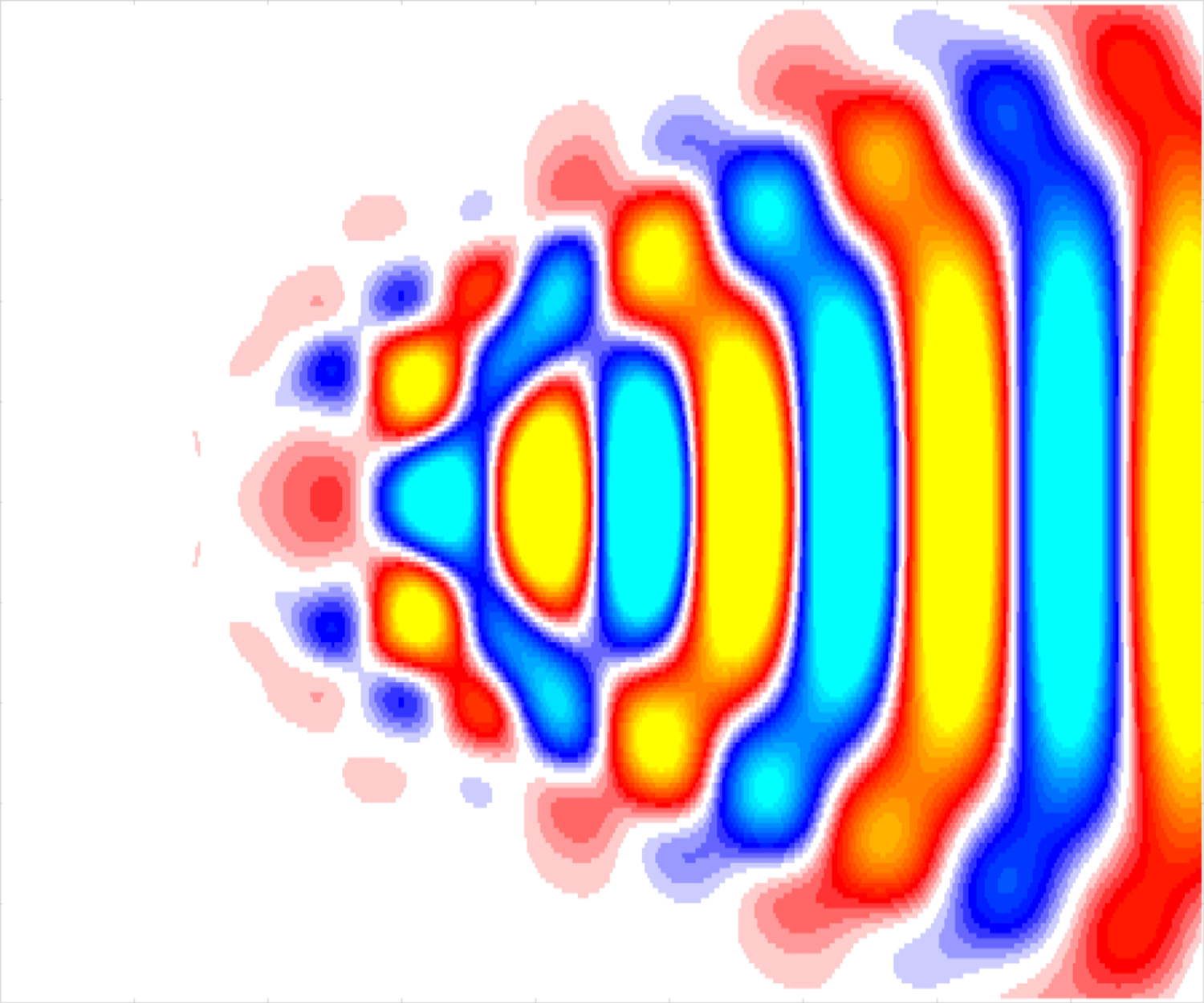}}
        \hfill
        Mode 103
        \hfill
        \raisebox{-.5\height}{\includegraphics[width=0.25\textwidth]{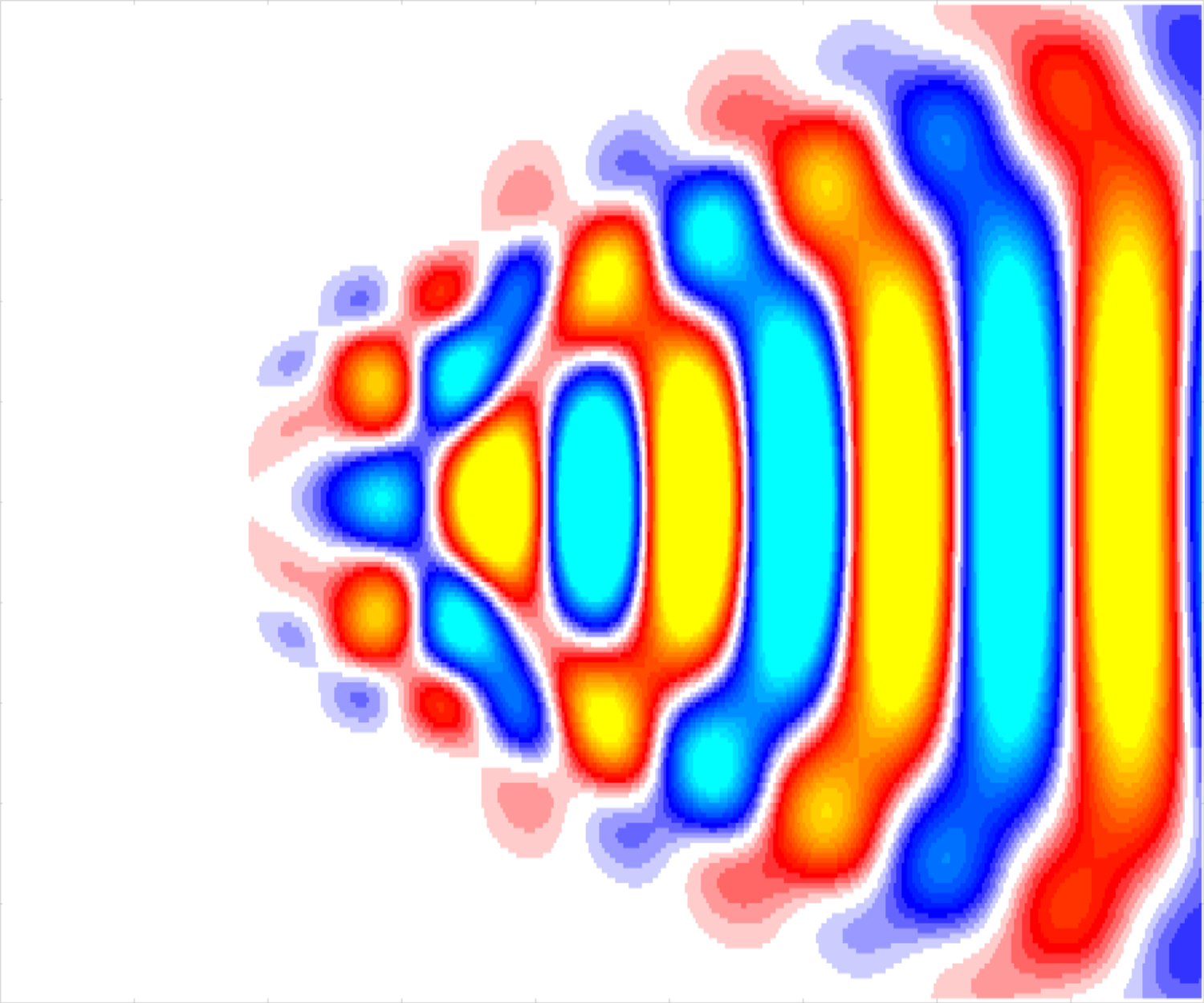}}
    \end{minipage}
    
    \begin{minipage}{1.0\textwidth}
        \raisebox{-.5\height}{\includegraphics[width=0.25\textwidth]{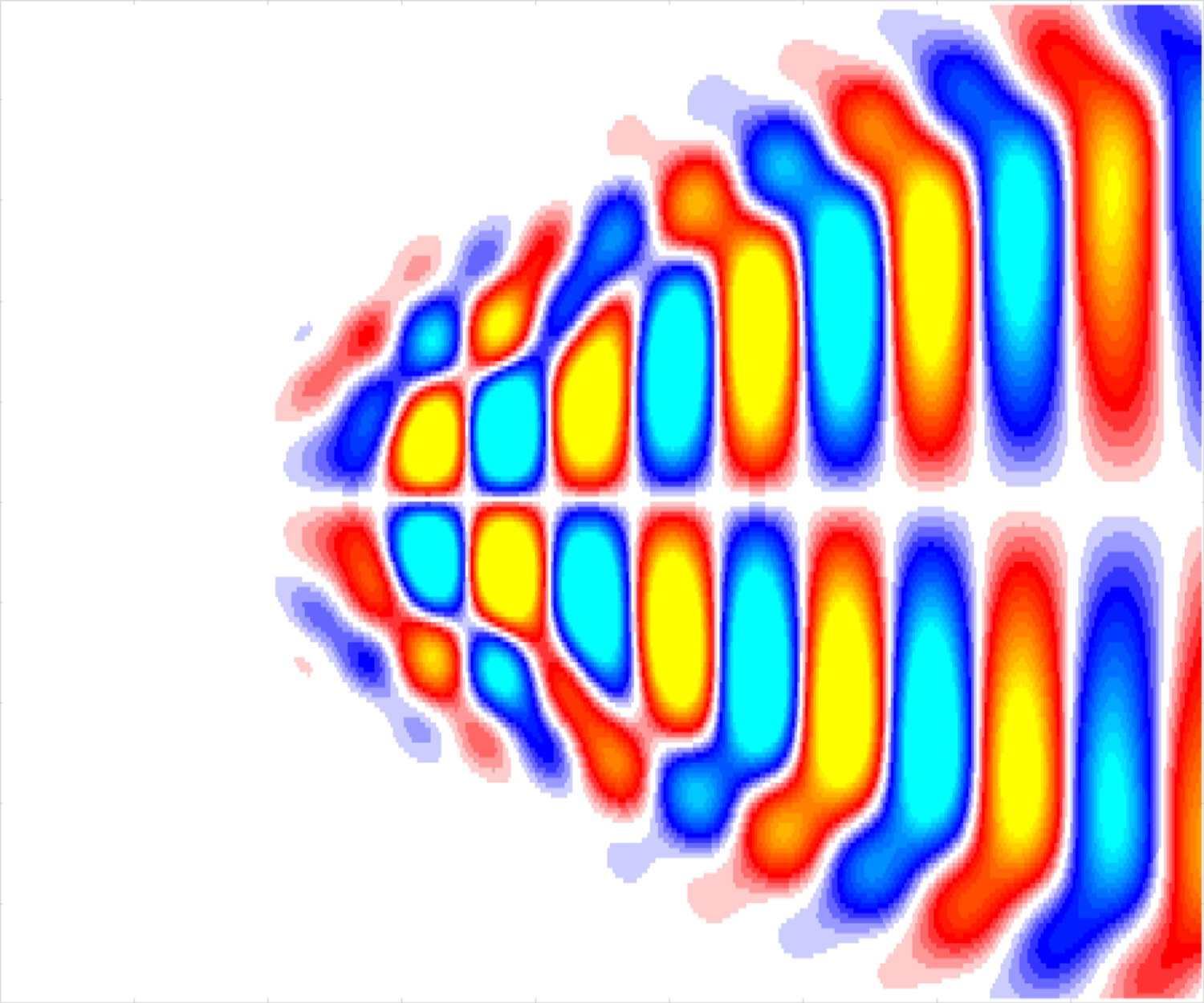}}
        \hfill
        Mode 94
        \hfill
        \raisebox{-.5\height}{\includegraphics[width=0.25\textwidth]{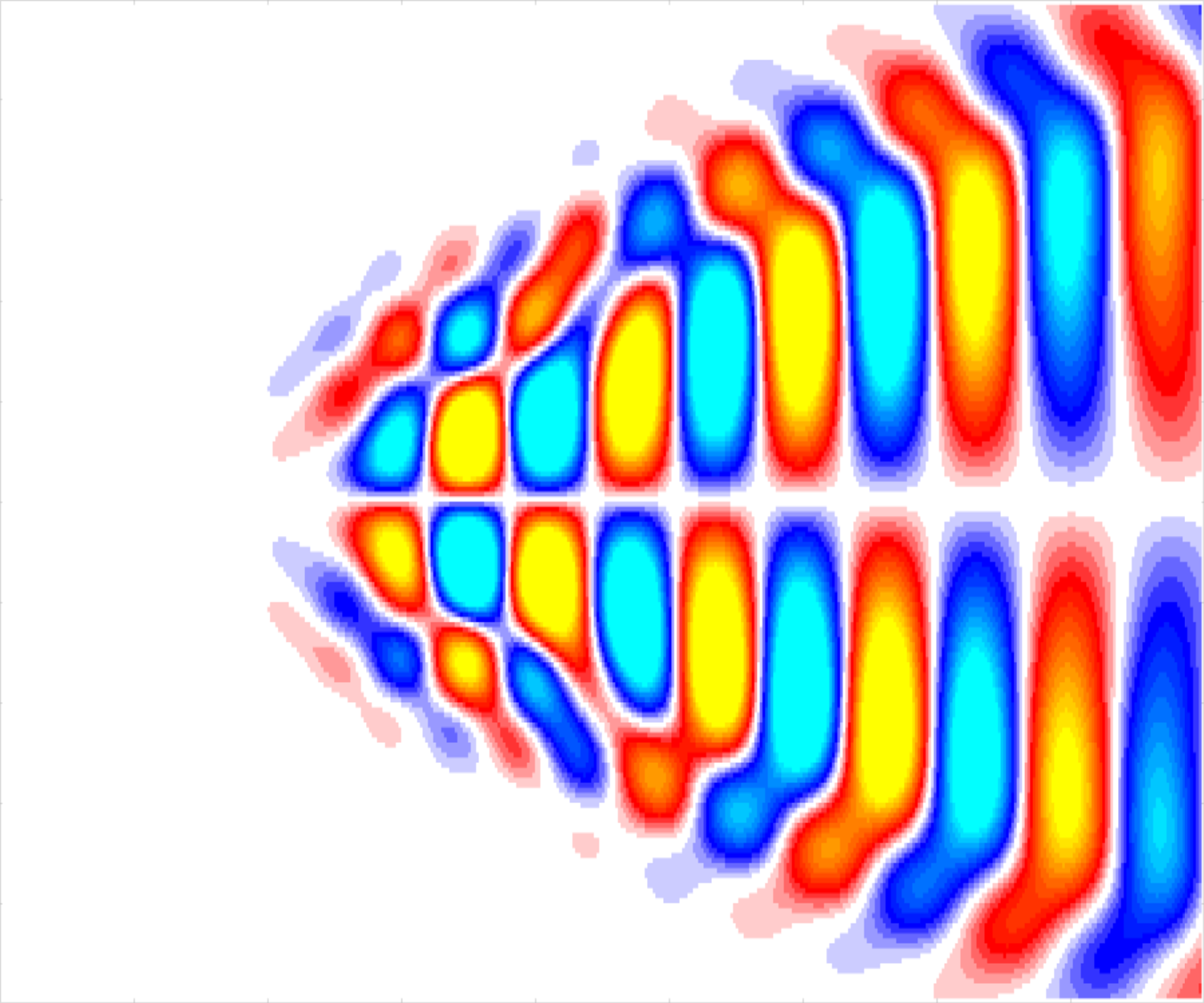}}
    \end{minipage}
    
    \begin{minipage}{1.0\textwidth}
        \raisebox{-.5\height}{\includegraphics[width=0.25\textwidth]{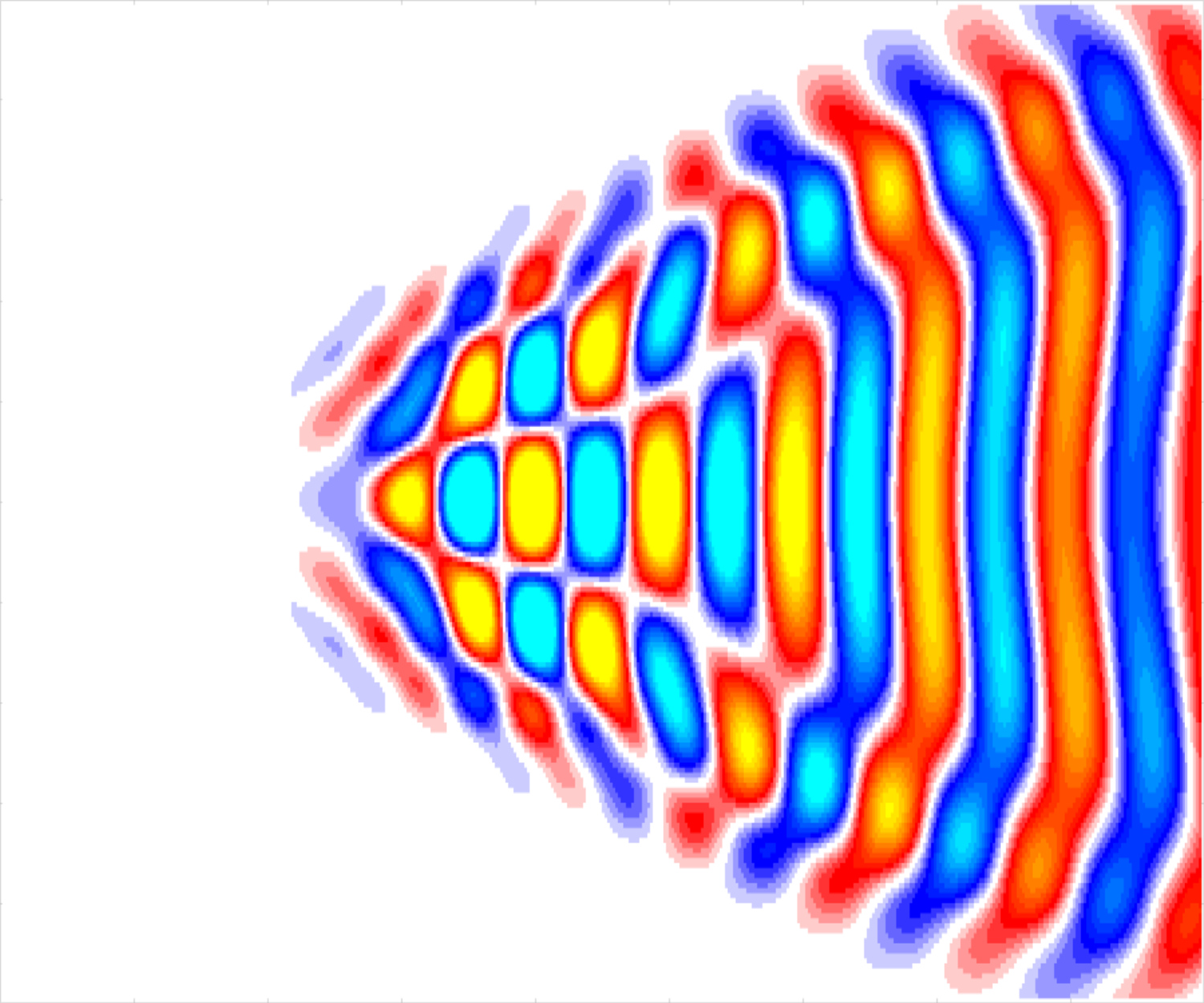}}
        \hfill
        Mode 80
        \hfill
        \raisebox{-.5\height}{\includegraphics[width=0.25\textwidth]{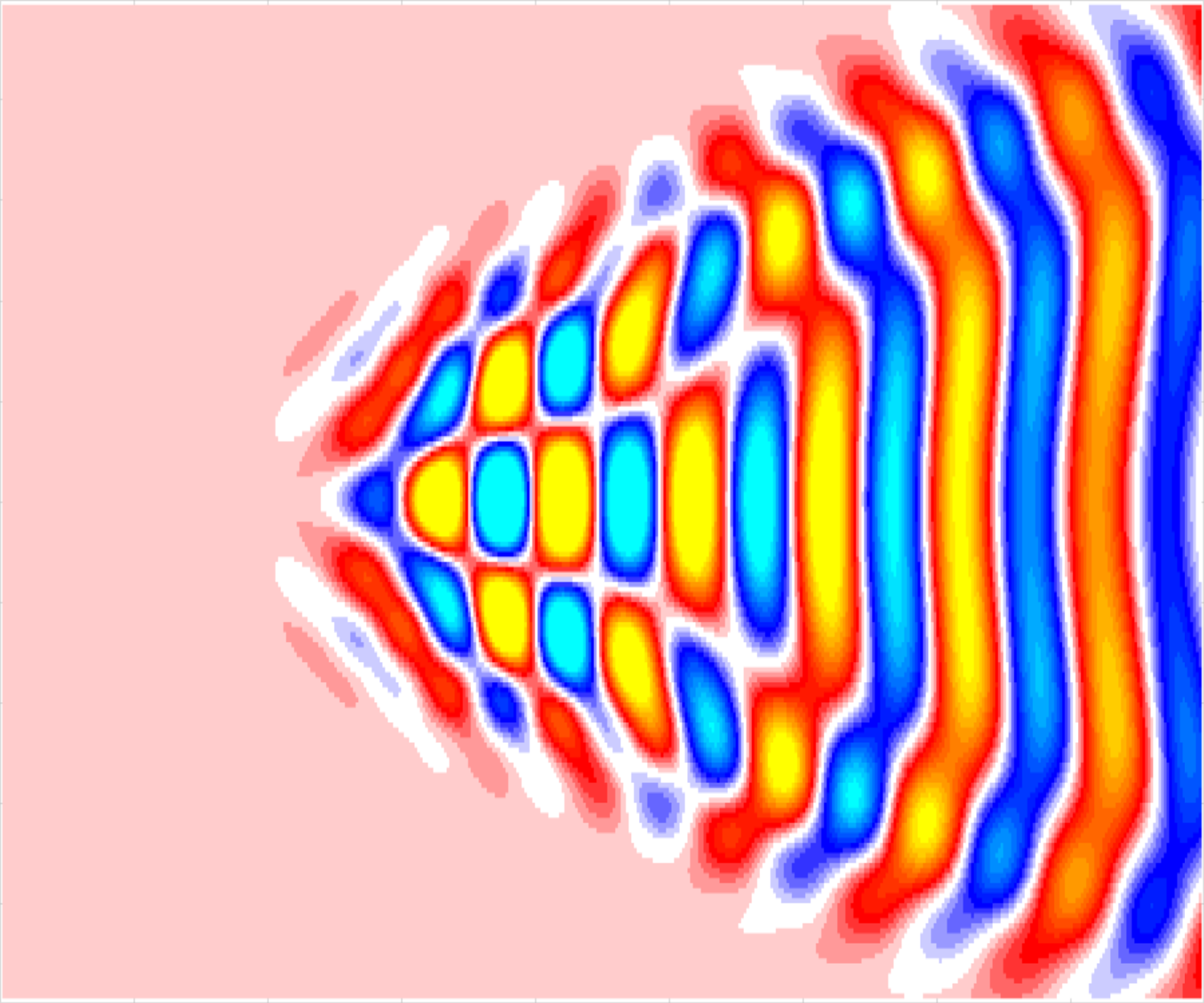}}
    \end{minipage}
    
    
    
    
    \caption{This figure presents the real and imaginary parts of a selection of ten Liouville modes determined by the continuous time DMD method given in the present manuscript corresponding to the vorticity of a flow across a cylinder (data available in \cite{kutz2016dynamic}). 
    }
    \label{fig:liouville-modes}
\end{figure}

\begin{figure}
    \begin{minipage}{0.25\textwidth}
        \centering Real part
    \end{minipage}\hfill\begin{minipage}{0.25\textwidth}
        \centering Imaginary part
    \end{minipage}
    
    \begin{minipage}{1.0\textwidth}
        \raisebox{-.5\height}{\includegraphics[width=0.25\textwidth]{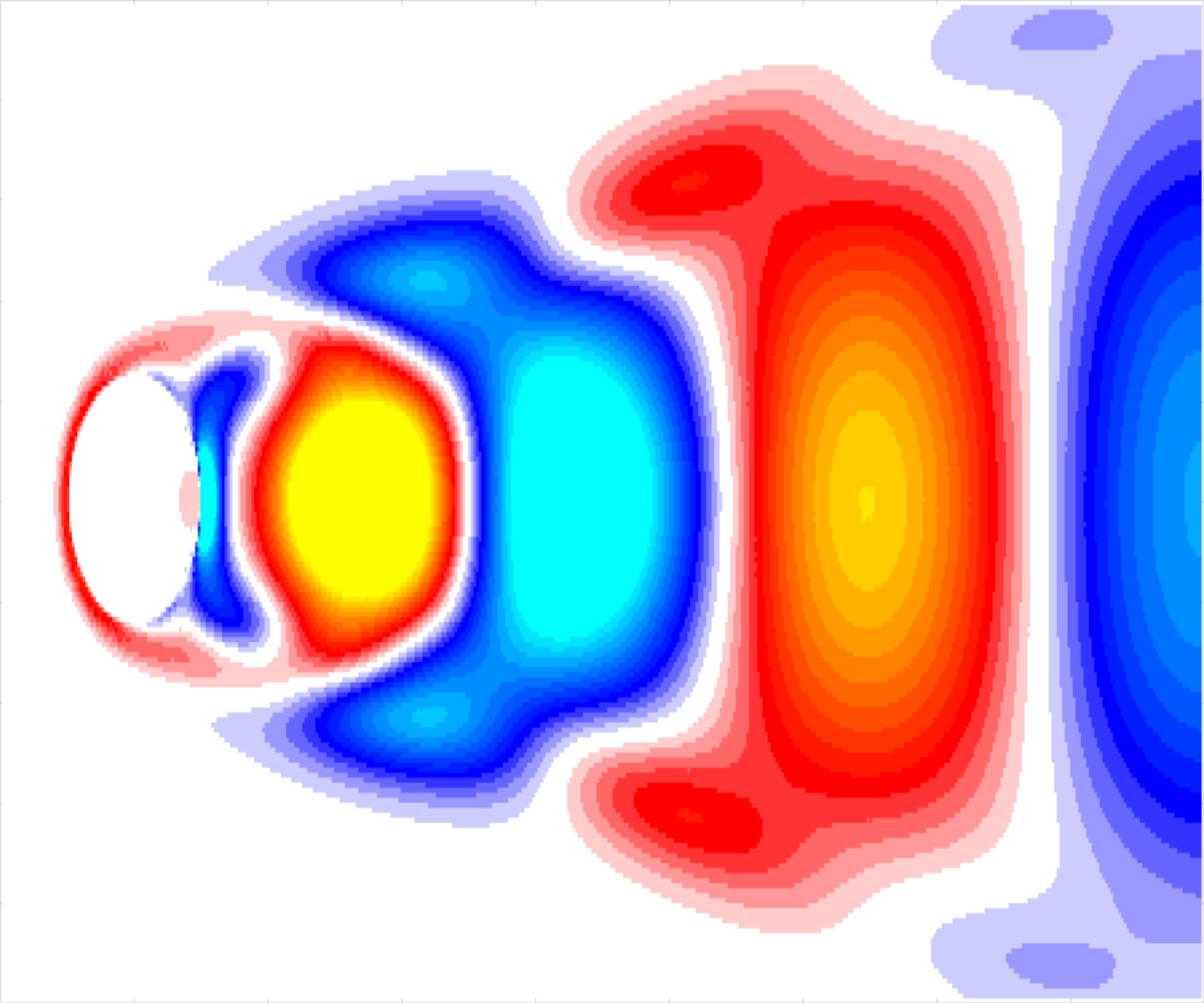}}
        \hfill
        Mode 142
        \hfill
        \raisebox{-.5\height}{\includegraphics[width=0.25\textwidth]{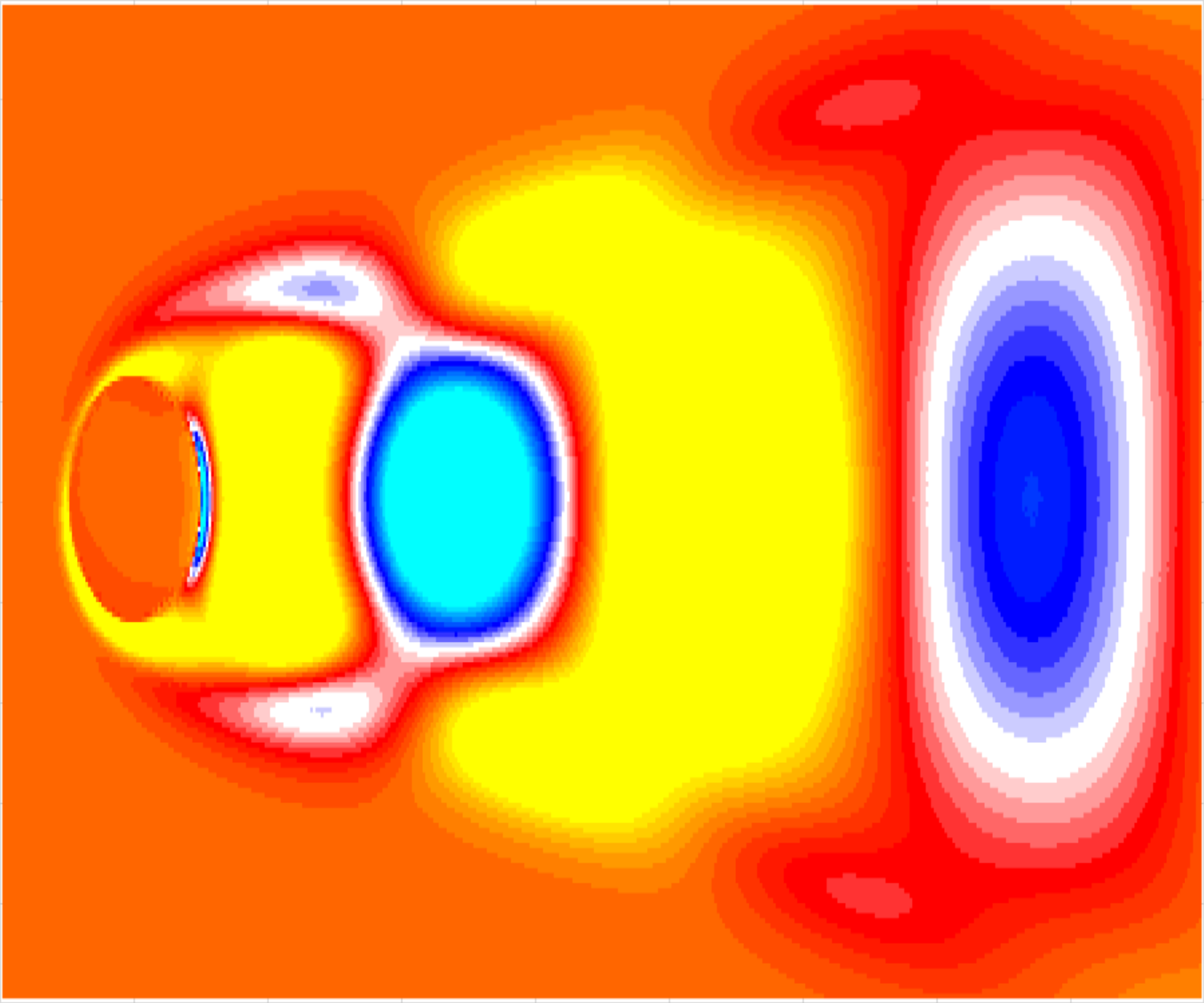}}
    \end{minipage}
    
    \begin{minipage}{1.0\textwidth}
        \raisebox{-.5\height}{\includegraphics[width=0.25\textwidth]{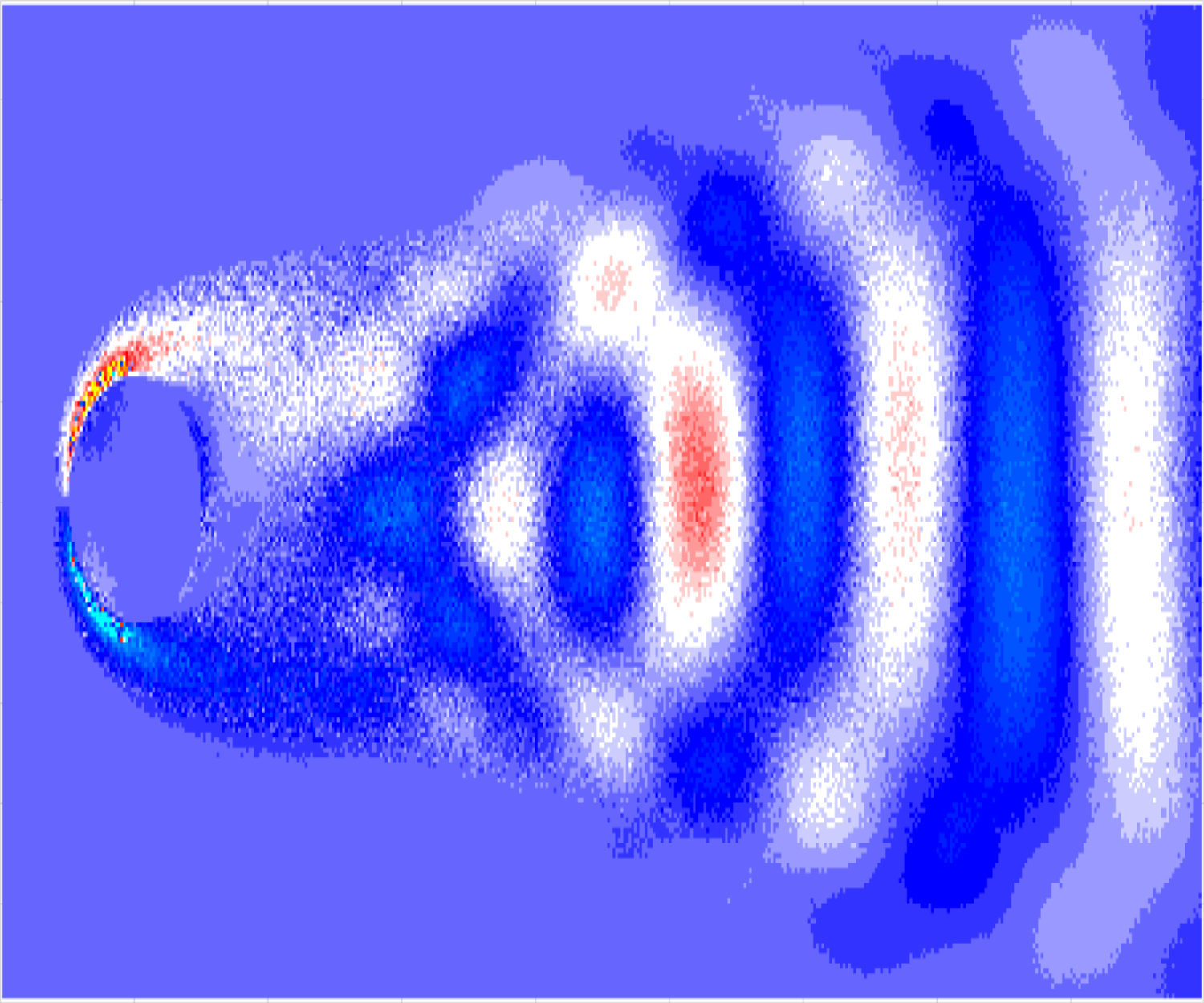}}
        \hfill
        Mode 120
        \hfill
        \raisebox{-.5\height}{\includegraphics[width=0.25\textwidth]{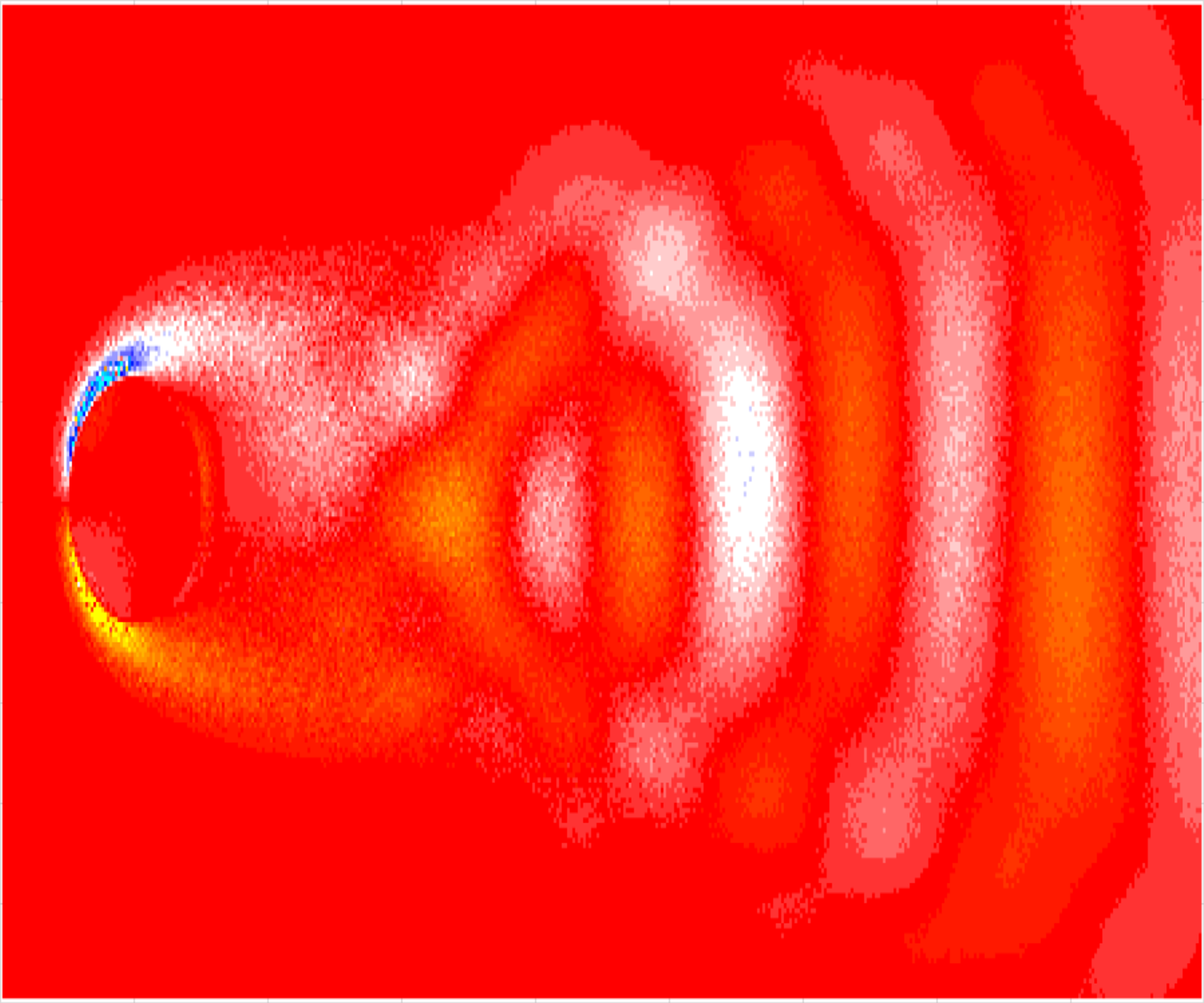}}
    \end{minipage}
    
    \begin{minipage}{1.0\textwidth}
        \raisebox{-.5\height}{\includegraphics[width=0.25\textwidth]{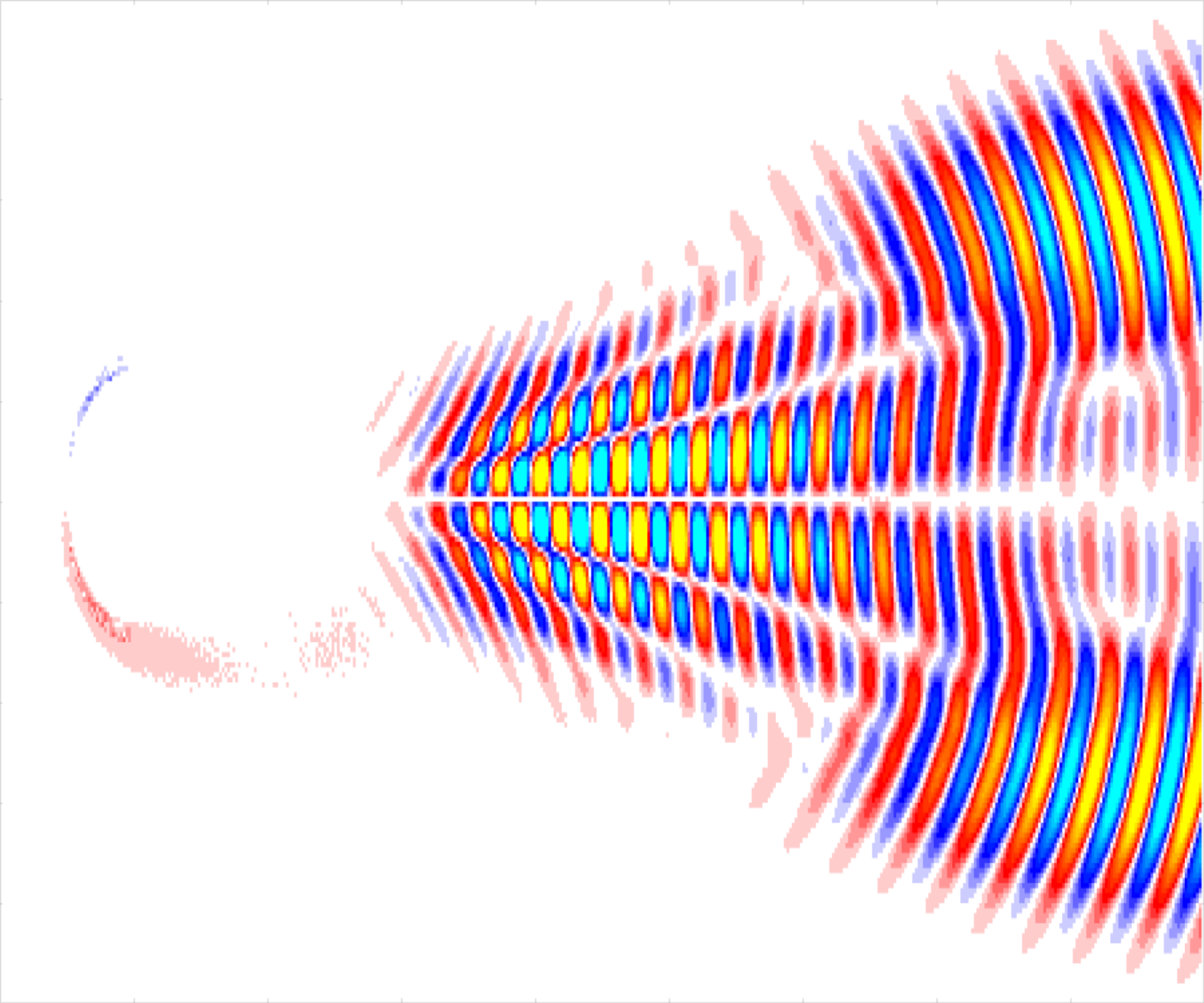}}
        \hfill
        Mode 103
        \hfill
        \raisebox{-.5\height}{\includegraphics[width=0.25\textwidth]{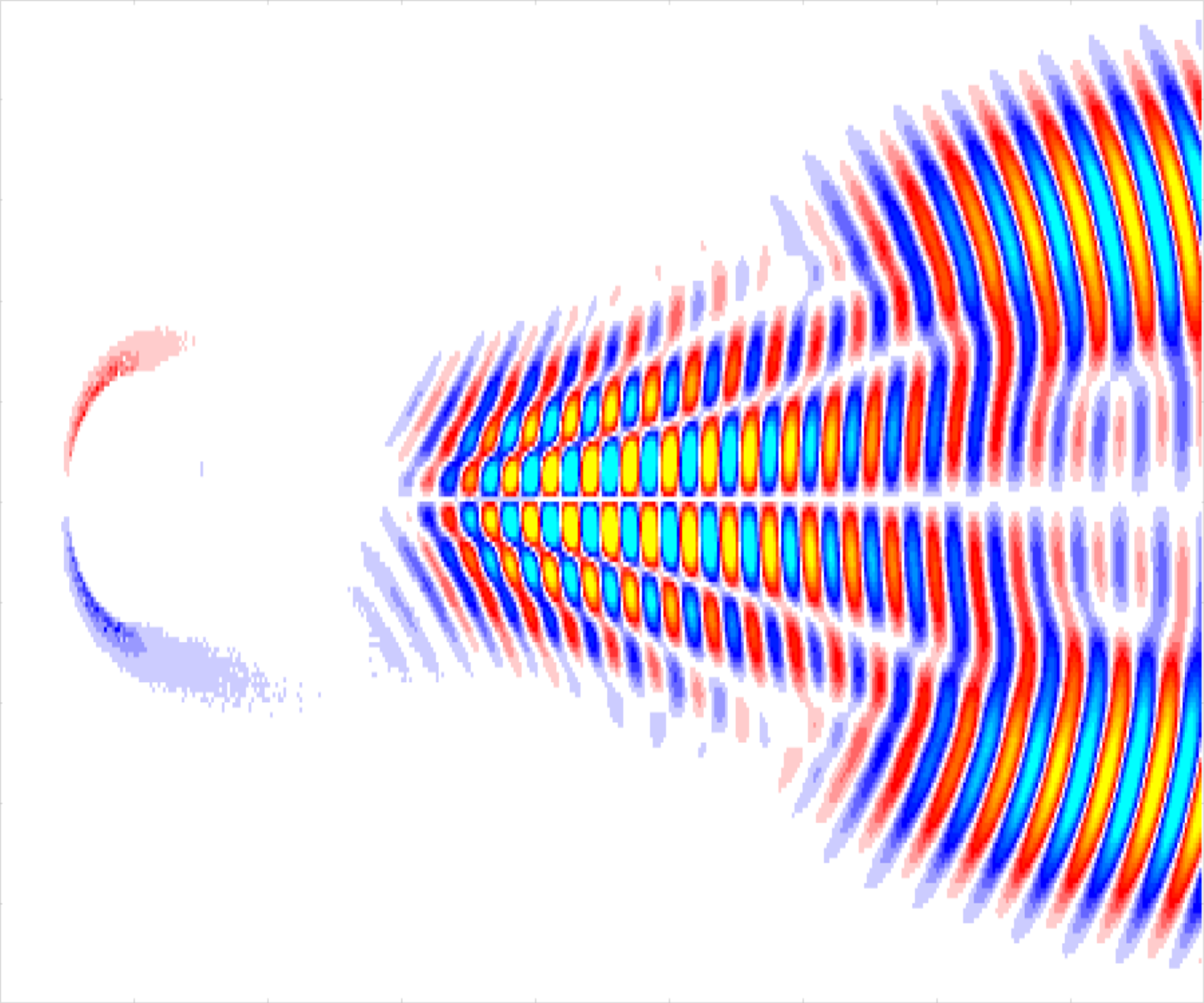}}
    \end{minipage}
    
    \begin{minipage}{1.0\textwidth}
        \raisebox{-.5\height}{\includegraphics[width=0.25\textwidth]{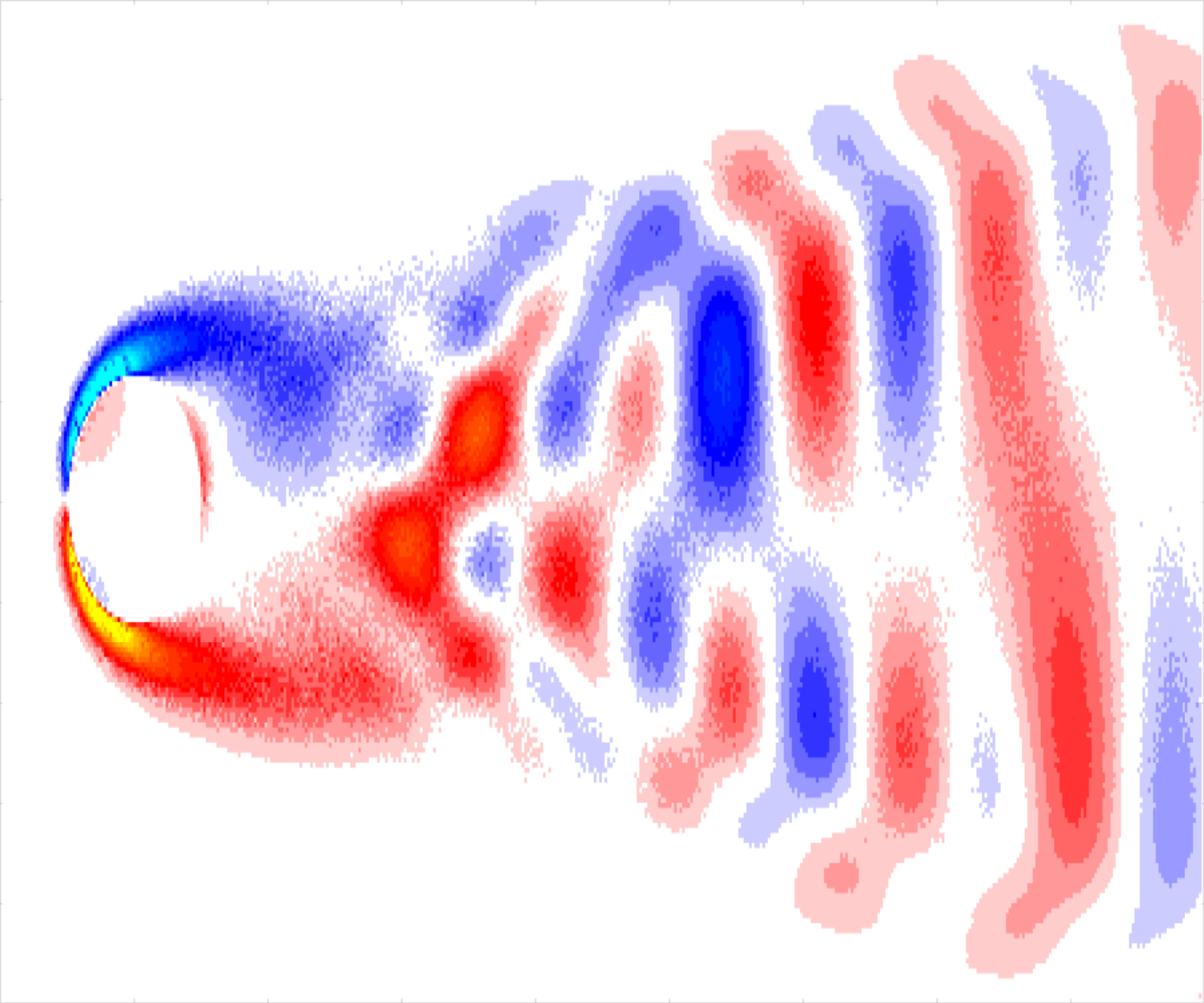}}
        \hfill
        Mode 94
        \hfill
        \raisebox{-.5\height}{\includegraphics[width=0.25\textwidth]{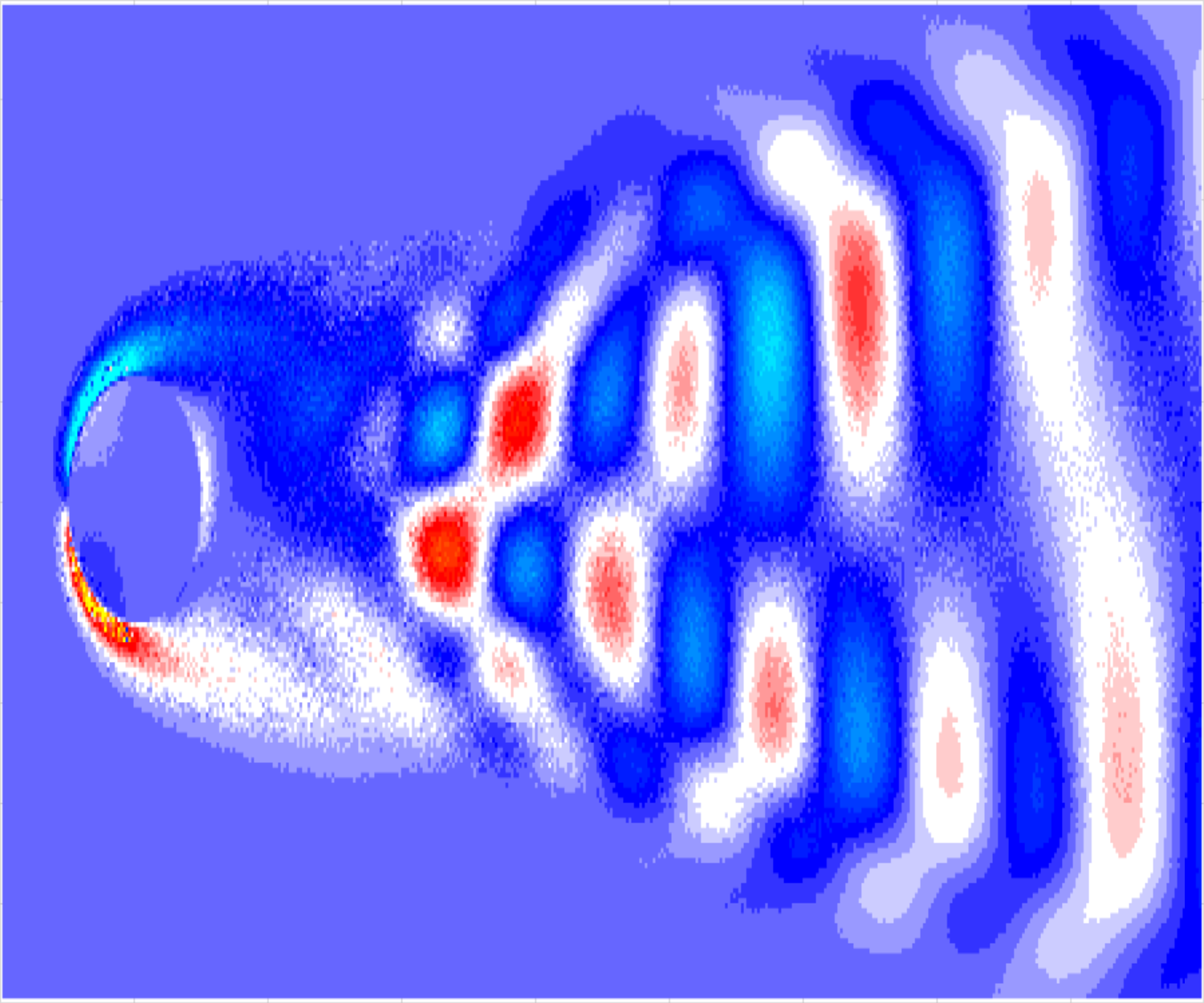}}
    \end{minipage}
    
    \begin{minipage}{1.0\textwidth}
        \raisebox{-.5\height}{\includegraphics[width=0.25\textwidth]{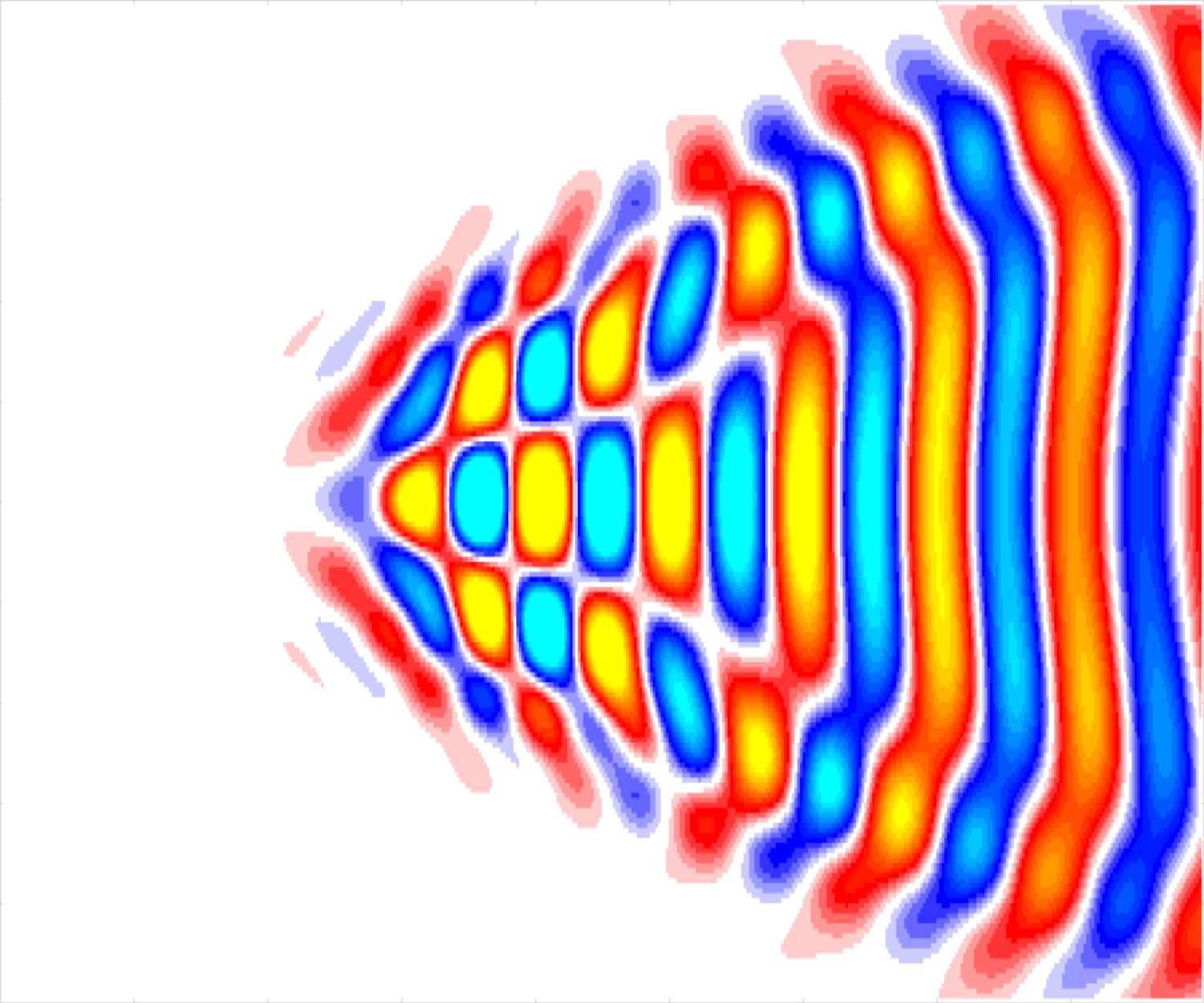}}
        \hfill
        Mode 80
        \hfill
        \raisebox{-.5\height}{\includegraphics[width=0.25\textwidth]{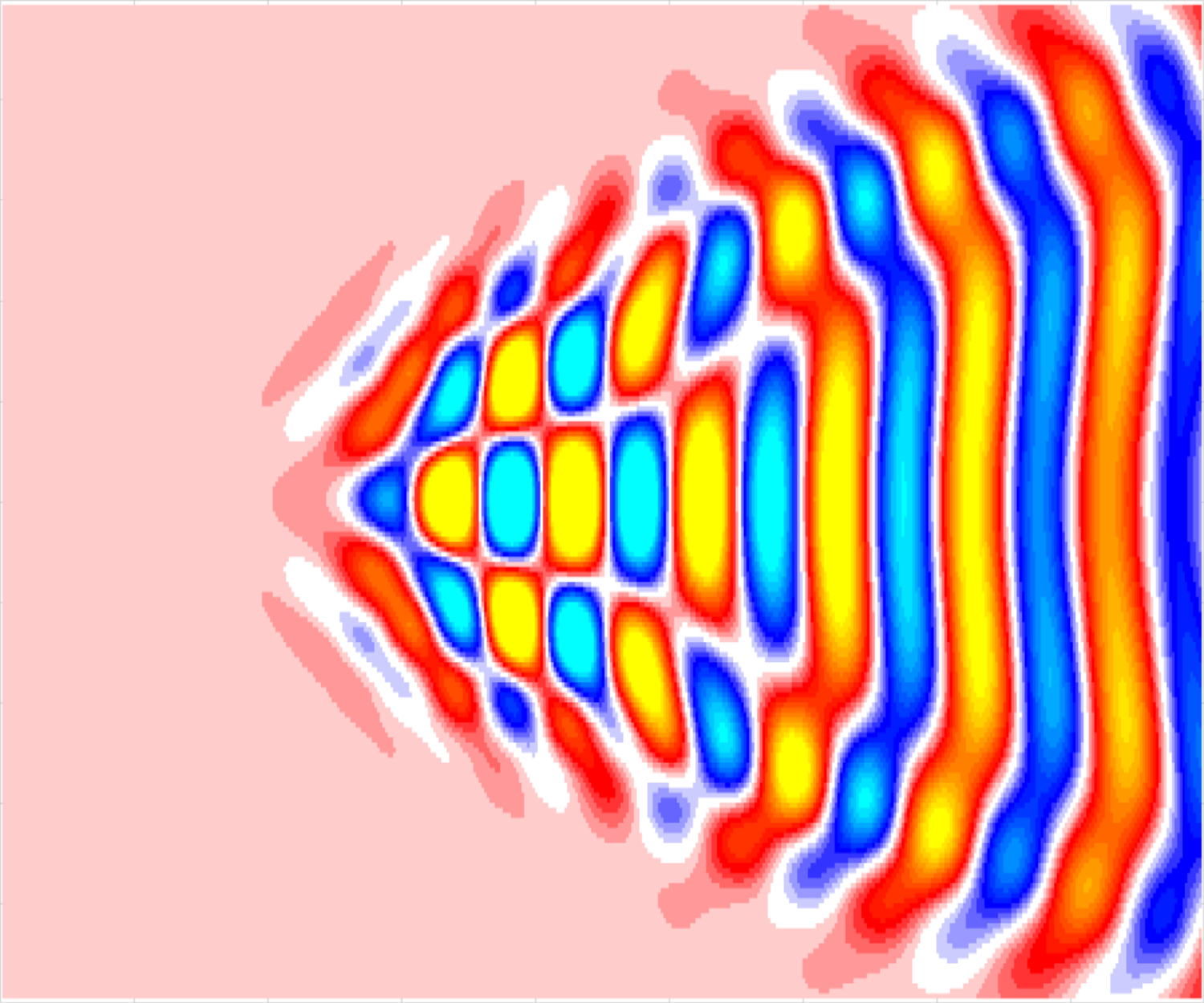}}
    \end{minipage}
	This figure presents the real and imaginary parts of a selection of five scaled Liouville DMD modes for the cylinder wake vorticity data in \cite{kutz2016dynamic}. The difference between these modes and the modes in Figure \ref{fig:liouville-modes} was anticipated for several reasons; the selection of $a=0.99$ is expected to result in slightly different modes, and there is no consistent method of ordering the Liouville modes as the significance of each mode depends not only on its magnitude, but also the associated eigenvector and initial value. 
    \label{fig:scaled-liouville-modes}
\end{figure}
\begin{figure}
    \centering
    \begin{minipage}{0.28\textwidth}
        \centering
        Snapshot of the true flow at
        
        \,
        \includegraphics[width=\textwidth]{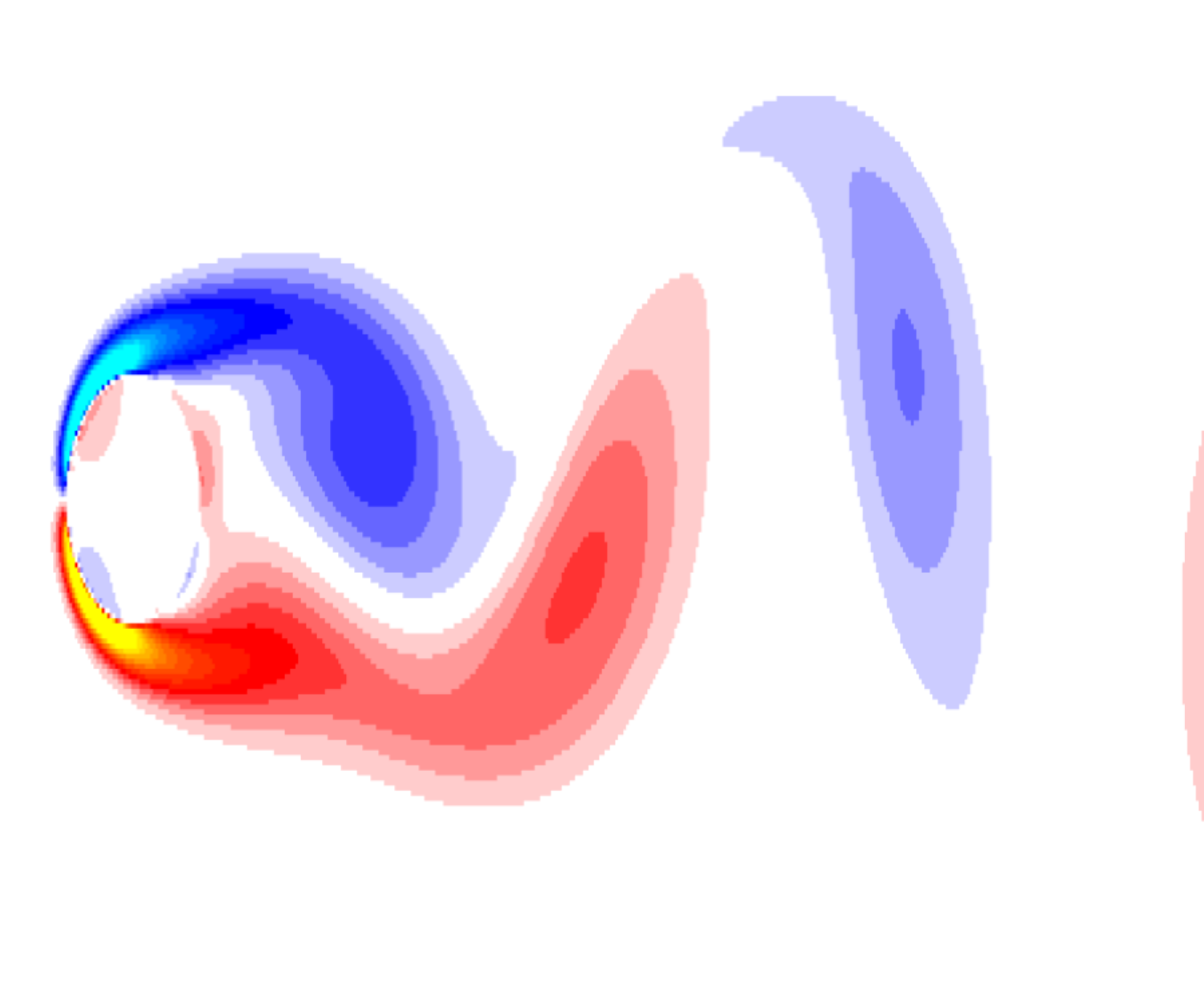}
        
        \,
        \includegraphics[width=\textwidth]{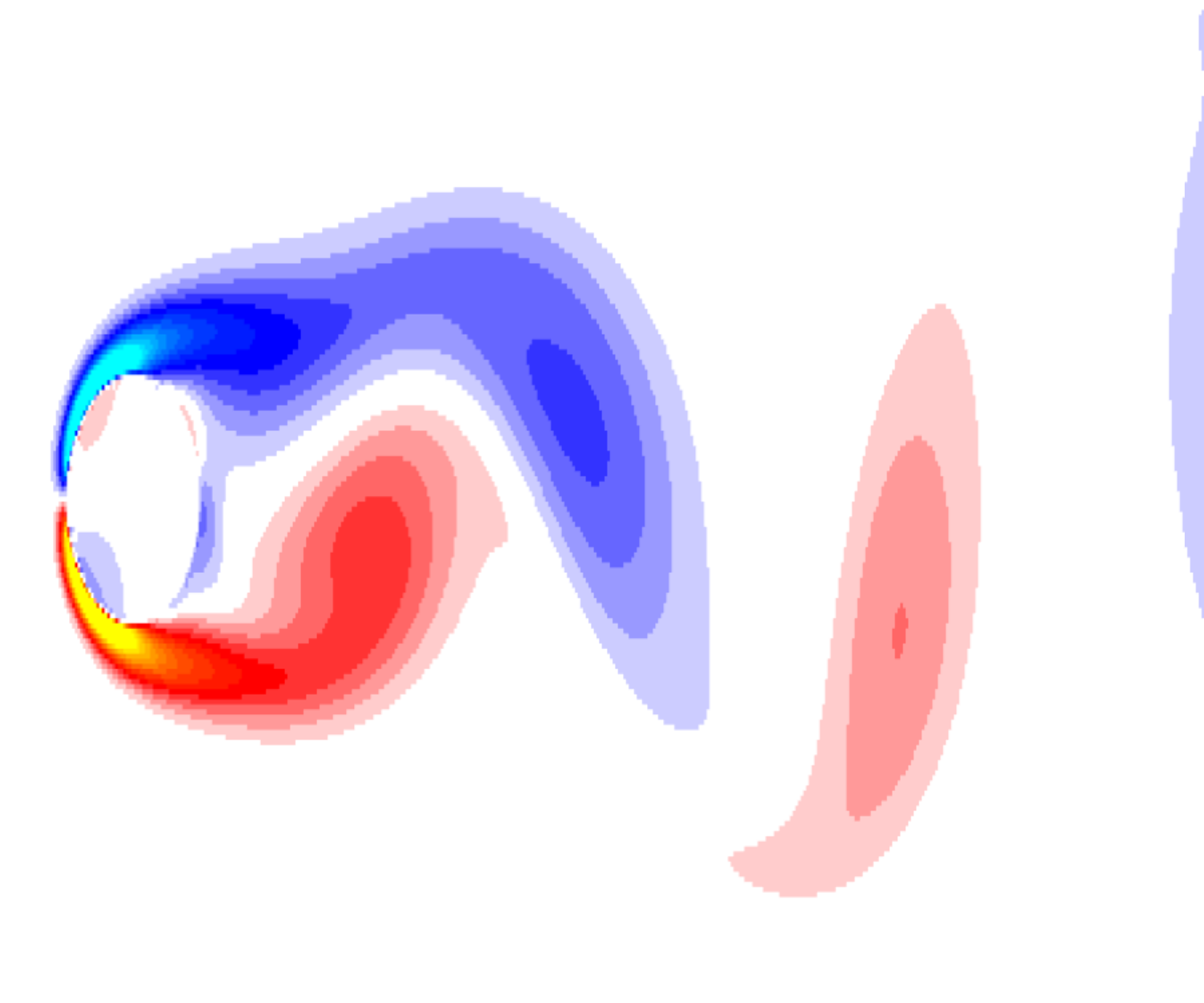}
        
        \,
        \includegraphics[width=\textwidth]{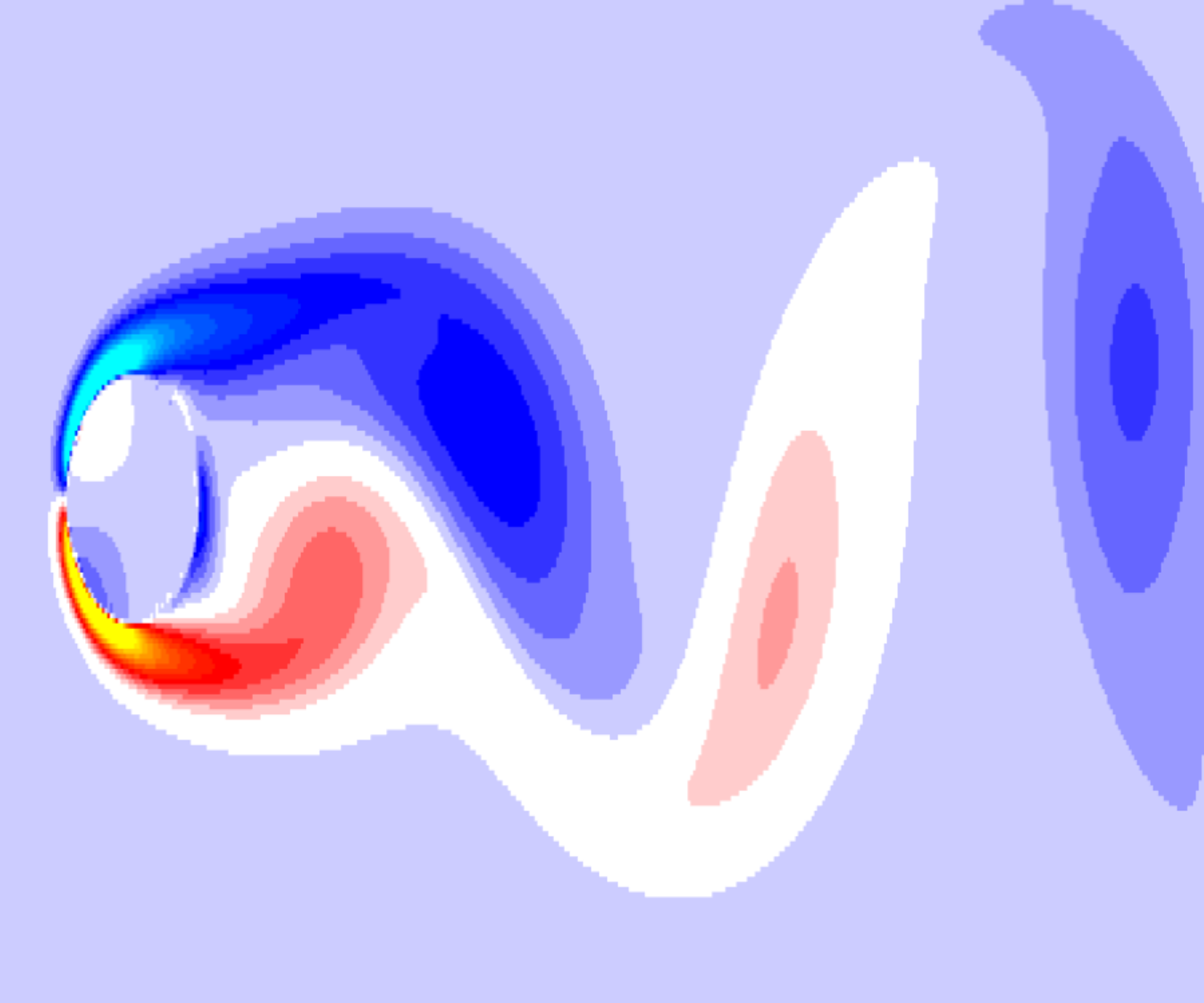}
        
        \,
        \includegraphics[width=\textwidth]{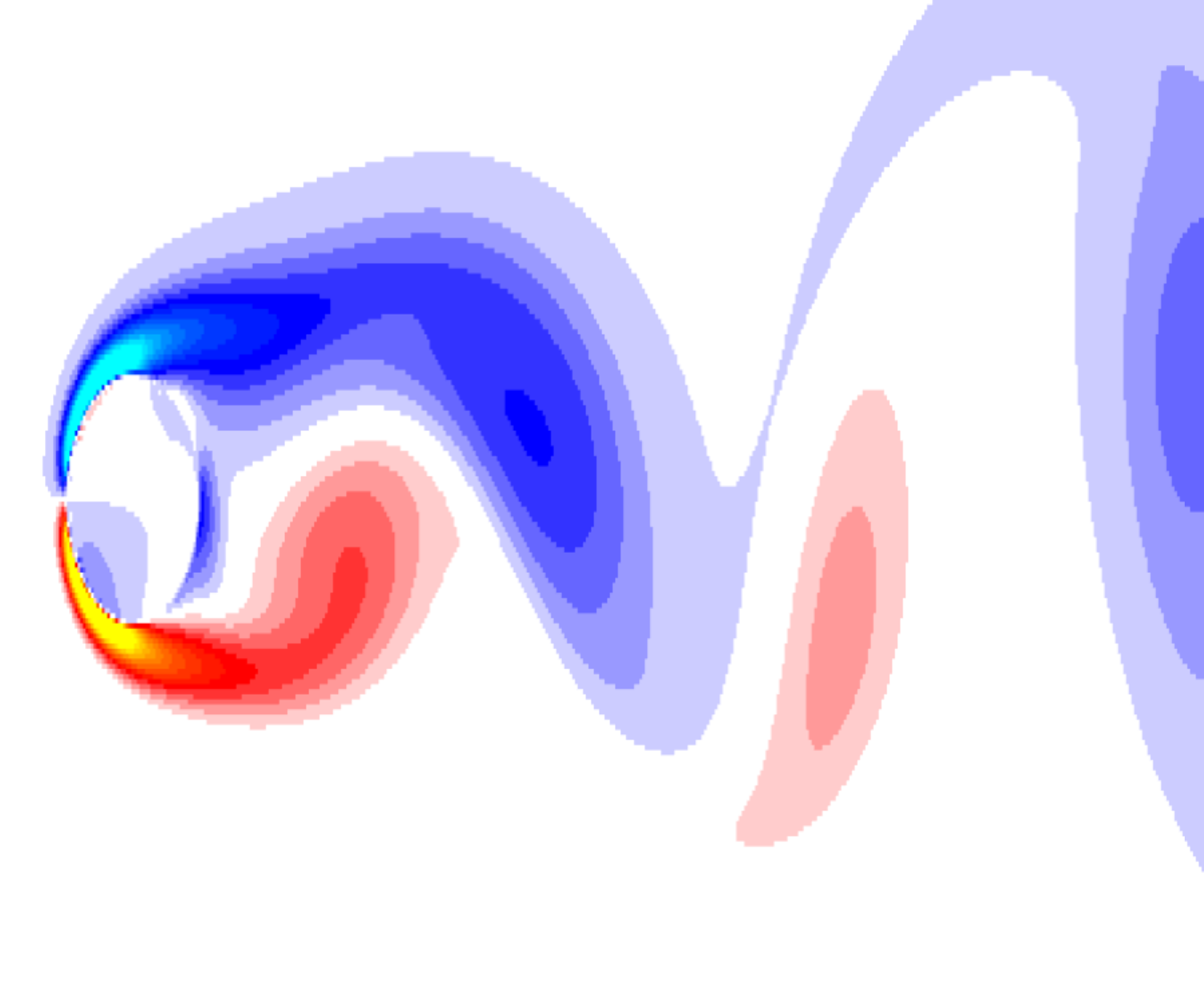}
        
        \,
        \includegraphics[width=\textwidth]{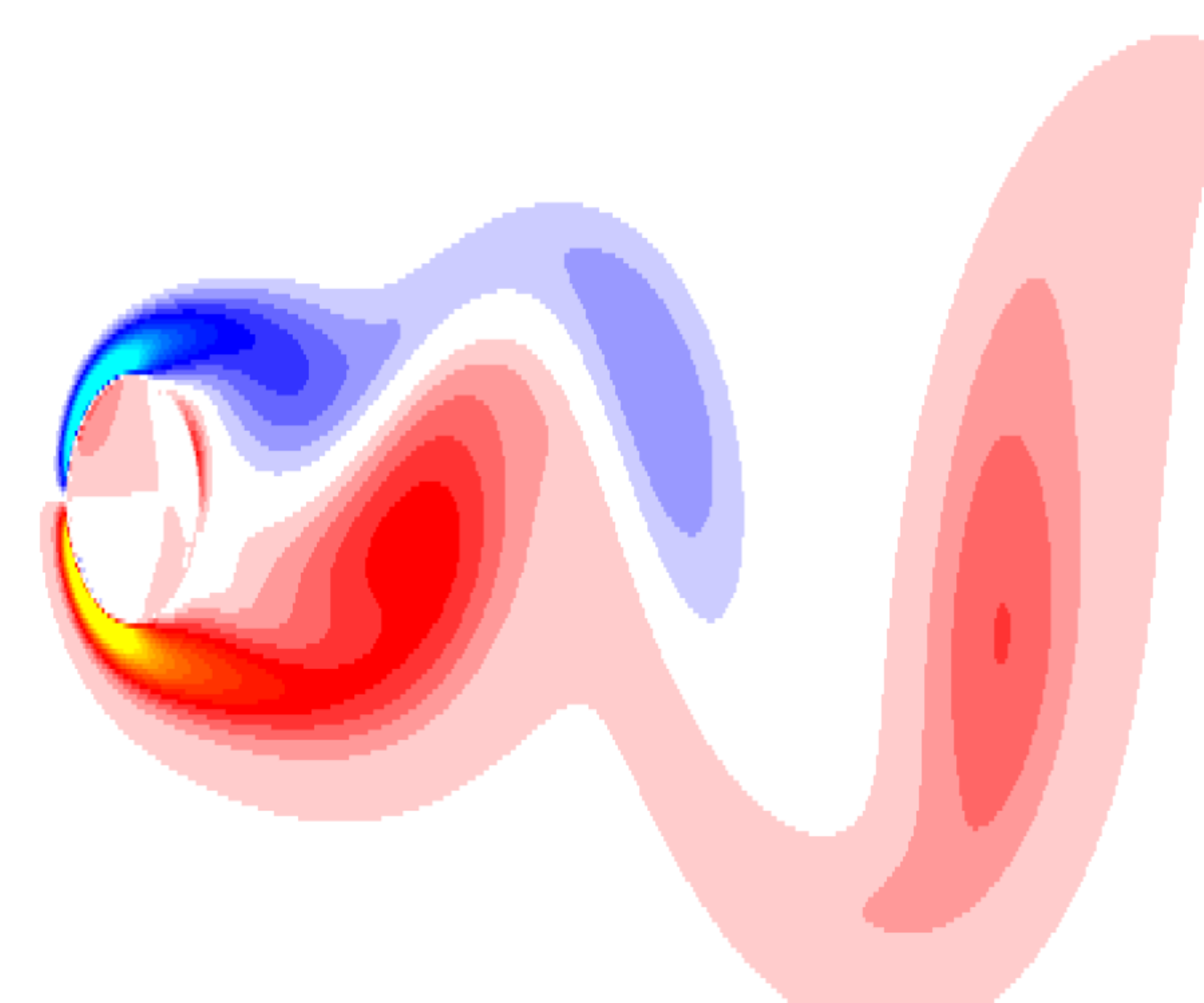}
    \end{minipage}
    \begin{minipage}{0.28\textwidth}
        \centering
        Reconstruction ($a=1$) at
        
        $t=0.18$s
        \includegraphics[width=\textwidth]{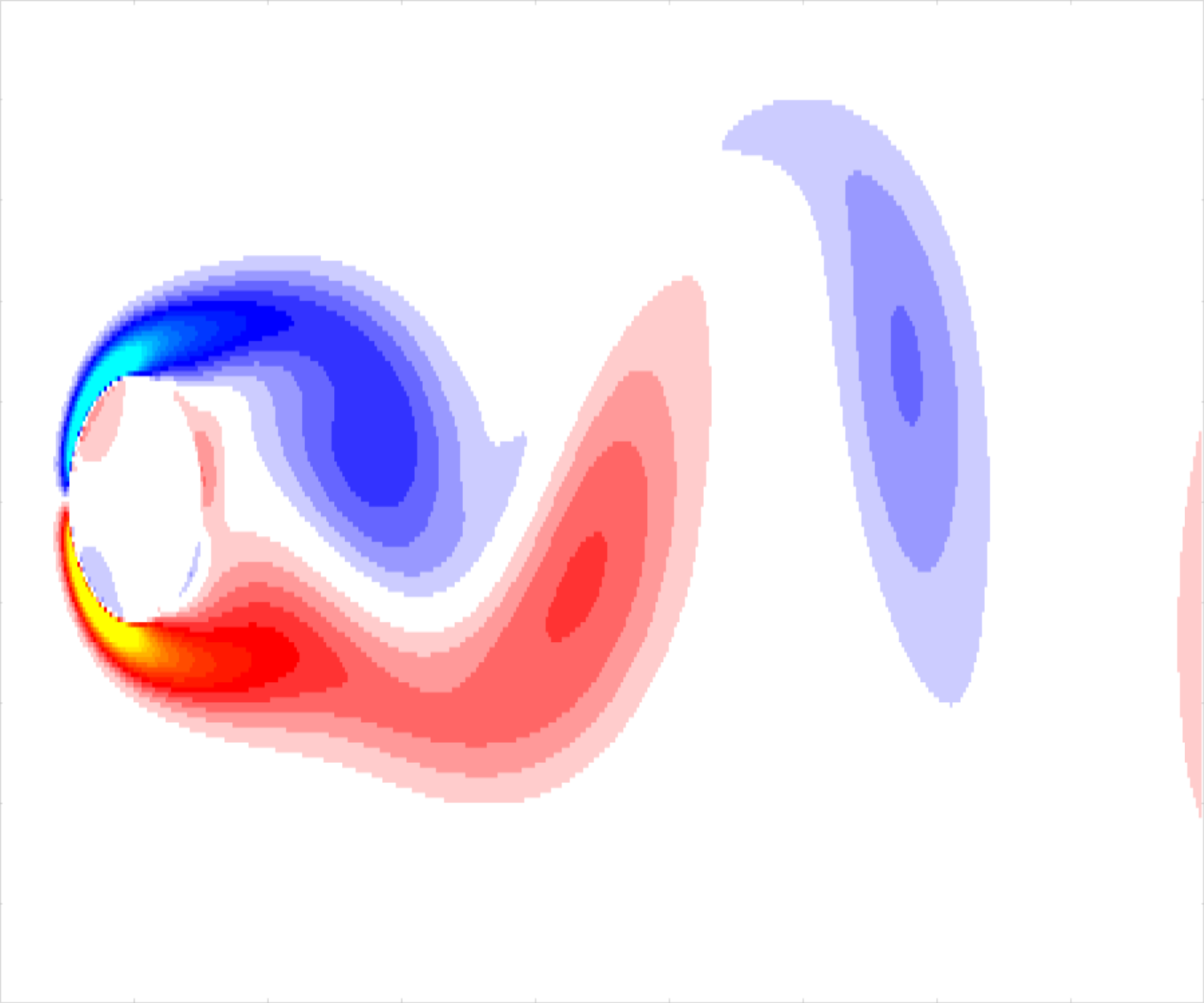}
        
        $t=1.08$s
        \includegraphics[width=\textwidth]{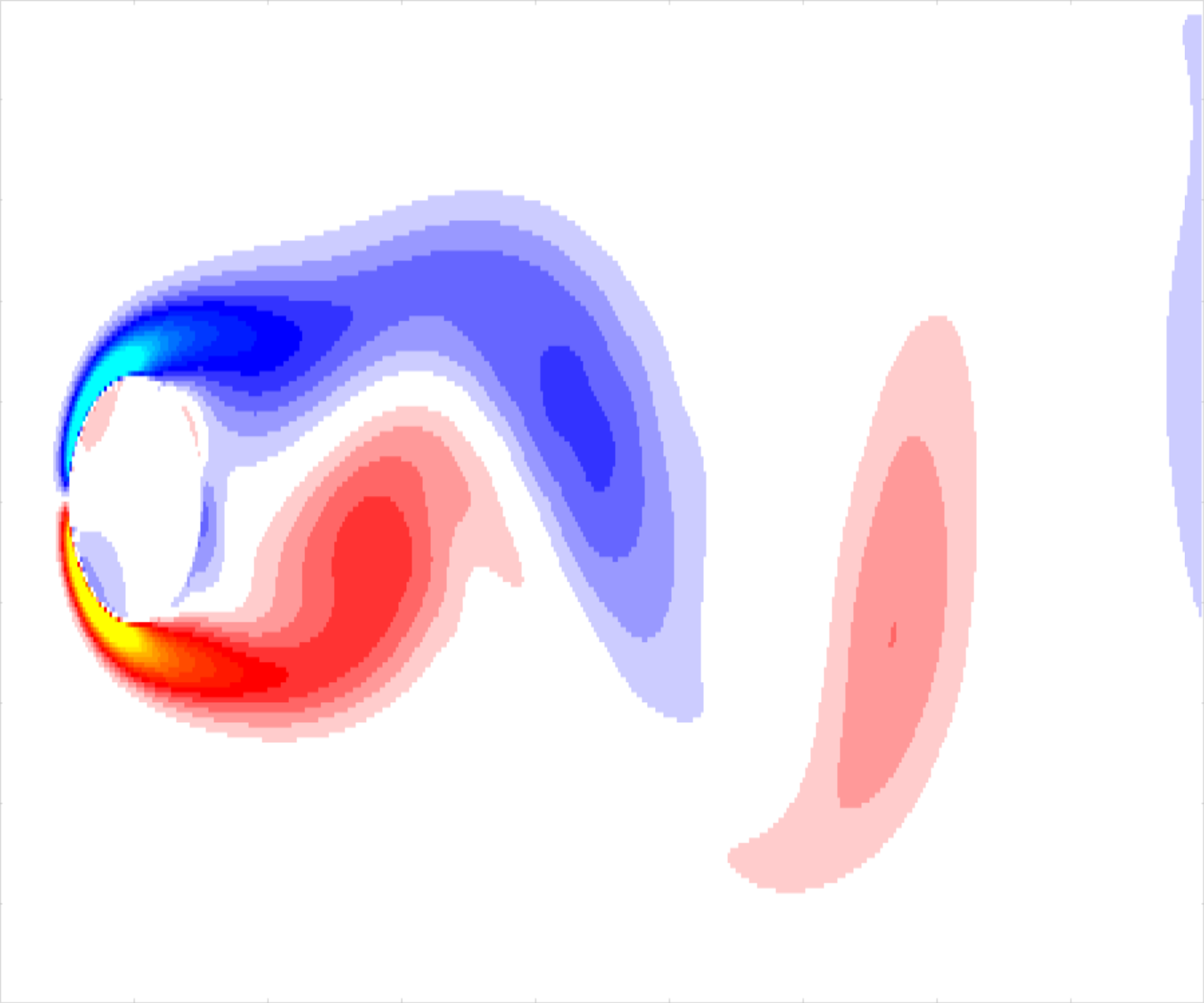}
        
        $t=1.58$s
        \includegraphics[width=\textwidth]{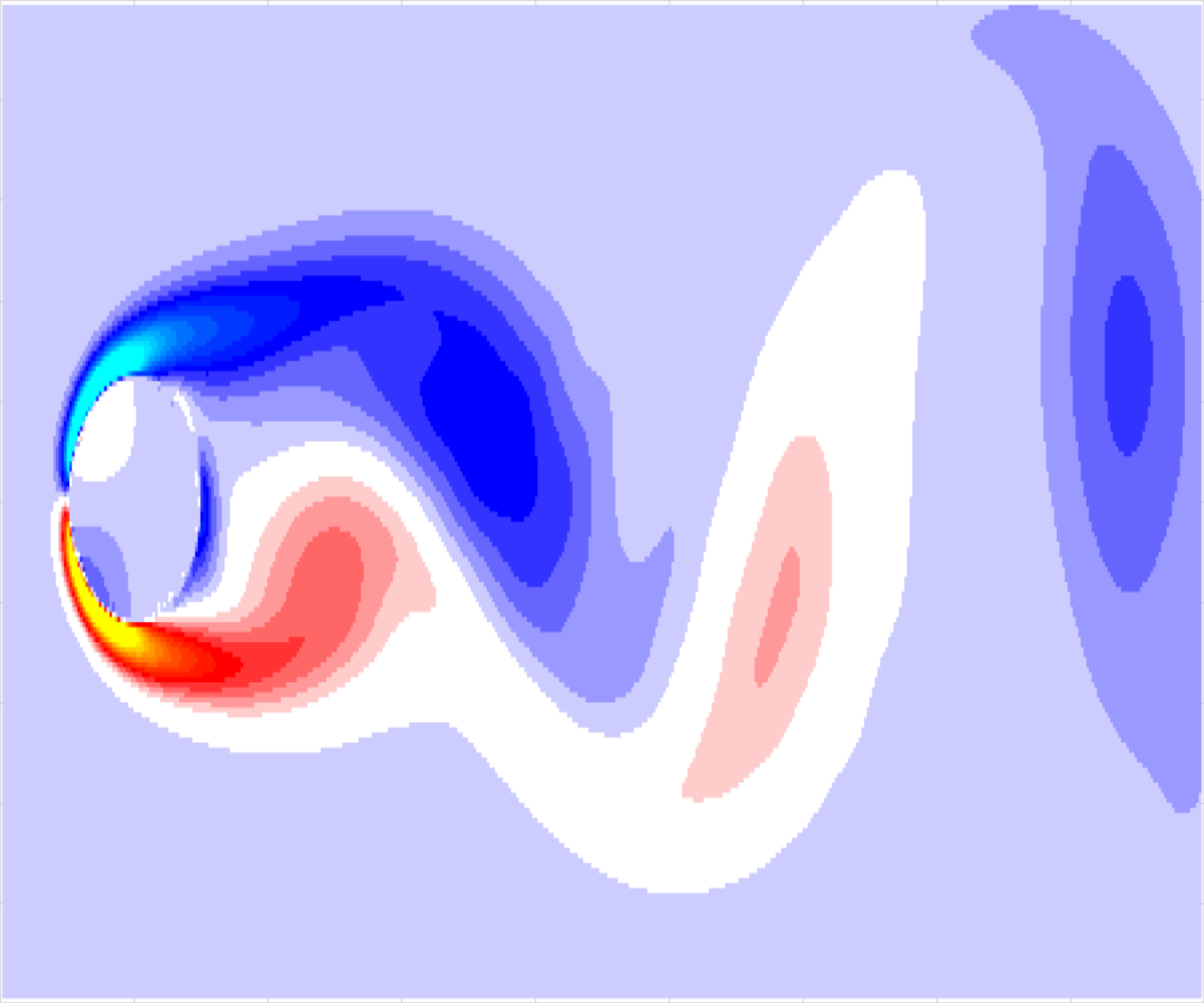}
        
        $t=2.24$s
        \includegraphics[width=\textwidth]{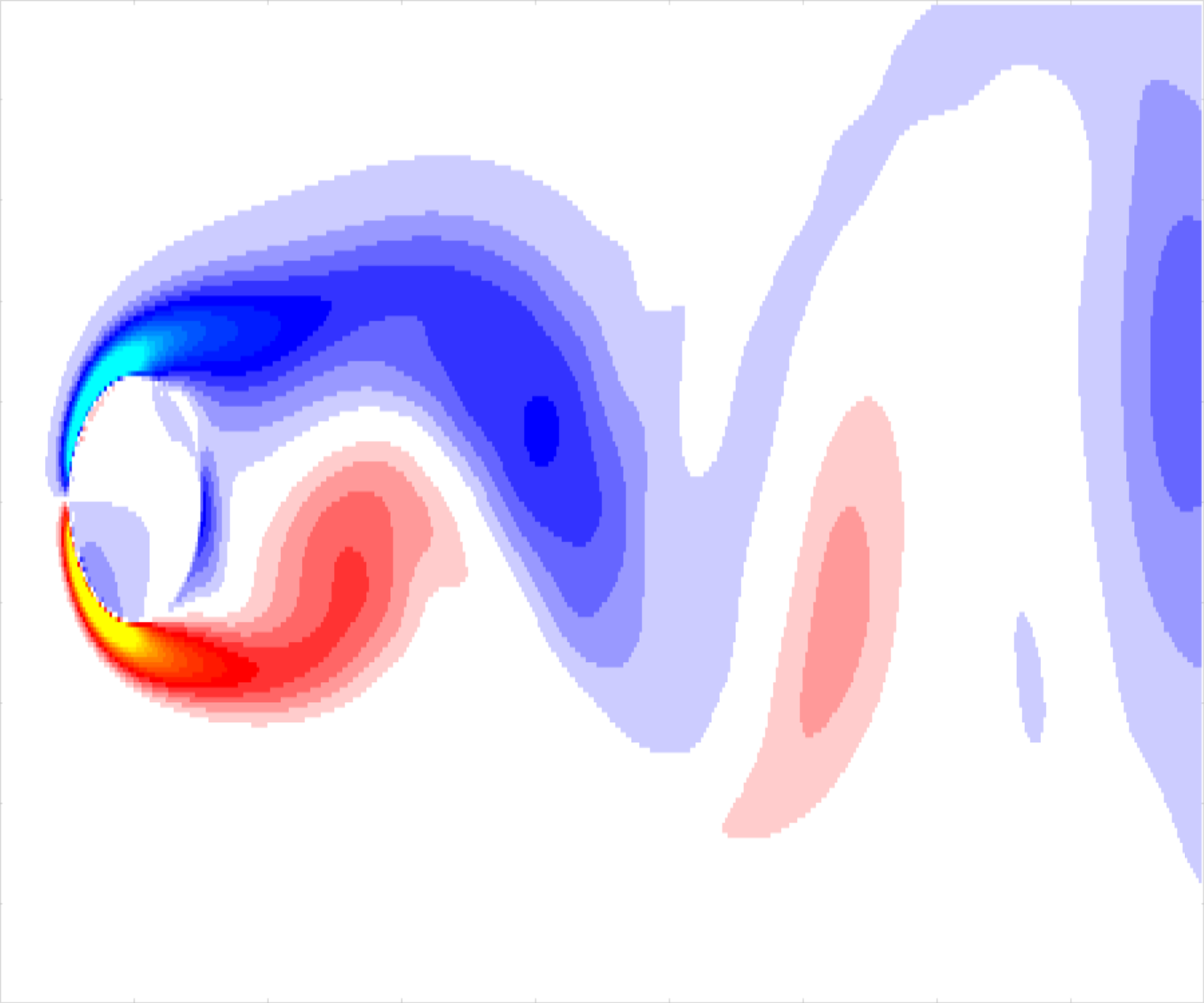}
        
        $t=2.98$s
        \includegraphics[width=\textwidth]{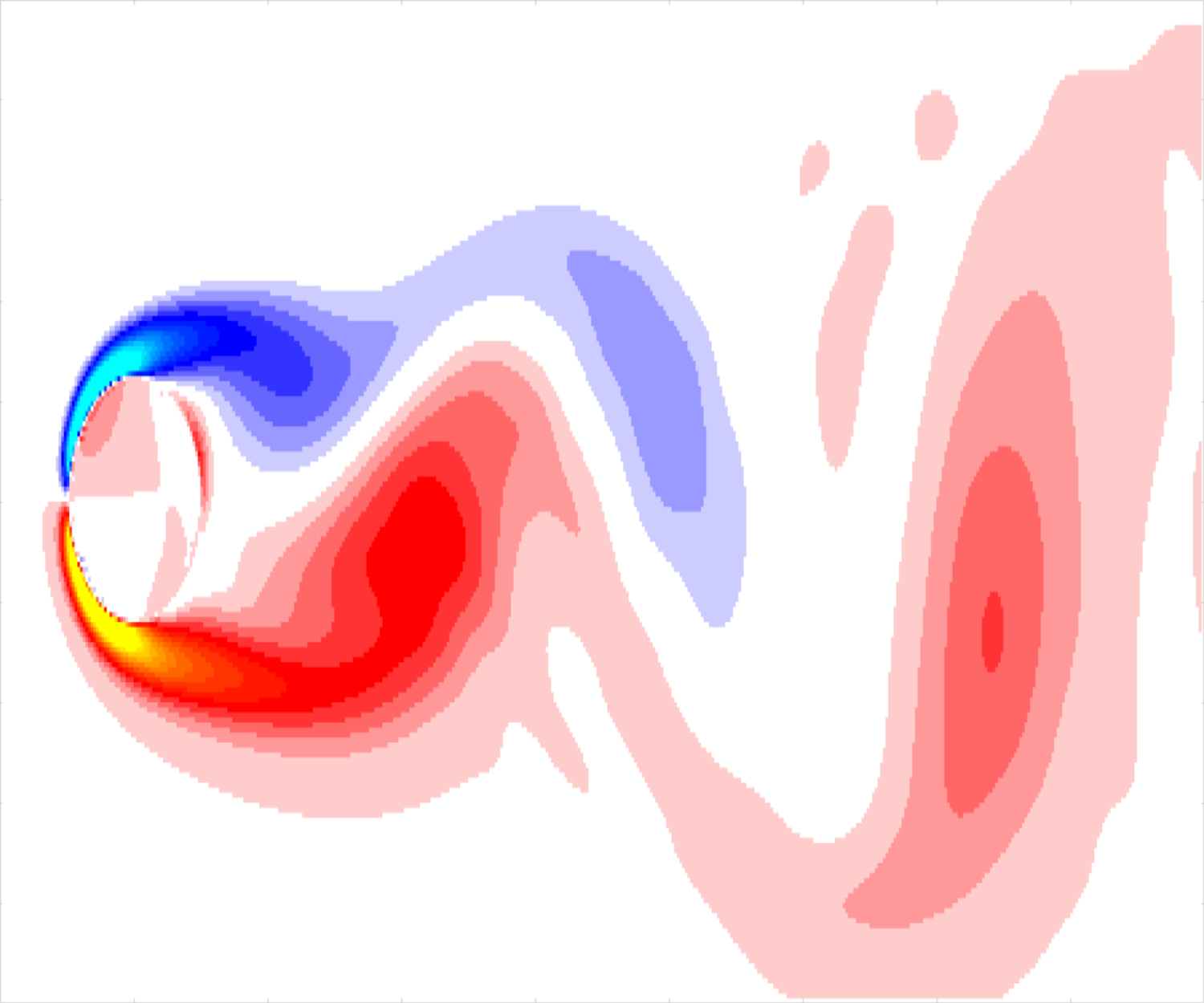}
    \end{minipage}
    \begin{minipage}{0.28\textwidth}
        \centering
        Reconstruction ($a=0.99$) at
        
        \,
        \includegraphics[width=\textwidth]{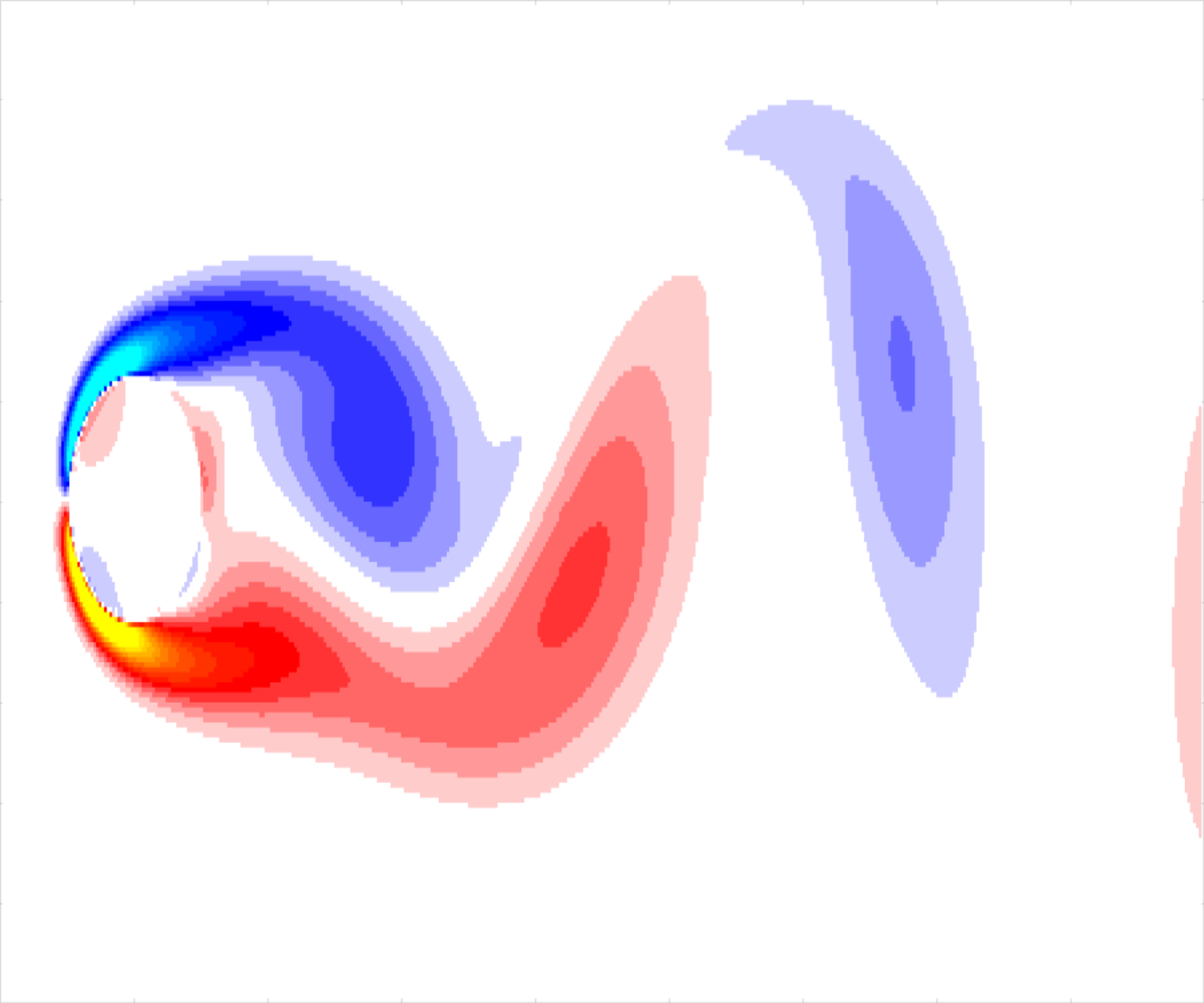}
        
        \,
        \includegraphics[width=\textwidth]{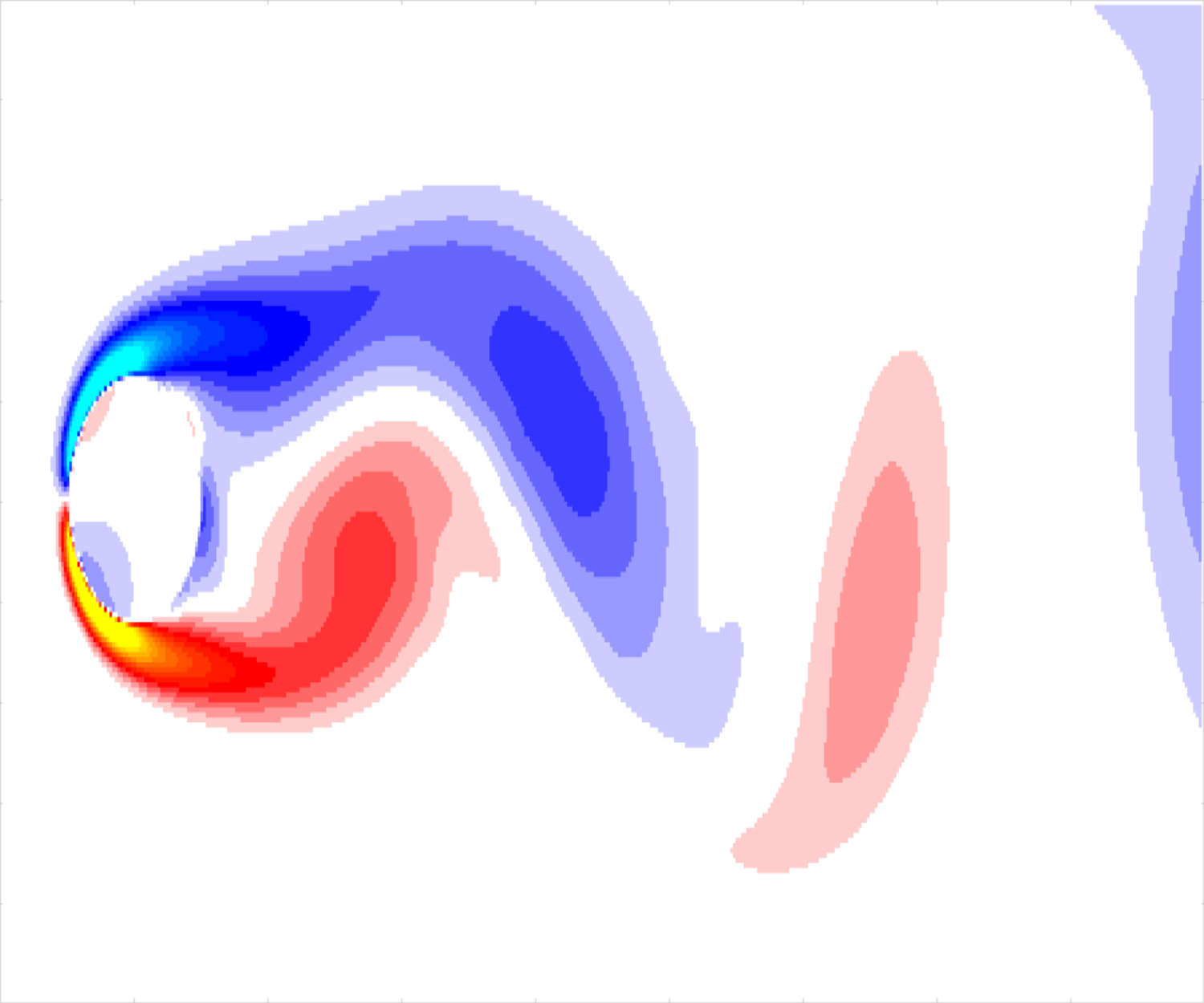}
        
        \,
        \includegraphics[width=\textwidth]{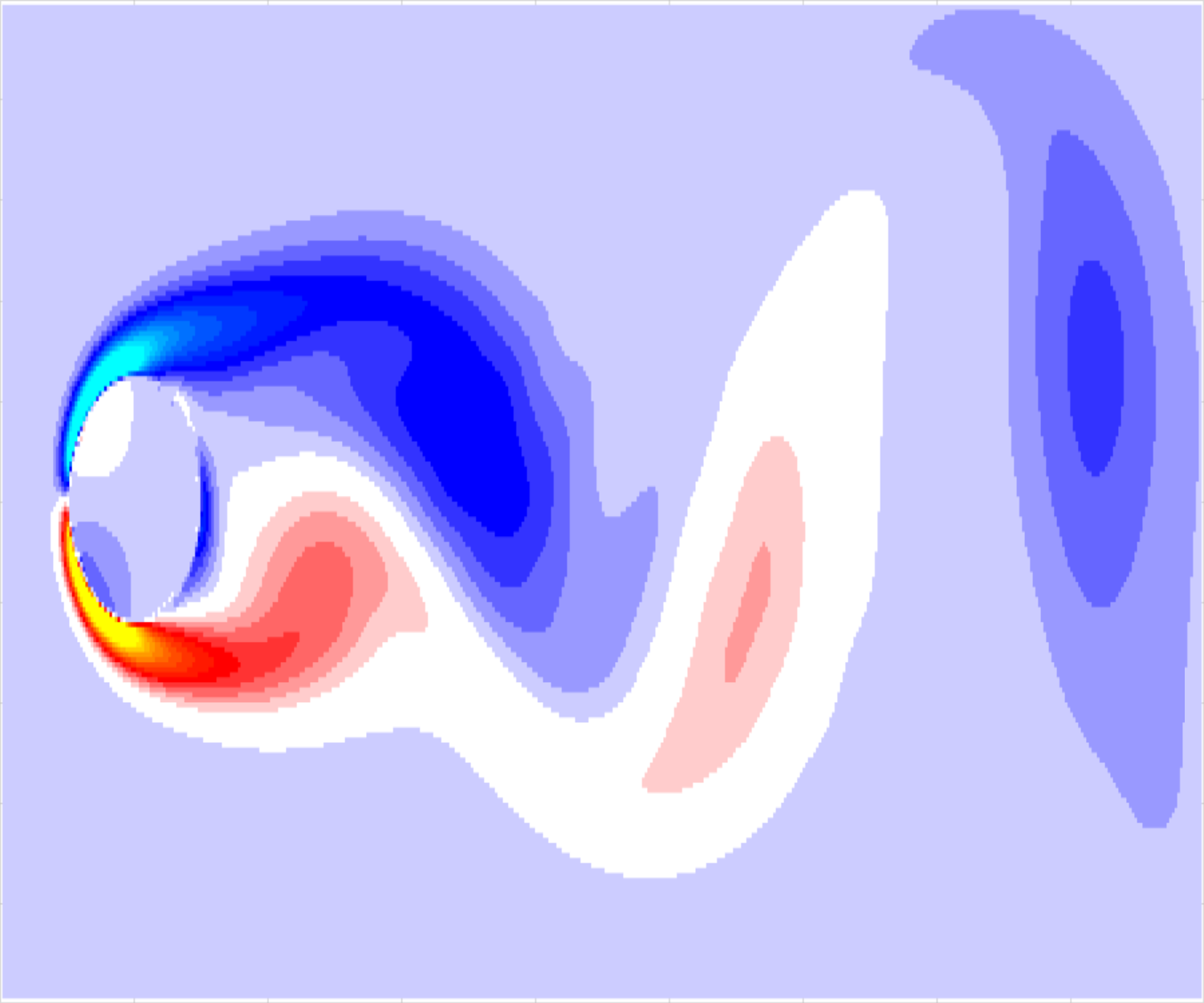}
        
        \,
        \includegraphics[width=\textwidth]{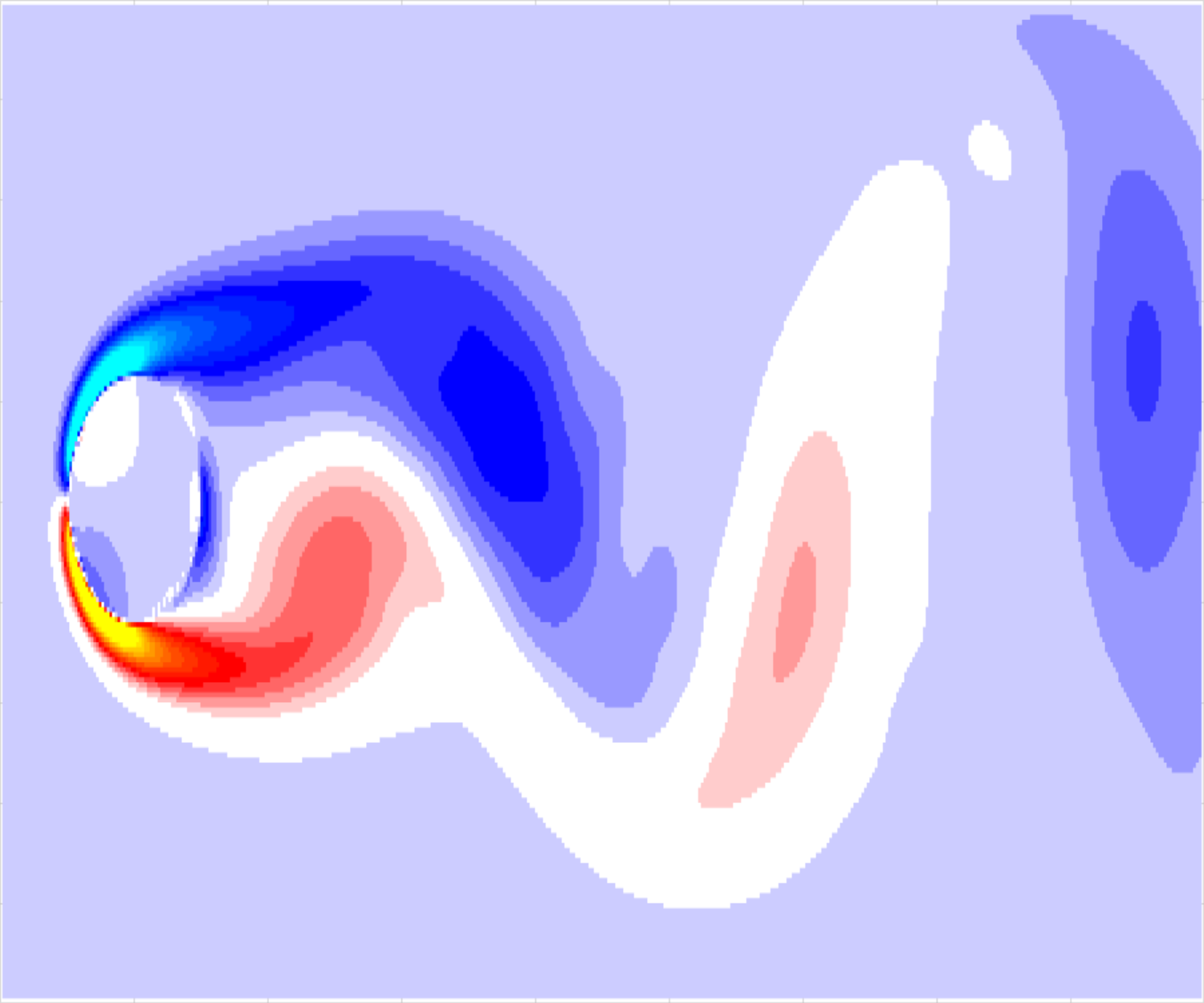}
        
        \,
        \includegraphics[width=\textwidth]{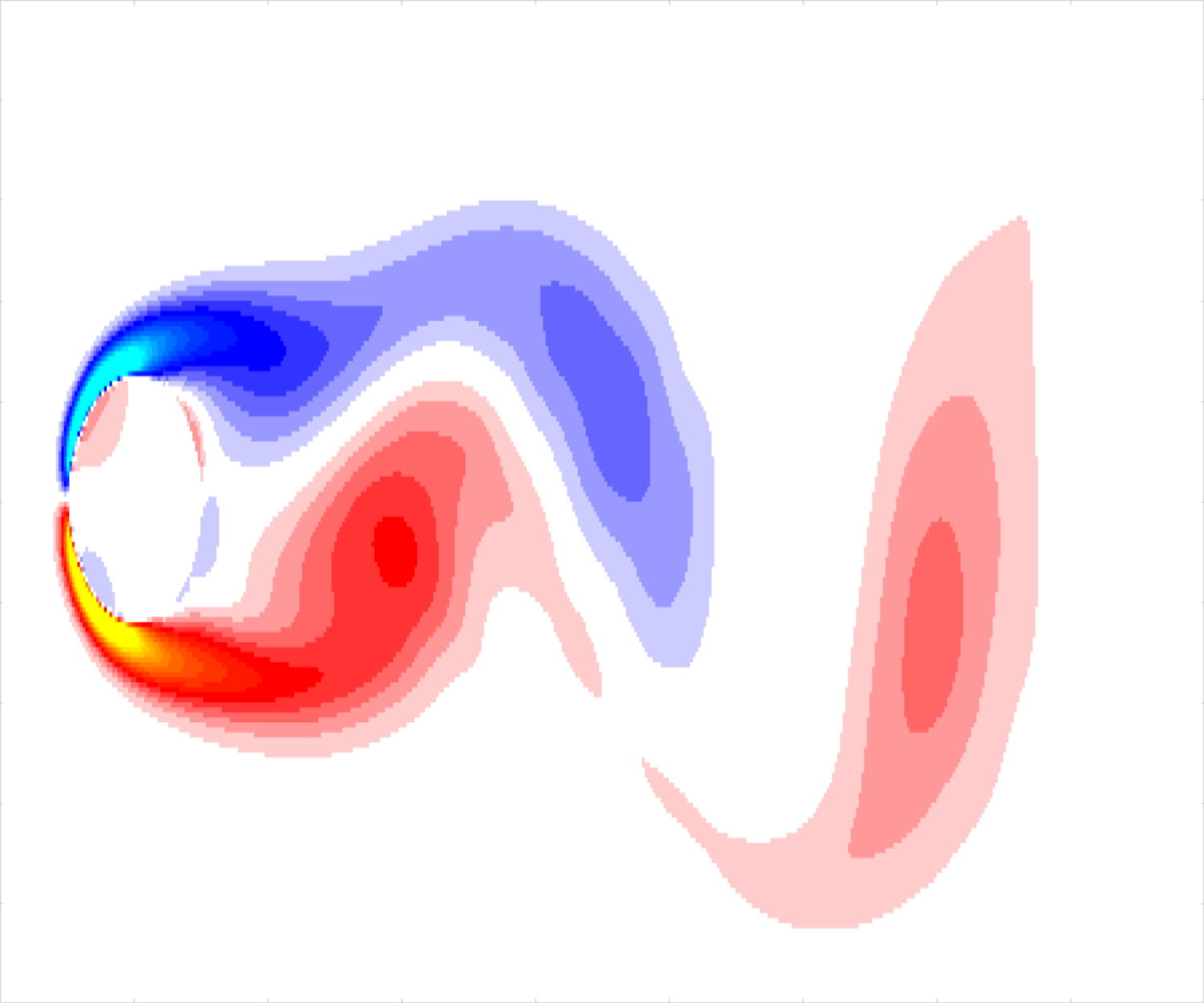}
    \end{minipage}
    \caption{Snapshots of the true flow compared with reconstruction via the Liouville DMD model in \eqref{eq:data-driven_model} and the scaled Liouville DMD model in \eqref{eq:scaled-data-driven_model} with $a=0.99$.}
    \label{fig:reconstruction}
\end{figure}

\subsection{SsVEP Dataset}

This experiment uses data from \cite{gruss2019sympathetic}. The data for this experiment was taken from an electroencephalography (EEG) recording of the visual cortex of one human participant during the active viewing of flickering images \cite{gruss2019sympathetic}. By modulating luminance or contrast of an image at a constant rate (e.g. 12Hz), image flickering reliably evokes the steady state visually evoked potential (ssVEP) in early visual cortex \cite{regan1989human,petro2017multimodal}, reflecting entrainment of neuronal oscillations at the same driving frequency. SsVEP in the current data was evoked by pattern-reversal Gabor patch flickering at 12Hz (i.e. contrast-modulated) for a trial length of $7$ seconds, with greatest signal strength originating from the occipital pole (Oz) of a 129-electrode cap. Data was sampled at $500$Hz, band-pass filtered online from $0.5$ – $48$Hz, offline from $3$ – $40$Hz, with $53$ trials retained for this individual after artifact rejection. Of these trials, the first $40$ trials were used in the continuous time DMD method and each trial was subdivided into $50$ trajectories. SsVEP data have the advantage of having an exceedingly high signal-to-noise ratio and high phase coherence due to the oscillatory nature of the signal, ideally suited for signal detection algorithms (such as brain-computer interfaces \cite{bakardjian2010optimization,bin2009online,middendorf2000brain}).

In this setting each independent trial can be used as a trajectory for a single occupation kernel. This differs from the implementation of Koopman based DMD, where most often each snapshot corresponds to a single trajectory. The continuous time DMD method was performed using the Gaussian kernel function with $\mu=50$.

Figure \ref{fig:eeg} presents the obtained eigenvalues, and Figure \ref{fig:eeg1} gives log scaled spectrum obtained from the eigenvectors. It can be seen that the spectrum has strong peaks near the $12$ Hz range, which suggests that the continuous time DMD procedure using occupation kernels can extract frequency information without using shifted copies of the trajectories as in \cite{kutz2016dynamic}.

For this example, the resultant dimensionality of Koopman based DMD makes the analysis of this data set intractable without discarding a significant number of samples.

\begin{figure}
    \centering
    \includegraphics[scale=0.3]{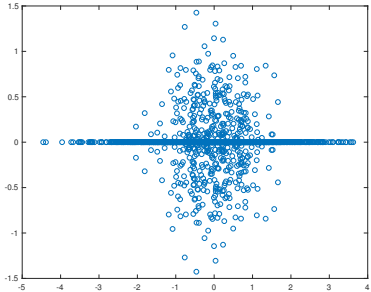}
    \caption{Eigenvalues corresponding to the ssVEP dataset from \cite{gruss2019sympathetic}.}
    \label{fig:eeg}
\end{figure}

\begin{figure}
    \centering
    \includegraphics[scale=0.3]{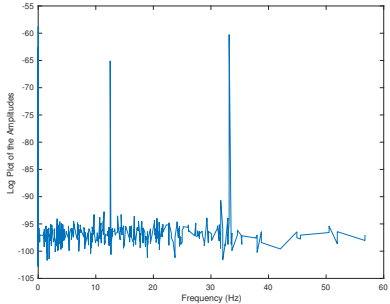}
    \caption{Rescaled spectrum obtained from the ssVEP dataset. This doesn't quite correspond to the spectrum that would be computed through the Fourier transform. However, note the significant peak around $12$ Hz, which corresponds to the ssVEP.}
    \label{fig:eeg1}
\end{figure}

\section{Discussion}
\label{sec:discussion}

\subsection{Unboundedness of Liouville and Koopman Operators}
Traditional DMD approaches aim to estimate a continuous nonlinear dynamical system by first selecting a fixed time-step and then investigate the induced discretized dynamics through the Koopman operator. The algorithm developed in this manuscript estimates the continuous nonlinear dynamics directly by employing occupation kernels, which represent trajectories via an integration functional that interfaces with the Liouville operator. That is, the principle advantage realized through DMD using Liouville operators and occupation kernels over that of kernel-based DMD and the Koopman operator is that the resulting finite-rank representation corresponds to a continuous-time system rather than a discrete time proxy. This is significant, since not all continuous time systems can be discretized for use with the Koopman operator framework. Moreover, the employment of scaled Liouville operators, many dynamics yield a compact operator over the Bargman-Fock space, which allows for norm convergence of DMD procedures.

Liouville operators are unbounded in most cases, which owes strongly to the inclusion of the derivative or gradient in its definition. However, Koopman operators are also unbounded operators in all but a few cases. In the specific instance where the selected kernel function is the exponential dot product kernel, the Koopman operators are only bounded for linear dynamics. Thus, connections between DMD and the Koopman operator necessarily invoke the theory of unbounded operators. However, large classes of both Liouville and Koopman operators can be realized where they are densely defined and closed operators over RKHSs.

\subsection{Approximating the Full State Observable}

The establishment the decomposition of the full state observable relies very strongly on the selection of RKHS. In the case of the Bargmann-Fock space, $x \mapsto (x)_i$ is a function in the space for each $i=1,\ldots,n$. However, this is not the case for the native space of the Gaussian RBF, which does not contain any polynomials in its native space. In both cases, these spaces are universal, which means that any continuous function may be arbitrarily well estimated by a function in the space with respect to the supremum norm over a compact subset. Thus, it is not expected that a good approximation of the full state observable will hold over all of $\mathbb{R}^n$, but a sufficiently small estimation error is possible over a compact workspace. 

\subsection{Scaled Liouville Operators}
One advantage of the Liouville approach to DMD is that the Liouville operators may be readily modified to generate a compact operator through the so-called scaled Liouville operator. A large class of dynamics correspond to a compact operator in this scale Liouville operator case, while Koopman operators cannot be modified in the same fashion. Allowing this compact modification, indicates that on an operator theoretic level, the study of nonlinear dynamical systems through Liouville operators allows for more flexibility in a certain sense.

The experiments presented in Section \ref{sec:experiments} demonstrate that the Liouville modes obtained with the continuous time DMD procedure using Liouville operators and occupation kernels are similar in form to the Koopman modes obtained using kernel-based extended DMD \cite{williams2015kernel}. Moreover, occupation kernels allow for trajectories to be utilized as a fundamental unit of data, which can reduce the dimensionality of the learning problem while retaining some fidelity that would be otherwise lost through discarding data.

\section{Conclusions}
\label{sec:conclusions}
By targeting the DMD decomposition on Liouville operators, which includes Koopman generators as a proper subset, a decomposition of a continuous time dynamical system can be performed directly rather than that of a discrete time proxy for the dynamical system with the Koopman operator. Moreover, by obviating the limiting process using the Koopman operators in the definition of Liouville operators, a broader class of dynamics is accessible through this method, since the requirement of forward completeness may be relaxed. The notion of occupation kernels were leveraged to enable a DMD analysis of the Liouville operator, and a scaled Liouville operator was introduced to provide a collection of compact operators which allows for Norm convergence of the DMD procedure. Two examples were presented, one from fluid dynamics and another EEG dataset, which allowed for the comparison of Liouville and scaled Liouville modes. The method presented here provides a new approach to DMD, which impacts the fundamental operator theory underlying traditional DMD with the Koopman operator.
\appendix

\section{Proofs of Theorem \ref{thm:scaled-compact} and Proposition  \ref{prop:estimate-exponential}}
\begin{proof}[Proof of Theorem \ref{thm:scaled-compact}]
The proof for the case $n = 1$ is presented to simplify the exposition. The case for $n > 1$ follows with some additional bookkeeping of the multi-index.

If $A_{x,a}$ is compact for all $|a| < 1$, then $A_{x^m,a} = A^{m}_{x,\sqrt[m]{a}}$ is compact since products of compact operators are compact. If $f(x) = \sum_{m=0}^\infty f_m x^m$ is such that $\sum_{m=0}^\infty |f_m| \| A_{x^m,a}\| < \infty$, then $A_{f,a} = \lim_{m\to\infty} \sum_{m=0}^M f_m A_{x^m,a},$ with respect the operator norm via the triangle inequality, and $A_{f,a}$ is compact since it is the limit of compact operators. Thus, it is sufficient to demonstrate that $A_{x,a}$ is compact to prove the theorem.

Let $g \in F^{2}(\mathbb{R})$, then $g(x) = \sum_{m=0}^\infty g_m \frac{x^m}{\sqrt{m!}}$ with norm $\| g \|_{F^{2}(\mathbb{R})}^2 = \sum_{m=0}^\infty |g_m|^2 < \infty.$ Applying the scaled Liouville operator, $A_{x,a}$, yields \[A_{x,a} g(x) = axg'(ax) = \sum_{m=0}^\infty g_{m} a^m m \frac{x^m}{\sqrt{m!}}.\] Hence, $\| A_{x,a} g \|_{F^(\mathbb{R})}^2 = |a|^{2m} m^2 |g_{m}|^2 < \infty$ as for large enough $m$, $|a|^{2m} m^2 < 1$. Hence, $A_{x,a}$ is everywhere defined and by the closed graph theorem $A_{x,a}$ is bounded.

As $|a|^{m} m^2 \to 0$, there is an $M$ such that for all $m > M$, $|a|^{m} m^2 < 1$.  Let $P_M$ be the projection onto $\vspan\{ 1, x, x^2, \ldots, x^M\}$. Now consider
\begin{align*}
\|(A_{x,a} - A_{x,a} P_M) g\|^2 &= \sum_{m=M+1}^\infty |g_{m}|^2 |a|^{2m} m^2\\
&\le \sum_{m=M+1}^\infty |g_{m}|^2 |a|^{m} \\
&\le |a|^M \sum_{m=M+1}^\infty |g_{m}|^2 |a|^{m-M} \\
&\le |a|^M \sum_{m=M+1}^\infty |g_{m}|^2 \le |a|^M \| g \|_{F^2(\mathbb{R})}^2.
\end{align*}

Hence, the operator norm of $(A_{x,a} - A_{x,a} P_M)$ is bounded by $|a|^{M/2}$, and as $|a| < 1$,  $A_{x,a} P_m \to A_{x,a}$ in the operator norm. $P_m$ is finite rank and therefore compact. It follows that $A_{x,a} P_m$ is compact, since compact operators form an ideal in the ring of bounded operators. Thus, $A_{x,a}$ is compact as it is the limit of compact operators.\qed
\end{proof}

\begin{proof}[Proof of Proposition \ref{prop:estimate-exponential}]
Suppose that $x(t)$ remains in a compact set $D \subset \mathbb{R}^n$. Since $\phi_{m,a} \in H$ and $H$ consists of twice continuously differentiable functions, there exists $M_1,M_2,F > 0$ such that \begin{gather*}\sup_{x \in D} \| f(x) \| < F \quad \sup_{x \in D}, \| \nabla \phi_{m,a}(x) \| < M_{1,a}, \text{ and} \quad \sup_{x \in D} \| \nabla^2\phi_{m,a}(x) \| < M_{2,a}.
\end{gather*}
First, it is necessary to demonstrate that $M_{1,a}$ and $M_{2,a}$ may be bounded independent of $a$. For each $i,j=1,\ldots,n$ and $y \in \mathbb{R}^n$, the functionals $g \mapsto \frac{\partial}{\partial x_i} g(y)$ and $g \mapsto \frac{\partial^2}{\partial x_i \partial x_j} g(y)$ are bounded (cf. \cite{steinwart2008support}). Setting, $k_{y} = K(\cdot,y)$, it can be seen that the functions $\frac{\partial}{\partial x_i} k_y$ and $\frac{\partial^2}{\partial x_i \partial x_j} k_y$ are the unique functions that represent these functionals through the inner product of the RKHS (cf. \cite{steinwart2008support}). As $\phi_{m,a}$ is a normal vector, $\| \phi_{m,a} \|_H = 1$, and by Cauchy-Schwarz 
\begin{align}\label{eq:nabla-bound}
\| \nabla \phi_{m,a}(y) \|_2 &= \sqrt{ \sum_{i=1}^n \left(\frac{\partial}{\partial x_i} \phi_{m,a}(y) \right)^2}\nonumber\\
& = \sqrt{\sum_{i=1}^n \left(\left\langle \phi_{m,a}, \frac{\partial}{\partial x_i} k_y \right\rangle_H \right)^2}\nonumber\\
& \le \sqrt{\sum_{i=1}^n \left\| \phi_{m,a}\right\|_H^2 \left\|\frac{\partial}{\partial x_i} k_y \right\|_H^2}\nonumber\\
&= \sqrt{ \sum_{i=1}^n \left\|\frac{\partial}{\partial x_i} k_y \right\|_H^2 }.
\end{align}
\eqref{eq:nabla-bound} is bounded over $D$ as $x \mapsto \frac{\partial}{\partial x_i} k_y(x)$ is continuous. Thus, $M_{1,a}$ is bounded independent of $a$. A similar argument may be carried out for $M_{2,a}$. Let $M_1$ and $M_2$ be the respective bounding constants.

Note that
\begin{equation*}
    \frac{\partial}{\partial t} \phi_{m,a}(ax(t))
    = a \nabla \phi_{m,a}(a x(t)) f(x(t))
    = A_{f,a} \phi_{m,a}(x(t))
    = \mu_{m,a} \phi_{m,a}(x(t)).
\end{equation*}
Then by the mean value inequality, Cauchy-Schwarz, and the bounds given above,
\begin{gather*}
    \left|\frac{\partial}{\partial t} \phi_{m,a}(ax(t)) - \frac{\partial}{\partial t} \phi_{m,a}(x(t))\right|\\
    = \left|a \nabla \phi_{m,a}(a x(t)) f(x(t)) - \nabla \phi_{m,a}(x(t)) f(x(t))\right|\\
    \le F\left\|a \nabla \phi_{m,a}(a x(t)) - a\nabla\phi_{m,a}(x(t))  + a\nabla\phi_{m,a}(x(t)) - \nabla \phi_{m,a}(x(t)) \right\|_2\\
    \le |a|F\left\| \nabla \phi_{m,a}(a x(t)) - \nabla\phi_{m,a}(x(t)) \right\|_2 + F|a-1| M_1 \| x(t)\|_2\\
    \le |a||a-1| M_2 F \|x(t)\|_2 + |a-1| M_1 F \|x(t)\|_2 = O(|a-1|).
\end{gather*}

Setting $\epsilon_a(t) := \frac{\partial}{\partial t} \phi_{m,a}(ax(t)) - \frac{\partial}{\partial t} \phi_{m,a}( x(t))$, it follows that $\sup_{0 \le t \le T} \| \epsilon_a(t) \|_2 = O(|a-1|)$.
Hence,
\begin{align*}
    \mu_{m,a} \phi_{m,a}(x(t)) &= \frac{\partial}{\partial t} \phi_{m,a}(a x(t))\\
    &= \frac{\partial}{\partial t} \phi_{m,a}(x(t)) + \epsilon(t),
\end{align*}
and \[ \phi_{m,a}(x(t)) = e^{\mu_{m,a}t} \phi_{m,a}(x(0)) - e^{\mu_{m,a}t} \int_0^t e^{-\mu_{m,a} \tau} \epsilon(\tau) d\tau. \]
As the time interval is fixed to $[0,T]$, $e^{\mu_{m,a}t} \int_0^t e^{-\mu_{m,a} \tau} \epsilon(\tau) d\tau = O(|a-1|),$ since $\mu_{m,a}$ is bounded with respect to $a$.\qed
\end{proof}

\begin{proof}[Proof of Theorem \ref{thm:DMD-norm-convergence}]
The following proof is more general than what is indicated in the theorem statement of Theorem \ref{thm:DMD-norm-convergence}. In fact, for any compact operator, $T$, and any set $\{ g_i \}_{i=1}^\infty$ such that $\overline{\vspan(\{g_i\}_{i=1}^\infty)} = H$, the sequence of operators $P_{\alpha_M} T P_{\alpha_M} \to T$ in norm, where $P_{\alpha_M}$ is the projection onto $\vspan(\{g_i\}_{i=1}^M)$. Henceforth, it will be assumed that $\{ g_i \}_{i=1}^\infty$ is an orthonormal basis for $H$, since given any complete basis in $H$, an orthonormal basis may be obtained via the Gram-Schmidt process.

First note that every compact operator has a representation as $T = \sum_{i=1}^\infty \lambda_i \langle \cdot, v_i \rangle_H u_i$, where $\{ v_i \}$ and $\{ u_i \}$ are orthonormal collections of vectors (functions) in $H$, and $\{ \lambda_i \}_{i=1}^\infty \subset \mathbb{C}$ are the singular values of $T$. If $T_M := \sum_{i=1}^M \lambda_i \langle \cdot, v_i \rangle_H u_i$ then $T_M \to T$ as $M \to \infty$ in the operator norm.

Suppose that $\epsilon > 0$. Select $M$ such that $\| T - T_M\| < \epsilon$, and select $N$ such that for all $n > N$,
\[
\|u_i - P_n u_i\|_H < \frac{\epsilon}{\sum_{i=1}^M |\lambda_i|^2} \text{ and } \|v_i - P_n v_i\|_H < \frac{\epsilon}{\sqrt{M} \left(\sum_{i=1}^M |\lambda_i|^2\right)^{1/2}}\]
for all $i=1,\ldots,M$. Let $g \in H$ be arbitrary.

Now consider,
\begin{gather*}
    \| Tg - P_n T P_n g\|_H = \| Tg - T_M g + T_M g - P_n T P_n f\|_H\\
    \le \| T - T_M\| \|g\|_H + \|T_M g - P_n T P_n g\|_H
    \le \epsilon \| g\|_H + \|T_M g - P_n T P_n g\|_H.
\end{gather*}
The second term after the inequality may be expanded as
\begin{gather*}
    \|T_M g - P_n T P_n g\|_H \le \|T_M g - P_n T_M g\|_H + \|P_n T_M g - P_n T P_n g\|_H\\
    \|T_M g - P_n T_M g\|_H + \|T_M g - T P_n g\|_H\\
    \le \| T_M g - P_n T_M g\|_H + \|T_M g - T_M P_n g\|_H + \|T_M P_n g - T P_n g\|_H\\
    \le \| T_M g - P_n T_M g\|_H + \|T_M g - T_M P_n g\|_H + \epsilon \| g\|_H.
\end{gather*}
Now the objective is to demonstrate that both $\| T_M g - P_n T_M g\|_H$ and $\|T_M g - T_M P_n g\|_H$ are proportional to $\epsilon \| g\|_H$. Note that
\begin{gather*}
    \| T_M g - P_n T_M g\|_H = \left\| \sum_{i=1}^M \lambda_i \langle g, v_i \rangle_H (u_i - P_n u_i)\right\|_H\\
    \le \sum_{i=1}^M |\lambda_i| |\langle g, v_i \rangle_H| \| u_i - P_n u_i \|_H\\
    \le \sqrt{\sum_{i=1}^M |\langle g, v_i\rangle_H|^2} \left( \sum_{i=1}^M |\lambda_i|^2 \| u_i - P_n u_i \|_H^2 \right)^{1/2}
    \\
    \le \|g\|_H \left( \sum_{i=1}^M |\lambda_i|^2 \| u_i - P_n u_i \|_H^2 \right)^{1/2} \le \epsilon \|g\|_H,
\end{gather*}
and
\begin{gather*}
    \|T_M(g - P_n g) \|_H = \left\| \sum_{i=1}^M \lambda_i \langle g - P_n g, v_i\rangle_H u_i \right\|\\
    \le \left\| \sum_{i=1}^M \lambda_i \left\langle \sum_{j=n+1}^\infty \langle g, g_j \rangle_H g_j, v_i \right\rangle_H u_i\right\|\\
    \le \sum_{j=n+1}^\infty |\langle g, g_j\rangle_H| \left(\sum_{i=1}^M |\lambda_i| |\langle g_j, v_i\rangle_H|\right)\\
    \le \left( \sum_{j=n+1}^\infty |\langle g, g_i\rangle_H|^2\right)^{1/2} \left( \sum_{j=n+1}^\infty \left( \sum_{i=1}^M |\lambda_i| |\langle g_i, v_i \rangle_H|\right)^{2}\right)^{1/2}\\
    \le \| g\|_H \left( \sum_{j=n+1}^\infty \left( \sum_{i=1}^M |\lambda_i|^2 \right) \left(\sum_{i=1}^M |\langle g_j, v_i\rangle_H|^2\right)\right)^{1/2}\\
    \le \| g\|_H \left(\sum_{i=1}^M |\lambda_i|^2\right)^{1/2} \left( \sum_{i=1}^M \sum_{j=n+1}^\infty |\langle g, v_i\rangle_H|^2 \right)^{1/2}\\
    = \| g\|_H \left(\sum_{i=1}^M |\lambda_i|^2\right)^{1/2} \left( \sum_{i=1}^M \|v_i - P_n v_i\|^2 \right)^{1/2}
    \le  \epsilon \|g\|_H.
\end{gather*}

Thus, for every $\epsilon > 0$, there is an $N$ such that for all $n > N$, $\|T g - P_n T P_n g\|_H \le 4\epsilon \| g\|_H.$ Hence, it follows that $\| T - P_n T P_n \| \le 4\epsilon$. Thus, as $n \to \infty$, $P_n T P_n \to T$ in the operator norm.\qed
\end{proof}

\section*{Acknowledgments}
This research was supported by the Air Force Office of Scientific Research (AFOSR) under contract numbers FA9550-20-1-0127, FA9550-15-1-0258, FA9550-16-1-0246, and FA9550-18-1-0122, the Air Force Research Laboratory (AFRL) under contract number FA8651-19-2-0009, and the Office of Naval Research (ONR) under contract N00014-18-1-2184. Any opinions, findings and conclusions or recommendations expressed in this material are those of the author(s) and do not necessarily reflect the views of the sponsoring agencies.
\bibliographystyle{spmpsci}
\bibliography{references}

\begin{thebibliography}{10}
\providecommand{\url}[1]{{#1}}
\providecommand{\urlprefix}{URL }
\expandafter\ifx\csname urlstyle\endcsname\relax
  \providecommand{\doi}[1]{DOI~\discretionary{}{}{}#1}\else
  \providecommand{\doi}{DOI~\discretionary{}{}{}\begingroup
  \urlstyle{rm}\Url}\fi

\bibitem{aronszajn1950theory}
Aronszajn, N.: Theory of reproducing kernels.
\newblock Transactions of the American mathematical society \textbf{68}(3),
  337--404 (1950)

\bibitem{bakardjian2010optimization}
Bakardjian, H., Tanaka, T., Cichocki, A.: Optimization of {SSVEP} brain
  responses with application to eight-command brain--computer interface.
\newblock Neuroscience letters \textbf{469}(1), 34--38 (2010)

\bibitem{bin2009online}
Bin, G., Gao, X., Yan, Z., Hong, B., Gao, S.: An online multi-channel
  {SSVEP}-based brain--computer interface using a canonical correlation
  analysis method.
\newblock Journal of neural engineering \textbf{6}(4), 046002 (2009)

\bibitem{bittracher2015pseudogenerators}
Bittracher, A., Koltai, P., Junge, O.: Pseudogenerators of spatial transfer
  operators.
\newblock SIAM Journal on Applied Dynamical Systems \textbf{14}(3), 1478--1517
  (2015)

\bibitem{budivsic2012applied}
Budi{\v{s}}i{\'c}, M., Mohr, R., Mezi{\'c}, I.: Applied koopmanism.
\newblock Chaos: An Interdisciplinary Journal of Nonlinear Science
  \textbf{22}(4), 047510 (2012)

\bibitem{cichella2015cooperative}
{Cichella}, V., {Kaminer}, I., {Dobrokhodov}, V., {Xargay}, E., {Choe}, R.,
  {Hovakimyan}, N., {Aguiar}, A.P., {Pascoal}, A.M.: Cooperative path following
  of multiple multirotors over time-varying networks.
\newblock IEEE Transactions on Automation Science and Engineering
  \textbf{12}(3), 945--957 (2015)

\bibitem{coddington1955theory}
Coddington, E.A., Levinson, N.: Theory of ordinary differential equations.
\newblock Tata McGraw-Hill Education (1955)

\bibitem{cowen1995composition}
Cowen~Jr, C.C., MacCluer, B.I.: Composition Operators on Spaces of Analytic
  Functions, vol.~20.
\newblock CRC Press (1995)

\bibitem{mezic2019}
{\v{C}}rnjari{\'{c}}-{\v{Z}}ic, N., Ma{\'{c}}e{\v{s}}i{\'{c}}, S., Mezi{\'{c}},
  I.: {K}oopman operator spectrum for random dynamical systems.
\newblock Journal of Nonlinear Science  (2019).
\newblock \doi{10.1007/s00332-019-09582-z}

\bibitem{cvitanovic2005chaos}
Cvitanovic, P., Artuso, R., Mainieri, R., Tanner, G., Vattay, G., Whelan, N.,
  Wirzba, A.: Chaos: classical and quantum.
\newblock ChaosBook. org (Niels Bohr Institute, Copenhagen 2005) \textbf{69}
  (2005)

\bibitem{das2020koopman}
Das, S., Giannakis, D.: {K}oopman spectra in reproducing kernel {H}ilbert
  spaces.
\newblock Applied and Computational Harmonic Analysis  (2020)

\bibitem{folland2013real}
Folland, G.B.: Real analysis: modern techniques and their applications.
\newblock John Wiley \& Sons (2013)

\bibitem{froyland2014detecting}
Froyland, G., Gonz{\'a}lez-Tokman, C., Quas, A.: Detecting isolated spectrum of
  transfer and {K}oopman operators with fourier analytic tools.
\newblock Journal of Computational Dynamics \textbf{1}(2), 249--278 (2014)

\bibitem{giannakis2019data}
Giannakis, D.: Data-driven spectral decomposition and forecasting of ergodic
  dynamical systems.
\newblock Applied and Computational Harmonic Analysis \textbf{47}(2), 338--396
  (2019)

\bibitem{giannakis2020extraction}
Giannakis, D., Das, S.: Extraction and prediction of coherent patterns in
  incompressible flows through space--time {K}oopman analysis.
\newblock Physica D: Nonlinear Phenomena \textbf{402}, 132211 (2020)

\bibitem{giannakis2018koopman}
Giannakis, D., Kolchinskaya, A., Krasnov, D., Schumacher, J.: {K}oopman
  analysis of the long-term evolution in a turbulent convection cell.
\newblock ar{X}iv:1804.01944 (2018)

\bibitem{gruss2019sympathetic}
Gruss, L.F., Keil, A.: Sympathetic responding to unconditioned stimuli predicts
  subsequent threat expectancy, orienting, and visuocortical bias in human
  aversive pavlovian conditioning.
\newblock Biological psychology \textbf{140}, 64--74 (2019)

\bibitem{haddad2019dynamical}
Haddad, W.: A Dynamical Systems Theory of Thermodynamics.
\newblock Princeton Series in Applied Mathematics. Princeton University Press
  (2019).
\newblock \urlprefix\url{https://books.google.com/books?id=gsJ9DwAAQBAJ}

\bibitem{hallam2012mathematical}
Hallam, T.G., Levin, S.A.: Mathematical ecology: an introduction, vol.~17.
\newblock Springer Science \& Business Media (2012)

\bibitem{hastie2005elements}
Hastie, T., Tibshirani, R., Friedman, J., Franklin, J.: The elements of
  statistical learning: data mining, inference and prediction.
\newblock The Mathematical Intelligencer \textbf{27}(2), 83--85 (2005)

\bibitem{jury2007c}
Jury, M.T.: C*-algebras generated by groups of composition operators.
\newblock Indiana University Mathematics Journal pp. 3171--3192 (2007)

\bibitem{khalil2002nonlinear}
Khalil, H.K., Grizzle, J.W.: Nonlinear systems, vol.~3.
\newblock Prentice hall Upper Saddle River, NJ (2002)

\bibitem{klus2020data}
Klus, S., N{\"u}ske, F., Peitz, S., Niemann, J.H., Clementi, C., Sch{\"u}tte,
  C.: Data-driven approximation of the {K}oopman generator: Model reduction,
  system identification, and control.
\newblock Physica D: Nonlinear Phenomena \textbf{406}, 132416 (2020)

\bibitem{korda2018convergence}
Korda, M., Mezi{\'c}, I.: On convergence of extended dynamic mode decomposition
  to the {K}oopman operator.
\newblock Journal of Nonlinear Science \textbf{28}(2), 687--710 (2018)

\bibitem{kutz2016dynamic}
Kutz, J.N., Brunton, S.L., Brunton, B.W., Proctor, J.L.: Dynamic mode
  decomposition: data-driven modeling of complex systems.
\newblock SIAM (2016)

\bibitem{luery2013composition}
Luery, K.E.: Composition operators on Hardy spaces of the disk and half-plane.
\newblock University of Florida (2013)

\bibitem{mezic2005spectral}
Mezi{\'c}, I.: Spectral properties of dynamical systems, model reduction and
  decompositions.
\newblock Nonlinear Dynamics \textbf{41}(1-3), 309--325 (2005)

\bibitem{mezic2013analysis}
Mezi{\'c}, I.: Analysis of fluid flows via spectral properties of the {K}oopman
  operator.
\newblock Annual Review of Fluid Mechanics \textbf{45}, 357--378 (2013)

\bibitem{middendorf2000brain}
Middendorf, M., McMillan, G., Calhoun, G., Jones, K.S.: Brain-computer
  interfaces based on the steady-state visual-evoked response.
\newblock IEEE transactions on rehabilitation engineering \textbf{8}(2),
  211--214 (2000)

\bibitem{pedersen2012analysis}
Pedersen, G.K.: Analysis now, vol. 118.
\newblock Springer Science \& Business Media (2012)

\bibitem{petro2017multimodal}
Petro, N.M., Gruss, L.F., Yin, S., Huang, H., Miskovic, V., Ding, M., Keil, A.:
  Multimodal imaging evidence for a frontoparietal modulation of visual cortex
  during the selective processing of conditioned threat.
\newblock Journal of cognitive neuroscience \textbf{29}(6), 953--967 (2017)

\bibitem{regan1989human}
Regan, D.: Human brain electrophysiology.
\newblock Evoked potentials and evoked magnetic fields in science and medicine
  (1989)

\bibitem{rosenfeld2015densely}
Rosenfeld, J.A.: Densely defined multiplication on several sobolev spaces of a
  single variable.
\newblock Complex Analysis and Operator Theory \textbf{9}(6), 1303--1309 (2015)

\bibitem{rosenfeld2015introducing}
Rosenfeld, J.A.: Introducing the polylogarithmic hardy space.
\newblock Integral Equations and Operator Theory \textbf{83}(4), 589--600
  (2015)

\bibitem{rosenfeld2016sarason}
Rosenfeld, J.A.: The sarason sub-symbol and the recovery of the symbol of
  densely defined toeplitz operators over the hardy space.
\newblock Journal of Mathematical Analysis and Applications \textbf{440}(2),
  911--921 (2016)

\bibitem{rosenfeld2021dynamic}
Rosenfeld, J.A., Kamalapurkar, R.: Dynamic mode decomposition with control
  {L}iouville operators.
\newblock In: 24th International Symposium on Mathematical Theory of Networks
  and Systems (MTNS2021) (2021).
\newblock Accepted.

\bibitem{rosenfeld2019occupation}
Rosenfeld, J.A., Kamalapurkar, R., Russo, B., Johnson, T.T.: Occupation kernels
  and densely defined {L}iouville operators for system identification.
\newblock In: IEEE Conference on Decision and Control, pp. 6455--6460. IEEE
  (2019)

\bibitem{rosenfeld2019occupation2}
Rosenfeld, J.A., Russo, B., Kamalapurkar, R., Johnson, T.T.: The occupation
  kernel method for nonlinear system identification.
\newblock ar{X}iv:1909.11792 (2019)

\bibitem{steinwart2008support}
Steinwart, I., Christmann, A.: Support vector machines.
\newblock Springer Science \& Business Media (2008)

\bibitem{szafraniec2000reproducing}
Szafraniec, F.H.: The reproducing kernel {H}ilbert space and its multiplication
  operators.
\newblock In: Complex Analysis and Related Topics, pp. 253--263. Springer
  (2000)

\bibitem{toth2018reaction}
T{\'o}th, J., Nagy, A.L., Papp, D.: Reaction kinetics: exercises, programs and
  theorems.
\newblock Springer (2018)

\bibitem{walters2018online}
{Walters}, P., {Kamalapurkar}, R., {Voight}, F., {Schwartz}, E.M., {Dixon},
  W.E.: Online approximate optimal station keeping of a marine craft in the
  presence of an irrotational current.
\newblock IEEE Transactions on Robotics \textbf{34}(2), 486--496 (2018)

\bibitem{wendland2004scattered}
Wendland, H.: Scattered data approximation, vol.~17.
\newblock Cambridge university press (2004)

\bibitem{williams2015data}
Williams, M.O., Kevrekidis, I.G., Rowley, C.W.: A data--driven approximation of
  the koopman operator: Extending dynamic mode decomposition.
\newblock Journal of Nonlinear Science \textbf{25}(6), 1307--1346 (2015)

\bibitem{williams2015kernel}
Williams, M.O., Rowley, C.W., Kevrekidis, I.G.: A kernel-based method for
  data-driven koopman spectral analysis.
\newblock Journal of Computational Dynamics \textbf{2}, 247 (2015).
\newblock \doi{10.3934/jcd.2015005}

\end{thebibliography}
\end{document}